%% file: main.tex
\documentclass[12pt]{article}

\usepackage[utf8]{inputenc}
\usepackage[T1]{fontenc}
\usepackage{lmodern}
\usepackage{float}

\usepackage[margin=2cm]{geometry}
\usepackage{ragged2e}
\usepackage[pagebackref, hidelinks]{hyperref}
\hypersetup{colorlinks=true, linkcolor = magenta}

\usepackage{mathtools}
\usepackage{amsthm, amscd, amssymb}
\usepackage{soul}
\usepackage{mathrsfs} % for \mathscr{L}

\usepackage{booktabs}
\usepackage[all]{xy}
\usepackage{caption}

\usepackage{comment}
\usepackage{xcolor}
\usepackage{enumitem}

\usepackage{cleveref}

\usepackage{array}
\usepackage{diagbox}
\usepackage{graphicx}
\usepackage[outline]{contour}

\usepackage{tikz}
\usepackage{pgf, pgfplots}
\pgfplotsset{compat=newest}

\usetikzlibrary{patterns}
\usetikzlibrary{arrows, arrows.meta}
\usetikzlibrary{calc} % pour arcs de cercles
\usetikzlibrary{intersections,through,backgrounds} %pour les loops
\usetikzlibrary{shapes.geometric} %pour triangle dans \node

\definecolor{grey}{rgb}{0.5,0.5,0.5}
\definecolor{forestgreen}{rgb}{0.133,0.545,0.133}
\definecolor{marron}{rgb}{0.6,0.2,0.0}
\definecolor{violet}{rgb}{0.5,0.,0.5}
\definecolor{lightpurple}{rgb}{0.6,0.2,0.6}

\usepackage{subfiles}
\theoremstyle{plain}
    \newtheorem{theorem}{Theorem}[section]
    
    \newtheorem{lemma}[theorem]{Lemma}
    \newtheorem{proposition}[theorem]{Proposition}
\theoremstyle{definition}
    \newtheorem{definition}[theorem]{Definition}
    \newtheorem{question}[theorem]{Question}
    \newtheorem{conjecture}[theorem]{Conjecture}
    \newtheorem{context}[theorem]{Context}
\theoremstyle{remark}
    \newtheorem{remark}[theorem]{Remark}
    \newtheorem{example}[theorem]{Example}

\newcommand{\demph}[1]{\underline{\emph{{#1}}}}

\newcommand{\Ecf}[1]{\ensuremath{\lfloor\,{#1}\,\rfloor}}

\newcommand{\N}{\mathbb{N}}
\newcommand{\Z}{\mathbb{Z}}
\newcommand{\Q}{\mathbb{Q}}
\newcommand{\R}{\mathbb{R}}
\newcommand{\C}{\mathbb{C}}
\newcommand{\K}{\mathbb{K}}

\newcommand{\Proj}{\mathbf{P}}
\newcommand{\Ball}{\mathbf{B}}

\newcommand{\HP}{\mathbf{HP}}
\newcommand{\M}{\mathbf{M}}

\newcommand{\Al}{\mathcal{A}}
\newcommand{\La}{\mathcal{L}}
\newcommand{\TT}{\mathcal{T}}
\newcommand{\Hex}{\mathcal{H}}
\newcommand{\Geo}{\mathcal{G}}
\newcommand{\GeoS}{\mathcal{S}}

\newcommand{\Id}{\mathbf{1}}

\newcommand{\VV}{\mathtt{V}}
\newcommand{\EE}{\mathtt{E}}
\newcommand{\XX}{\mathtt{X}}
\newcommand{\Ee}{\mathtt{e}}
\newcommand{\Xx}{\mathtt{x}}

\DeclareMathOperator{\Card}{Card}

\DeclareMathOperator{\Out}{Out}
\DeclareMathOperator{\Stab}{Stab}

\DeclareMathOperator{\Homeo}{Homeo}

\DeclareMathOperator{\Map}{Map}
\DeclareMathOperator{\Mod}{Mod}

\DeclareMathOperator{\GL}{GL}
\DeclareMathOperator{\SL}{SL}
\DeclareMathOperator{\PSL}{PSL}
\DeclareMathOperator{\PGL}{PGL}

\DeclareMathOperator{\AD}{AD}

\DeclareMathOperator{\Rad}{Rad}
\DeclareMathOperator{\len}{len}

\DeclareMathOperator{\fac}{fac}
\DeclareMathOperator{\Shift}{Shift}

\DeclareMathOperator{\Dio}{Dio}

\DeclareMathOperator{\LC}{L}
\DeclareMathOperator{\MC}{M}

\title{Transcendence of Simple Geodesics on Finite Modular Covers}
\author{Christopher-Lloyd Simon}
\date{\today}

\begin{document}

\maketitle

\begin{abstract}
    The real projective line $\mathbb{R}\mathbf{P}^1$ is the boundary of $\mathbf{HP}=\{z\in \mathbb{C}\colon \Im(z)>0\}$, a model of the hyperbolic plane whose space of geodesics identifies with $\mathcal{G}(\mathbf{HP})=\mathbb{R}\mathbf{P}^1 \times \mathbb{R}\mathbf{P}^1 \setminus \mathrm{diagonal}$.
    The modular group $\Gamma=\operatorname{PSL}_2(\mathbb{Z})$ acts on $\mathbf{HP}$ with quotient the modular orbifold $\mathbf{M}=\Gamma\backslash \mathbf{HP}$.
    Consider a finite-index subgroup of the modular group $\Gamma^\prime \subset \Gamma = \operatorname{PSL}_2(\mathbb{Z})$ corresponding to a finite cover $\mathbf{M} \to \mathbf{M}^\prime$.
    A geodesic $(\xi^-,\xi^+)\in \mathcal{G}(\mathbf{HP})$ projects $\bmod{\Gamma^\prime}$ to a geodesic $\xi^\prime \subset \mathbf{M}^\prime$.
    We show that if $\xi^\prime$ is simple, then $\xi^+$ is either rational or quadratic or transcendental.
    In the transcendental case, we obtain bounds on the Mahler measures and show that those can be improved for geodesics fixed by pseud-Anosov maps.
    Finally, we also explain in detail why all this was known for the modular torus cover associated to the derived subgroup $\Gamma^\prime = [\Gamma, \Gamma]$.
\end{abstract}

\renewcommand{\contentsname}{Plan of the paper}
\setcounter{tocdepth}{2}
\tableofcontents

\newpage

\subfile{sec0-Intro}

\newpage

\subfile{sec1-ContFrac}

\newpage

\subfile{sec2-TranSimple}

\newpage

\subfile{sec3-ModTorus}

\newpage

\subsection*{Acknowledgements}

There are several people i wish to thank for (ongoing) discussions related to this project (including those mentioned in the conjectures), their help and their encouragement.

This work arose, evolved and matured through many discussions at blackboard and email exchanges with Scott Schmieding, Jean-Pierre Otal and François Maucourant.
While not used in this work, I learned a lot from the answers by Pierre Arnoux, Pascal Hubert and Christopher Leininger to many questions on laminations and interval exchanges, by Boris Adamczewski, Yann Bugeaud on diophantine approximation, and Bastian Espinoza on symbolic dynamics. 

I am also grateful to the organizers of the Simons semester on \href{https://sites.google.com/impan.pl/simons-cfd2026/home-page}{Continued Fractions, Fractals, Ergodic theory and Dynamics} for inviting me to present some of this material in a minicourse, especially Valérie Berthé for her interest and enthusiasm.

This work was completed while supported by the \href{https://perso.univ-rennes1.fr/serge.cantat/ERCGroupsOfAlgebraicTransformations.html}{ERC GOAT 101053021}.

%\nocite{*} %(afficher les ref non citees)
\bibliographystyle{alpha} %apalike
\bibliography{biblio}

\end{document}

%% file: sec0-Intro.tex
\begin{comment}
\begin{abstract}
    The real projective line $\R\Proj^1$ is the boundary of $\HP=\{z\in \C\colon \Im(z)>0\}$, a model of the hyperbolic plane whose space of geodesics identifies with $\Geo(\HP)=\R\Proj^1 \times \R\Proj^1 \setminus \operatorname{diagonal}$.
    The modular group $\Gamma=\PSL_2(\Z)$ acts on $\HP$ with quotient the modular orbifold $\M=\Gamma\backslash \HP$.
    Consider a finite-index subgroup of the modular group $\Gamma^\prime \subset \Gamma = \PSL_2(\Z)$ corresponding to a finite cover $\M\to \M^\prime$.
    %
    A geodesic $(\xi^-,\xi^+)\in \mathcal{G}(\mathbf{HP})$ projects $\bmod{\Gamma^\prime}$ to a geodesic $\xi^\prime \subset \mathbf{M}^\prime$.
    We show that if $\xi^\prime$ is simple, then $\xi^+$ is either rational or quadratic or transcendental.
    %
    % We also explain in detail why this was known for the modular torus cover associated to the derived subgroup $\Gamma^\prime = [\Gamma, \Gamma]$.
\end{abstract}
\end{comment}

\section{Introduction}

The sections of this introduction are roughly organised as the sections of the article, except that we take this opportunity to provide some motivation and position this work at the crossroad of low dimensional topology (of simple geodesics on modular covers), symbolic dynamics (of continued fractions) and Diophantine approximation (of algebraic numbers).
% The main result is Theorem~\ref{Intro_thm:simple-geodesic-trichtomy}.

\subsection{Motivation: approximating reals by rational and quadratic}

% General references in this subsection are~\cite{Schmidt-diophantine-approximation_1970}.

\subsubsection*{Continued fractions as coding geodesics in the modular surface}

Classical diophantine approximation concerns the approximation of numbers in the real projective line $\R\Proj^1{}$ by numbers in the rational projective line $\Q\Proj^1{}$ whose complexity is measured by height.
It began with the Euclidean algorithm: every \(\xi^+ \in \R_{> 1}\) has a unique continued fraction expansion
\begin{equation*}
%\textstyle
\xi^+ =
\Ecf{x_0;x_1,\dots}=x_0+\tfrac{1}{x_1+\dots} 
% = R^{x_0}L^{x_1}\cdots \infty
%\Ecf{x_0,x_1,x_2,\dots}=x_0+\frac{1}{x_1+\frac{1}{x_2+\dots}}
\quad \mathrm{with} 
\quad x_j \in \N_{\ge 1}
% \quad \mathrm{and} 
% \quad \forall j>0,\; n_j>0
\end{equation*}
that is finite if and only if $\xi^+$ is rational in which case it is required to have even length $2k\in 2\N$, % (and it is sometimes convenient to add another last entry $x_{2k}=\infty$);
and that is pre-periodic if and only if $\xi^+$ is quadratic.
% To represent numbers $\xi^- \in [-\infty,0)$ we apply the involution $\xi\mapsto S(\xi^-) = -1/\xi^- \in [0,\infty)$, so that every number in $\R\Proj^1{}$ admits exactly one representation (including $0=\Ecf{}$ and $\infty = -1/\Ecf{}$).

One may rewrite this in terms of the Euclidean monoid $\PSL_2(\N)$ freely generated by the parabolic linear-fractional transformations $R\colon z\mapsto z+1$ and $L\colon z \mapsto z/(1+z) $, which acts on $[0,\infty]$ so that $\xi^+\in (1,\infty)$ is the image of $\infty$ by the finite or infinite product $X=R^{x_0}L^{x_1}\cdots$:
\begin{equation*}
%\textstyle
\PSL_2(\N)=\{R,L\}^\star
\qquad
R = % TS^{-1}
\begin{psmallmatrix}
1 & 1\\
0 & 1
\end{psmallmatrix}
\quad 
L = % T^{-1}S
\begin{psmallmatrix}
1 & 0\\
1 & 1
\end{psmallmatrix}
\qquad
% X = R^{x_0}L^{x_1}\dots
% \quad
\xi^+ = X \cdot \infty = \lim_{k\to +\infty}
(R^{x_0}\cdots L^{x_{2k-1}} \cdot \infty)
\end{equation*}
This yields the correspondence with the symbolic encoding of geodesics on the modular orbifold.

The real projective line $\R\Proj^1 \subset \C\Proj^1{}$ is the boundary of $\HP=\{z\in \C\colon \Im(z)>0\}$, a model of the hyperbolic plane whose space of geodesics identifies with $\Geo(\HP)=\R\Proj^1 \times \R\Proj^1 \setminus \mathrm{diagonal}$. 

The \demph{modular group} $\Gamma=\PSL_2(\Z)$ generated by $\{R,L\}$ acts on $\HP$ by linear-fractional transformations with quotient the \demph{modular orbifold} $\M=\Gamma\backslash \HP$, which has a cusp $\Gamma \backslash \Q\Proj^1$.
% The space of geodesics on $\M^\prime$ identifies with the quotient $\Geo(\M)=\Gamma^\prime \backslash \Geo(\HP)$ by the diagonal $\Gamma^\prime$-action.
The geodesic $\xi=(0,\xi^+)\in \Geo(\HP)$ projects $\bmod{\Gamma}$ to a geodesic $\xi^\prime \in \Geo(\M)$, which escapes to the cusp if and only if $\xi^+$ is rational and accumulates on a periodic geodesic if and only if $\xi^+$ is quadratic.

The real $\xi^+$ has rational approximants $\tfrac{p_n}{q_n}=\Ecf{x_0,\dots, x_{n}}$ for coprime $p_n,q_n\in \N$ satisfying $\lvert \xi^+ -\tfrac{p_n}{q_n} \rvert < \tfrac{1}{q_n^2 x_{n+1}}$ and with denominators $\prod_0^{n} x_i \le q_n \le \prod_0^n (1+x_i)$.
The geodesic $\xi^\prime \subset \M$ has successive cusp-excursions with peaks of height $\log(x_{n+1})$ between times $\log(q_n)$ and $\log(q_{n+1})$.

More generally, we have a dictionary 
% (\ref{tab:arithmetic-dynamics}) 
between the symbolic dynamics of the sequence of convergents $x\in (\N_{\ge 1})^\Z$ under the shift map, the homogeneous dynamics of the geodesic $\xi^\prime \subset \M$ as time flows and how it approaches the cusp or shadows periodic geodesics, and the quality of successive best rational or quadratic approximations to $\xi^+$ measured by their arithmetic complexity.

\begin{table}[H]
    \centering
    \begin{tabular}{||c|c|c||}
        \hline
        \hline
        Continued fractions $x$ &
        Flow of geodesic $\xi^\prime$ in $\M$ &
        Diophantine approximation to $\xi^+$
        \\
        \hline
        \hline
        $x_{2k}=\infty$ &
        escapes to the cusp &
        $\xi^+$ is rational
        \\
        \hline
        $\log(x_{n+1}) > \epsilon \log(q_n)$ &
        high excursion very soon &
        Roth $\implies \xi^+\notin \Bar{\Q}$
        \\
        \hline
        \hline
        $x$ pre-periodic &
        accumulates on a period &
        $\xi^+$ quadratic irrational
        \\
        \hline
        repetitions of length $>\epsilon$ &
        over-shadowing periods &
        Schmidt, Bugeaud $\implies \xi^+\notin \Bar{\Q}$
        \\
        \hline
        \hline
        % Gauss-Kusmin statistics &
        % Haar measure is mixing &
        % Khintchine $0\vee 1$ laws
        % \\
        % $\limsup (\log x_{n+1})/(\log n)=1$ &
        % Sullivan's law of logarithms &
        % $\|q\xi^+ \| <1/(q \log\log q)$
        % \\
        % \hline
    \end{tabular}
    % \caption{Symbolic dynamics $\mid$ homogeneous dynamics, diophantine approximation}
    % \label{tab:arithmetic-dynamics}
\end{table}

\subsubsection*{Lagrange constants to measure approximation by rationals}

For $\xi^+=\Ecf{x}\in \R$, its \demph{Lagrange constant} measures the quality of its best rational approximations: $\LC(\xi^+)^{-1} = \liminf_n \left( q_n \lvert q_n \xi^+ - p_n\rvert \right)$ and $ \LC(\xi^+)= \limsup_{n}({\Ecf{0, x_{n-1},\dots, x_{0}}+\Ecf{x_{n},x_{n+1},\dots}})$.
In particular $\LC(\xi^+)=0 \iff \xi \in \Q$ and more importantly $\LC(\xi^+)<\infty \iff \limsup_n(x_n) < \infty$, so it is only of interest for numbers that are badly approximable by rationals.
Since $\LC(\xi^+)$ only depends on the tail of the continued fraction expansion $x=(x_n)$, it is invariant by $\PSL_2(\N)$.

In $\M$, the horoball neighborhood $\Ball(h)$ of the cusp $\infty$ at height $h\ge 1$ is the projection of $\{\Im(z) > h\} \subset \HP$, having area $1/h$.
The geodesic $(\infty,\xi^+)\in \HP$ projects to a geodesic on $\M$ whose penetrations into $\Ball(h)$ are indexed by $n\in \N$ such that $\tfrac{1}{2}\left(\Ecf{0, x_{n-1},\dots, x_{0}}+\Ecf{x_{n},x_{n+1},\dots}\right)>h$.
Thus $\LC(\xi^+)$ is the area of the largest cusp horoball-neighbourhood $\Ball(h)$ that is visited infinitely many times by the geodesic $(\infty,\xi^+)\bmod{\Gamma}$.

\subsubsection*{Approximation of algebraic numbers by rationals and quadratics}

Algebraic numbers of degree $>2$ are expected to behave like Lebesgue-normal numbers with respect to the Gauss-Kuzmin statistics (see Remark~\ref{rem:Gauss-Kuzmin}), in particular they would satisfy $\limsup_{n} \tfrac{\log x_n}{\log n}= 1$,
% $\lim_n \tfrac{1}{n} \log q_n=\tfrac{\pi^2}{12\log(2)}$, and 
and should be neither too well nor too badly approximable by rationals or quadratics.
However at this time of writing, very little is known about this.
Indeed, there is not a single example (explicit or not) of an algebraic number $\xi^+$ of degree $>2$ for which we can decide whether its continued fraction expansion $x$ is bounded or unbounded.
A longstanding conjecture going back to Khintchine (\cite[§4]{Shallit_Badly-approximable-numbers_2023}) asserts that if $x$ is bounded then $\xi^+$ is rational or quadratic or transcendental.
This trichotomy is the leitmotive of this work.

The best result we have concerning good rational approximations to $\xi^+$ (upper bounds on $x$), remains Roth's theorem implying that if an irrational $\xi^+$ is algebraic then $\log(x_{n+1})=o(\log q_n)$, thus algebraic numbers cannot have exponentially better rational approximations than normal numbers (for which $\limsup \log(x_{n+1})\sim \log(n) \sim \log \log q_n$).
This dramatic improvement on Liouville's theorem enables to extend the construction of transcendental numbers, but we need more.

The main improvement of Roth's theorem is Schmidt's subspace theorem, which implies in particular that if an irrational $\xi^+$ is algebraic then it cannot have too good quadratic approximations. This means that the geodesic $\xi^\prime \subset \M$ cannot get infinitely often too close to a periodic geodesics.
The sequence $x\in (\N_{\ge 1})^\N$ has \demph{long repetitions} when there is $\epsilon>0$ such that there are infinitely many times $n\in \N$ at which the prefix $(x_0,\dots,x_n)$ contains two disjoint occurrences of a word over $\N_{\ge 1}$ of length $>\epsilon n$.
This enabled~\cite{Bugeaud_Transcendence-stammering-continued-fractions_2013} to show that if $x$ is bounded aperiodic but with long repetitions then $\xi^+$ is transcendent.

% To be precise, the sequence $x\in (\N_{\ge 1})^\N$ has \demph{long repetitions} when there is $\epsilon>0$ such that there are infinitely many times $n$ at which the prefix $(x_0,\dots,x_n)$ contains two disjoint occurrences of a word over $\N$ of length $>\epsilon n$.
% This means that the geodesic $\xi \subset \M$ gets infinitely often very close to a periodic geodesic, and that $\xi^+$ is well approximated by quadratic numbers.
% By~\cite{Bugeaud_Transcendence-stammering-continued-fractions_2013}, to show that if $x$ is bounded and has long repetions, then $\xi$ is either quadratic or transcendent.

Let us mention that we may unify the last two transcendence results for irrational $\xi^+$ as follows.
Say that $X\in \{R,L\}^\N$ has long repetitions when there is $\epsilon>0$ such that there are infinitely many times $N\in \N$ at which the prefix $(X_0,\dots,X_n)$ contains two disjoint occurrences of a factor $U\in \{R,L\}^\star$ with $\len(U)>\epsilon N$.
If $X$ has long repetitions then $\xi^+$ is quadratic or transcendent.

\subsubsection*{Mahler's exponents measure to approximation by algebraic numbers of degree $\le d$}

Let us mention Mahler's exponents and hierarchy of reals (which we will explain in Subsection~\ref{subsec:Lagrange-Mahler}).
For each degree $d\in \N_{\ge 1}$, there is a constant $w_d(\xi^+) \in [0,+\infty]$ measuring the quality of approximations to $\xi^+$ by algebraic numbers of degree $\le d$.
The growth of the increasing sequence $d\mapsto w_d(\xi^+)$ determines Mahler partition of numbers into four classes: $\mathrm{A}, \mathrm{S}, \mathrm{T}, \mathrm{U}$.
The class $\mathrm{A}$ coincides with the set of algebraic numbers, while Lebesgue almost all numbers are of class $\mathrm{S}$.
Our geometric setting leads to approximations by quadratic numbers, so Mahler's exponent $w_2$ will play a special role.

\subsection{Transcendence of simple geodesics on modular covers}

 Let us now explain the main Theorem~\ref{Intro_thm:simple-geodesic-trichtomy} of this work, which is the content of Section~\ref{sec:TranSimple}.

Consider a finite-index subgroup $\Gamma^\prime \subset \Gamma$, corresponding to a finite cover $\M^\prime\to \M$.
The cusps of $\M^\prime$ correspond to the orbits of rational points $\Q\Proj^1 \bmod{\Gamma^\prime}$, hence to the conjugacy classes of parabolic elements in $\Gamma^\prime$.
A cusp $o \in \Q\Proj^1 \bmod{\Gamma^\prime}$ has a \demph{width} defined as the index of stabilisers $w=[\Stab(o,\Gamma)\colon \Stab(o, \Gamma^\prime)]$.
% The \demph{maximal cusp-width} of $\Gamma^\prime$ is the maximal with of its cusps (when $\Gamma^\prime$ is normal in $\Gamma$ all cusps have the same width).

The space of geodesics of $\HP$ identifies with $\Geo(\HP)=\R\Proj^1 \times \R\Proj^1 \setminus \mathrm{diagonal}$.
The space of geodesics on $\M^\prime$ identifies with the quotient $\Geo(\M)=\Gamma^\prime \backslash \Geo(\HP)$ by the diagonal $\Gamma^\prime$-action.
A geodesic $\xi^\prime\subset \M^\prime$ has a future whose geometry and dynamics are governed by the continued fraction expansion of $\xi^+$; in particular: $\xi^\prime$ escapes into a cusp if and only if $\xi^\prime$ is rational, whereas $\xi^\prime$ accumulates on a periodic geodesic if and only if $\xi^+$ is quadratic.
% when $(\xi^\prime,\xi^+)$ form a pairs of real quadratic numbers that are Galois conjugates.

A geodesic $\xi^\prime \subset \Geo(\M^\prime)$ is \demph{simple} when there are no $\gamma \in \Gamma^\prime$ such the pairs $\{\xi^-,\xi^+\}$ and $\{\gamma \xi^-, \gamma \xi^+\}$ are disjoint and linked with respect to the cyclic order on $\partial \HP$, namely it has no transverse intersection points in $\M'$. 
The space of simple geodesics on $\M^\prime$ is a closed subset $\GeoS(\M^\prime)\subset \Geo(\M^\prime)$.

\begin{theorem}[transcendence of simple geodesics in finite modular covers]
    \label{Intro_thm:simple-geodesic-trichtomy}
    Consider a finite index subgroup $\Gamma^\prime \subset \Gamma$ corresponding to a finite cover $\M^\prime\to \M$.
    
    A geodesic $\xi =(\xi^-,\xi^+) \in \Geo(\HP)$ projects $\bmod{\Gamma^\prime}$ to a geodesic $\xi^\prime \in \Geo(\M^\prime)$.

    If $\xi^\prime$ is simple then $\xi^+\in \R\Proj^1{}$ is either rational or quadratic or transcendental.

    More precisely in the transcendental case, either $w_2(\xi^+)=\infty$ (of type $\mathrm{U}_2$) or else there exists $c\in \R_{>0}$ such that for all $d\in \N_{\ge 1}$ we have $w_d(\xi^+)\le \exp\left(c (\log 3d)^3 (\log \log 3d)^2\right)$ (of class $\mathrm{S}$ or $\mathrm{T}$).
\end{theorem}

\begin{proof}[Outline and moral of the proof]
    Assume that $\xi^+$ is irrational.
    We may act by $\PSL_2(\Z)$ to assume that $(\xi^-,\xi^+)\in (-1,0)\times (1,\infty)$, hence the continued fraction expansions $-1/\xi^-=\Ecf{x_{-1}; x_{-2},\dots}$ and $\xi^+=\Ecf{x_{0};x_{1},x_{2},\dots}$ encode the trajectory of $\xi^\prime$ on $\M^\prime$ as a sequence $x=(x^-,x^+)\in (\N_{\ge 1})^\Z$. 

    First, we will recall in Lemma~\ref{lem:simple-low-cusp-excursions} the well known fact (extensively used by 
   ~\cite{Lehner-Sheingorn_self-intersections-modular-covers_1985}) that the simplicity of $\xi^\prime$ implies that $x$ must be bounded by the maximal cusp-width of $\Gamma$.
    At this stage, Khintchine's conjecture would imply that $\xi^+$ should be quadratic or transcendental.
    
    Next, we will show in~\ref{prop:simple-roughly-linear-complexity} that if $\xi^\prime$ is simple then the sequence $x$ must have ultimate affine factor complexity hence long repetitions (this relies on the symbolic combinatorics of geodesic laminations and interval exchanges or linear involutions).
    Geometrically, if a simple geodesic in the finite volume surface $\M^\prime$ does not escape to the cusp, then it must be recurrent to some interval of the surface, and each time its position returns to that small interval, its momentum must also align closely to that of the previous return point, so much so that the geodesic will travel almost parallel to itself for some time, which is proportional to the return time.
    
    Thus topological simplicity implies that the bounded continued fraction expansion $x^+ \in (\N_{\ge 1})^\N$ has sub-affine factor complexity, hence long repetitions.
    This enables to apply~\cite[Theorem 1.1]{Bugeaud_Automatic-continued-fractions-transcendental-quadratic_2013} on the transcendence of bounded continued fractions with long repetitions and~\cite[Theorem 2.2]{Bugeaud_contfrac-low-complexity-transcendence-measures_2012} to control the Mahler exponents of $\xi^+$ from the sub-affine complexity of $x^+$. 

    In conclusion, we are able to prove transcendence not only because $x^+$ is bounded (which only says that $\xi^+$ is badly approximable by rationals), but especially because $x$ has long repetitions and subaffine factor complexity (which is saying that $\xi^+$ has too good quadratic approxmiations).
\end{proof}

\begin{remark}[transcendence measures for purely morphic geodesics]
    For geodesics that are fixed by pseudo-Anosov classes, our Theorem~\ref{thm:transcend-morphic-geodesics} will refine the Mahler measures in Theorem~\ref{Intro_thm:simple-geodesic-trichtomy}.
\end{remark}

\begin{comment}
\begin{remark}[arithmetic Fuchsian groups]
    We will generalise all this work from $\Gamma=\PSL_2(\Z)$ (the fundamental group of the modular orbifold) to an arithmetic cocompact Fuchsian group $F\subset \PSL_2(\mathcal{O}_\K)$ (the fundamental group of a compact Shimura curve).
    %
    This will first require to generalize the (hard) work in~\cite{Bugeaud_Transcendence-stammering-continued-fractions_2013, Bugeaud_contfrac-low-complexity-transcendence-measures_2012}, from bounded continued fractions to certain products of hyperbolic matrices in $F$.
    For this, we will need to apply the (quantitave) Schmidt subspace theorem over the invariant trace field of $F$, and use the property that the trace of hyperbolic matrices in $F$ dominates its conjugates in absolute value (see~\cite[Theorem 2.4]{Bugeaud-Hubert-Schmidt_transcendenc-Rosen_2013}).
\end{remark}
\end{comment}

\subsubsection*{Definition and questions on profinitely simple geodesics}

Let us give a name to the (pairs of) numbers covered by Theorem~\ref{Intro_thm:simple-geodesic-trichtomy}.

\begin{definition}[profinite simple]
    \label{Intro-def:profinitely-simple}
    Define the subset $\widetilde{\GeoS}(\Gamma)\subset \Geo(\HP)$ of \demph{profinite simple geodesics} (for $\Gamma$) as consisting of those $\xi \in \Geo(\HP)$ such that there exists a finite index subgroup $\Gamma^\prime \subset \Gamma$ such that $\xi^\prime = \xi \bmod{\Gamma^\prime} \in \GeoS(\M^\prime)$ is simple.
    Define the subset $\widetilde{\GeoS}^+(\Gamma) \subset \R\Proj^1{}$ of \demph{profinite simple numbers} (for $\Gamma$) as its image by the projection on (any) one of the two factors $\Geo(\HP)\to \R\Proj^1{}$.

    We may define the subsets of \demph{congruence-profinite simple} geodesics and numbers $\mathcal{CS}^+(\Gamma)$, by restricting to congruence covers $\M(N)\to \M$ associated to the congruence subgroups $\Gamma(N)\subset \Gamma$ defined for $N\in \N_{>2}$ as the kernel of the reduction $\bmod{N}$ morphism $\PSL_2(\Z)\to \PSL_2(\Z/N)$.
\end{definition}

\begin{remark}[Hausdorff dimension]
    \label{Intro-rem:profinitely-simple-dimH}
    It follows from~\cite{Birman-Series_Hausdorff-dimension-simple-geodesics_1985} that for every finite cover $\M^\prime\to \M$ we have $\dim_H \GeoS(\M^\prime)=0$.
    Thus $\widetilde{\GeoS}(\Gamma)\subset (\R\Proj^1{})^2$ hence $\widetilde{\GeoS}^+(\Gamma)\subset \R\Proj^1{}$ have Hausdorff dimension $0$.
\end{remark}

We ask a couple of questions regarding (congruence-)profinite simple numbers, namely~\ref{quest:profinitely-simple-contfrac} on the symbolics of their continued fraction expansions, as well as~\ref{quest:profinitely-simple-Lagrange-Spec} and~\ref{quest:profinitely-simple-Mahler-Spec} about the topology of their Lagrange spectra $\LC(\mathcal{CS}^+(\Gamma))\subset \LC(\mathcal{S}^+(\Gamma))\subset \LC(\R)$, and Mahler spectra $w_2(\widetilde{\mathcal{CS}}(\Gamma)) \subset w_2(\widetilde{ \GeoS}(\Gamma))$ as well as $\hat{w}_2(\widetilde{\mathcal{CS}}(\Gamma)) \subset \hat{w}_2(\widetilde{\GeoS}(\Gamma)) \subset [2,1+\phi]$.

\begin{comment}
\begin{question}[continued fractions]
    A combinatorial characterisation of the continued fractions of profinite simple numbers may be given: they are all $\{L,R\}$-sequences associated to lamination languages in trivalent graphs, which have mostly been characterised in~\cite{Ferenczi-Zamboni_Languages-k-IET_2008, Lopez-Narbel_lamination-languages_2013}.
    
    What if we restrict to (congruence-)profinite simple numbers? Can we characterize describe the continued fraction expansions of ends of simple geodesics in the congruence cover $\M(N)$? (The case $N=3$ boils down to the Markov and Sturmian sequences discussed in Section~\ref{sec:ModTorus}.) 
\end{question}

\begin{question}[Lagrange spectrum]
    By Lemma~\ref{Intro_lem:simple-implies-low-cusp-excursions}, profinitely simple reals have finite Lagrange constant: can we describe the metric topology of the subsets $\LC(\mathcal{CS}^+(\Gamma))\subset \LC(\mathcal{S}^+(\Gamma))\subset \LC(\R)$?
    What is their Hausdorff dimension? (Note: $\exists S\subset \R$ with $\dim_H(S)=0$ but $\dim_H(\LC(S))=1$.)
\end{question}
\end{comment}

\subsection{Simple geodesics in the modular torus: Christoffel, Markov, Sturm}

In Section~\ref{sec:ContFrac} we explain the motivation for this work, namely why Theorem~\ref{Intro_thm:simple-geodesic-trichtomy} was known for the derived subgroup $\Gamma^\prime=[\Gamma,\Gamma]$ corresponding to the modular cover $\M'\to \M$ whose total space is a hyperbolic once-punctured torus.

The simple geodesics $\xi'\subset \M'$ fall into a finite number of typs which all have a very description in terms of continued fraction expansion.
We recall this well known story in the spirit of~\cite{Series_Geo-Markov-Num_1985}, while emphasizing some details either missing from or dispersed in the literature.
The trichotomy distinguishes Markov rational numbers, Markov (pairs of) quadratic numbers and Sturmian (pairs of) transcendental numbers, which are precisely those whose (Markoff and) Lagrange constant lie before and on the first accumulation point $3$ of the spectra.

The Lagrange spectrum $\LC(\R_{\ge 1})$ begins with a discrete countable set increasing to its first accumulation point $3$, and the corresponding $\xi^+$ are precisely the $\PSL_2(\Z)$ orbits of Markov quadratic irrationals corresponding to simple periodic geodesics in $\M'$.
There are uncountably many numbers $\xi^+\in \R_{>1}$ with $\LC(\xi^+)=3$, which we call Sturmian since up to the action of $\PGL_2(\Z)$ they are precisely the  to numbers whose continued fraction expansions are the aperiodic Sturmian sequences on $\{1,2\}$, and they correspond to simple aperiodic geodesics in $\M'$.
The transcendence of Sturmian numbers was shown in~\cite[Theorem 7]{ADQZ_transcendence-sturmian_2001}.

We will also take this opportunity to describe the structure of the metabelian quotient $\Gamma/\Gamma''$ (in Theorem~\ref{thm:CLS_hexagonal-group}, first observed in~\cite[Section 3.2]{CLS_phdthesis_2022}), as we believe that it deserves more attention. 
We are working on its generalisation to modular groups of higher genus Shimura curves and its relation to certain modular forms for symplectic groups.

\begin{remark}
    Our survey of diophantine properties about Sturmian numbers will not be exhaustive.
    We refer to~\cite{Bugeaud-Laurent_Diophantine-exponents-Sturmian_2005}, as well as the recent works~\cite{Roy_Markoff-Lagrange-Extremal_2011, Roy-Zelo_measures-approx-Sturmian_2011} about the fine diophantine properties of characteristic Sturmian numbers whose slope is in the $\PGL_2(\Z)$ orbit of $\phi$.
\end{remark}

%% file: sec1-ContFrac.tex
\section{The modular orbifold and continued fractions}
\label{sec:ContFrac}

This background section recalls some material about the modular orbifold and continued fractions from~\cite[Chapter 2]{CLS_phdthesis_2022} and diophantine approximation from~\cite{Haas_geometry-Markoff-forms_1987} and results on transcendence of bounded continued fractions iwth low complexity from~\cite{Adamczewski-Bugeaud_transcendence-measures-contfrac_2010,  Bugeaud_contfrac-low-complexity-transcendence-measures_2012}.

\subsection{Modular group and its conjugacy classes}
\label{subsec:M-geodesics}

The modular group $\Gamma = \PSL_2(\Z)$ acts on the ideal triangulation $\triangle$ of $\HP$ with vertex set $\Q\Proj^1{}$ and edges $\left(\frac{a}{c},\frac{b}{d}\right)$ such that $ad-bc=1$, freely transitively on its flags, corresponding to the half-edges of its dual trivalent tree $\TT$.
The connected components of $\HP \setminus \TT$ correspond to the vertices of $\triangle$, parametrized by $\Gamma/ \langle R \rangle$.
Denote
\begin{equation*}
    S =
    \begin{pmatrix}
    0 & -1\\
    1 & 0
    \end{pmatrix}
    \quad
    T =
    \begin{pmatrix}
    1 & -1\\
    1 & 0
    \end{pmatrix}
    \qquad
    L = % T^{-1}S
    \begin{pmatrix}
    1 & 0\\
    1 & 1
    \end{pmatrix}
    \quad
    R = % TS^{-1}
    \begin{pmatrix}
    1 & 1\\
    0 & 1
    \end{pmatrix}.
\end{equation*}
\begin{figure}[h]
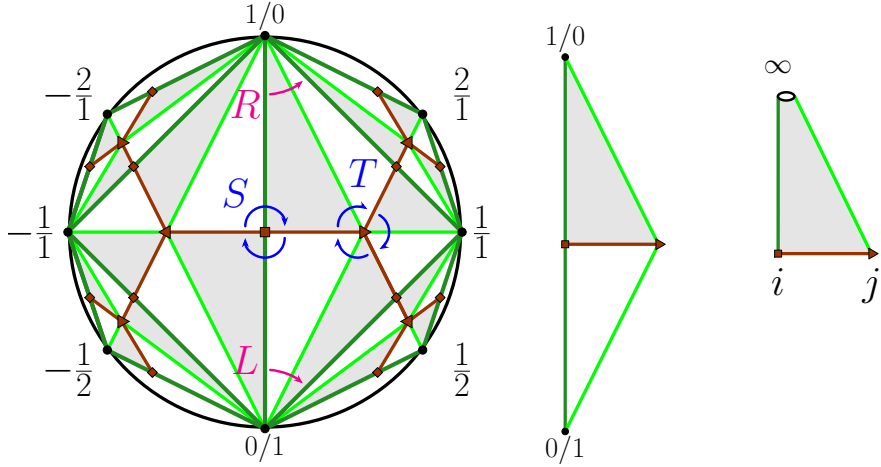

    \centering
    \vspace{-1cm}
    % \scalebox{0.9}{\subfile{images/tikz/action-L-R-Hex}}
    \scalebox{0.62}{\subfile{images/tikz/PSL2Z-pavage-PH}}
    \caption{Projective model of $\HP$, with its ideal triangulation $\triangle$ and dual tree $\TT$.
    Fundamental domain under the action of $\Gamma = \PSL_2(\Z)$. The orbifold $\M = \Gamma \backslash \HP$.}
\end{figure}

The modular orbifold $\M = \Gamma \backslash \HP$ has genus zero, a cusp associated to the fixed point $\infty \in \Q\Proj^1 \subset \partial \HP$ of $R$, as well as two conical singularities of order two and three associated to the fixed points $i,j \in \HP$ of $S$ and $T$.
Thus $\Gamma = \pi_1(\M)$ is the free amalgam of its subgroups $\Z/2$ and $\Z/3$ generated by $S$ and $T$. 
(More precisely $\SL_2(\Z)$ is the amalgam of its subgroups $\Z/4$ and $\Z/6$ generated by $S$ and $T$ over their intersection $\Z/2$ generated by $S^2=-\Id=T^3$.)%, and we will often write this relation even when working in $\PSL_2(\Z)$ where $-\Id \equiv \Id$.)

It follows that in $\Gamma$, a finite order element is conjugate to a power of $S$ or $T$; whereas an infinite order $A$ is conjugate to $\prod_{1}^{l} T^{\epsilon_i}S^{-\epsilon_i}$ for a unique $l\in \N^*$ and a unique sequence of $\epsilon_i = \pm 1$ up to cyclic permutation: this corresponds to a cyclic word in $R=TS^{-1}$ and $L=T^{-1}S$, and we may thus define its $\{R,L\}$-length $\len(A)=l=\#R+\#L$ and its Rademacher number $\Rad(A)=\sum_{1}^{l} \epsilon_i = \#R-\# L$.
%\begin{equation*}\textstyle
%    \Rad(A)= \sum_{1}^{l} \epsilon_i = \#R-\# L.
%\end{equation*}

%We define the Rademacher numbers in $\SL_2(\Z)$ as $\Rad(S^\epsilon)=\epsilon/2$ and $\Rad(T^\epsilon)=\epsilon/3$.

Conjugacy classes in $\Gamma$ correspond to orbifold homotopy classes of loops on $\M$.
The elliptic and parabolic classes correspond to loops circling around the singularities and cusp, whereas every hyperbolic class is represented by a unique closed geodesic.

An element of $\Gamma$ is primitive when it generates a maximal cyclic subgroup. This notion is invariant by conjugacy. Correspondingly, a cyclic $\{R,L\}$-word is primitive when it is not a proper power, and a closed geodesic is primitive when it does not wind several times around itself.

A complete geodesic of $\HP$ is uniquely determined by the sequence of triangles in $\triangle$ that it intersects: it corresponds to a geodesic of the dual tree $\TT$ which is not horocyclic (bounding a region of the complement $\HP\setminus \TT$).
Such geometric and combinatorial geodesics are equally determined by their common endpoints in $\R\Proj^1{}$.

A closed geodesic of $\M$ lifts to the geodesics in $\HP$ corresponding to the periodic geodesics of $\TT$ whose period is given by $\{R,L\}$-cycle of the associated conjugacy class.
Their endpoints form the $\Gamma$-orbit of a pair of Galois-conjugate quadratic irrationals.
Such geodesics determine a unique hyperbolic conjugacy class in $\Gamma$ which is primitive. 

\subsection{Euclidean monoïd and continued fraction expansions}
\label{subsec:Euclid-contfrac}

The monoïd $\SL_2(\N)\subset \SL_2(\Z)$ is freely generated by the transvections $L$ and $R$, and
it identifies with its image $\PSL_2(\N)\subset \PSL_2(\Z)$ which we call the Euclidean monoïd.

In $\PSL_2(\Z)$, the conjugacy classes of torsion elements are those of $\Id,S,T,T^{-1}$, whereas the conjugacy class of an infinite order element intersects the Euclidean monoïd $\PSL_2(\N)$ along the cyclic permutations of a unique $\{R,L\}$-word.
Hence the hyperbolic geodesics of $\M$ are indexed by cyclic words in $\{R,L\}$ containing both letters.

Observe that for $A\in \SL_2$ we have $A^\dag = SA^{-1}S^{-1}$, hence transposition of cyclic words in $\PSL_2(\N)$ corresponds to inversion of the associated conjugacy classes in $\PSL_2(\Z)$.

% \subsubsection*{Continued fraction expansions of real numbers}

Every number \(\xi \in (0,+\infty]\) admits a unique Euclidean continued fraction expansion:
\begin{equation*}
%\textstyle
\Ecf{x_0,x_1,\dots}=x_0+\frac{1}{x_1+\dots} = R^{x_0}L^{x_1}\cdots \infty
%\Ecf{x_0,x_1,x_2,\dots}=x_0+\frac{1}{x_1+\frac{1}{x_2+\dots}}
\quad \mathrm{with} 
\quad x_j \in \N %_{\ge 0}
\quad \mathrm{and} 
\quad \forall j>0,\; n_j>0
\end{equation*}
which is finite if and only if $\xi$ is rational, in which case it must have an even number of terms (as the notation implies that the $\{R,L\}$-word begins with a power of $R$ and ends with a power of $L$).
The involution $\xi\mapsto S(\xi) = -1/\xi$ with no fixed points yields a partition $\R\Proj^1{} =(0,\infty] \sqcup\, S\cdot (0,\infty]$, hence every $\xi \in\R\Proj^1{}$ admits a unique representation of the form $\xi=\Ecf{x_0,x_1\dots}$ or $\xi=-1/\Ecf{x_0,x_1\dots}$ (including $\infty =\Ecf{}$ and $0 = -1/\Ecf{}$).

Now consider the action of the modular group $\PSL_2(\Z)$ on $\R\Proj^1{}$ and of its Euclidean submonoïd $\PSL_2(\N)$ on $[0,\infty]$.
If $\xi_i$ denotes the \(i^{\mathrm{th}}\) remainder of \(\xi\in (0,\infty]\) given by the tail $\Ecf{x_i,\dots}$ of its continued fraction expansion, then $\xi_0=(R^{x_0}L^{x_1})\xi_2$.
Hence the orbits of \(\alpha,\beta \in (0,\infty]\) under $\PSL_2(\N)$ have non-empty intersection if and only if there exist even starting points $i,j$ at which the tails $\alpha_i$ and $\beta_j$ coincide.
% Since $\{R,L\}$ generate the group $\PSL_2(\Z)$, 
Consequently \(\alpha,\beta \in (0,\infty]\) belong to the same $\PSL_2(\Z)$-orbit if and only if there exist even starting points $i,j$ at which the tails $\alpha_i$ and $\beta_j$ coincide.

\begin{example}[translation axes of elements in $\PSL_2(\Z)$]
	A hyperbolic $C\in \PSL_2(\Z)$ has translation axis $(\gamma^-,\gamma^+)\in \partial \HP \times \partial \HP$ connecting a pair of Galois conjugate real quadratic irrationals.
	We have $C\in \PSL_2(\N)\iff \gamma^-<0<\gamma^+$ and $C\in R\PSL_2(\N)L \iff (\gamma^-,\gamma^+)\in (-1,0)\times (1,\infty)$. 
	
	The axis of $C=R^{c_0} \dots L^{c_{k}}\in R\cdot \PSL_2(\N) \cdot L$ has endpoints the Galois conjugate pair of quadratic numbers $(\gamma^-,\gamma^+)\in (-1,0)\times (1,\infty) \subset \partial \HP \times \partial \HP$ given by the purely periodic continued fraction expansions:
	\begin{equation*}
		\gamma^+ = \Ecf{(c_0, \dots, c_{k})^\N} 
		\quad \mathrm{and} \quad
		-1/\gamma^- = \Ecf{(c_k, \dots, c_{0})^\N}.
	\end{equation*}
\end{example}

Any geodesic $(\xi^-,\xi^+)\subset \HP$ intersects $(\infty,0)$ positively if and only if $\xi^-<0<\xi^+$.
Moreover if $-1<\xi^-<0$ and $1<\xi^+<\infty$, then it intersects $\triangle$ along a sequence of triangles whose encoding in $\{R,L\}^{\Z}$ is obtained from the continued fraction expansions of $\xi^+$ and $-1/\xi^-$ by concatenating the transpose of the latter with the former.
% The following result can be traced back to the works of Gauss~\cite{Gauss_disquisitiones_1807} and Galois~\cite{Galois_fraction-continue_1828}.

Every geodesic on $\M$ has a lift in $\HP$ 
(namely every geodesic in $\HP$ has a $\PSL_2(\Z)$-translate)
whose endpoints $\xi^\pm$ satisfy $-1<\xi^-<0$ and $1<\xi^+<\infty$: it intersects $\triangle$ along a sequence of triangles whose encoding in $\{R,L\}^{\Z}$ is obtained from the continued fraction expansions of $\xi^+$ and $-1/\xi^-$ by concatenating the transpose of the latter with the former.
The closed geodesics on $\M$ correspond to the periodic sequences, hence to the $\Gamma$-orbits of pairs $(\gamma^-,\gamma^+)$ of Galois-conjugate quadratic irrationals.

For distinct $\xi^-,\xi^+ \in \R\Proj^1{}$, the geodesic $(\xi^-,\xi^+) \bmod{\PSL_2(\Z)}\subset \M$ escapes (in the past or future) to the cusp when $\xi^-$ or $\xi^+$ is rational.
Otherwize, the geodesic intersects the segment $[i,j]\subset \M$ infinitely many times, with intervals given by the entries in the continued fraction expansions of $-1/\xi^-$ and $\xi^+$, as we now explain.

\subsection{Measuring approximation of reals by rationals and quadratics}

\label{subsec:Lagrange-Mahler}

\subsubsection*{Lagrange constants to measure approximation by rational numbers}

Let us recall from~\cite{Aigner_Markov-uniqueness-100-years_2013} the Lagrange constant of a real number and the Markov constant of a pair of distinct real numbers (or equivalently of a geodesic).

For  $\xi^+ \in \R$, its \demph{Lagrange constant} $\LC(\xi^+)$ is the supremum of $\LC \in \R_+$ such that there are infinitely many $p,q\in \Z\times \N^*$ with $\lvert \xi -p/q \rvert < 1/(\LC q^2)$, in formula $\LC(\xi^+)^{-1}= \liminf_{q}({q^2\lvert \xi -p/q \rvert})$.
%\begin{equation*}\textstyle
%    \LC(\xi^+)^{-1}= \liminf_{q}({q^2\lvert \xi -p/q \rvert}).
% \end{equation*}
It may be expressed in terms of the continued fraction expansion $\xi^+ = \Ecf{x_0,x_1,\dots}$:
\begin{equation*}\textstyle
    \LC(\xi^+)= \limsup_{n}({\Ecf{0, x_{n-1},\dots, x_{0}}+\Ecf{x_{n},x_{n+1},\dots}})
\end{equation*}
In particular $\LC(\xi^+)=0 \iff \xi^+ \in \Q$ and more importantly $\LC(\xi^+)=+\infty \iff \limsup_n(x_n) = +\infty$.
Moreover, for $\xi \notin \Q$ we have $\LC(\xi)\ge \limsup_n \left(\Ecf{0, 1,\dots, 1}+\Ecf{1,1,\dots}\right) = \sqrt{5}$.
The Lagrange constant $\LC(\xi^+)$ only depends on the tail of the continued fraction expansion $(x_n)$ (by the $\limsup$), hence it is invariant by the action of $\PSL_2(\N)$.
% Geomtrically in $\M$, it measures the excursions of the geodesic $(\infty, \xi)\bmod \PSL_2(\Z)$ towards the cusp $\Q \bmod \PSL_2(\Z)$.

In $\M$, the horoball neighborhood $\Ball(h)$ of the cusp $\infty$ at height $h\ge 1$ is the projection of $\{\Im(z) > h\} \subset \HP$, having area $1/h$.
The geodesic $(\infty,\xi^+)\in \HP$ projects to a geodesic on $\M$ whose penetrations into $\Ball(h)$ are indexed by $n\in \N$ such that $\tfrac{1}{2}\left(\Ecf{0, x_{n-1},\dots, x_{0}}+\Ecf{x_{n},x_{n+1},\dots}\right)>h$.

\begin{figure}[h]
	\centering
	\scalebox{0.87}{\subfile{images/tikz/cusp-excursion}}
	\caption{The geodesic $(0,\xi^+)\bmod \PSL_2(\Z)$ penetrates $\Ball(h)\subset \M$ each time $n\in \N$ satisfies $\tfrac{1}{2}\left(\Ecf{0, x_{n-1},\dots, x_{0}}+\Ecf{x_{n},x_{n+1},\dots}\right)>h$.}
\end{figure}

% \begin{comment}
For a geodesic $(\xi^-,\xi^+) \in \Geo(\M)$, its \demph{Markov constant} $\MC(\xi^-,\xi^+)$ is the supremum of $2h\in \R_+$ such that it is disjoint from $\Ball(h)$.
If $(\xi^-,\xi^+)\in [-1,0)\times [1,\infty)$ with $-1/\xi^- = \Ecf{x_{-1}; x_{-2}, \dots}$ and $\xi^+=\Ecf{x_0, x_{1}, x_{2}, \dots}$, the geodesic $(\xi^-,\xi^+)\subset \HP$ intersects the ideal triangles of $\triangle$ starting from the base edge $(0,\infty)\subset \HP$ according to the sequence of $(x_{n})$.
It projects to a geodesic on $\M$ which penetrates $\Ball(h)$ each time $\tfrac{1}{2}\left(\Ecf{0, x_{n-1}, x_{n-2}, \dots}+\Ecf{x_{n}, x_{n+1}, \dots}\right)>h$, thus
\begin{equation*}\textstyle
    \MC(\xi^-,\xi^+)=\sup_n (\Ecf{0, x_{n-1}, x_{n-2}, \dots}+\Ecf{x_{n}, x_{n+1}, \dots}).
\end{equation*}
% \end{comment}

\begin{remark}[Gauss-Kuzmin statistics]
\label{rem:Gauss-Kuzmin}
The continued fraction expansions of random numbers follow the Gauss-Kuzmin statistics as we now recall.

The Gauss–Kuzmin measure $\mu_{GK}$ on $(0,1)$ is defined on intervals $I\subset (0,1)$ by $\mu_{GK}(I) = \tfrac{1}{2}\int_I \frac{dt}{1+t}$.
For a finite sequence of positive integers $x=(x_1,\dots, x_n)\in (\N^*)^n$ the interval of real numbers in $(0,1)$ whose continued fraction expansion begins with $(0,x_1,\dots,x_n)$ forms an interval denoted $I(x)$.
For a real number sampled according to the Lebesgue measure, the frequency with which the finite sequence $x$ of positive integers appears in its continued fraction expansion is given by $\mu_{GK}(I(x))$. 
For instance on cylinders $x=(x_1)$ of length $1$, we find that the probability to witness the integer $x_1$ is given by $\log_2\left(1+\tfrac{1}{x_1(2+x_1)}\right)$

We refer to~\cite{Singh-Zhang-Hildebrand_elem-charac-gauss-kuzmin_2025} for an elegant characterisation of the Gauss-Kuzmin measure.
\end{remark}

\begin{remark}[transcendence badly approximable numbers?]
\label{rem:BAD-implies-ratio-quad-transcend}
A positive real number whose continued fraction expansion has bounded entries is called \demph{badly approxmimable}.
A longstanding conjecture going back to Khintchine (see~\cite[§4]{Shallit_Badly-approximable-numbers_2023}) asserts a badly approximable numbers is rational or quadratic or transcendental.
This trichotomy is the leitmotiv of this work. 
Let us warn that to this date, there is not a single example (explicit or non-explicit) of algebraic number with degree $>2$ for which we can decide whether its continued fraction expansion is bounded or unbounded.

In fact, it is more generally expected that algebraic numbers of degree $>2$ are normal with respect to the Gauss-Kuzmin statistics, which would imply in particular that their corresponding $\{R,L\}$-expansion has exponential factor complexity (a notion that we will introduce in \Cref{subsec:transcend-low-complexity}).
\end{remark}

\subsubsection*{Mahler exponents measure to approximation by algebraic numbers of degree $\le d$}

Let us briefly recall Mahler's exponents and hierarchy of real numbers (see~\cite{Bugeaud_Approximation-by-algebraic_book_2004, Bugeaud-Laurent_Diophantine-exponents-Sturmian_2005}).

The \demph{height} of a polynomial $P(x)=\sum_0^d a_k x^k\in \Z[x]$ is $H(P)=\max\{\lvert a_k\rvert \colon 0\le k\le d\}\in \Z$, and the \demph{height} of an algebraic number $\alpha\in \Bar{\Q}$ is that of its minimal polynomial in $\Z[x]$.

For a real number $\xi^+\in \R$ and $d\in \N$, its \demph{Mahler exponent} $w_d(\xi)$ and $\hat{w}_d(\xi)$ for best and uniform approximations are defined as the supremum of $w\in \R$ such that the equations
    \begin{equation*}
        P\in \Z[x] \quad 
        \deg(P)\le d \quad 
        H(P) \le h \qquad 
        0<\lvert P(\xi^+) \rvert \le h^{-w}
    \end{equation*}
have solutions for infinitely many $h\in \N$ and for all sufficiently large $h\in \N$, respectively.

By definition and the Schubfachprinzip, for $d\in \N_{\ge 1}$, if $\xi^+\in \R$ is not algebraic of degree $\le d$, then $d\le \hat{w}_d(\xi^+)\le w_d(\xi^+)$. For Lebesgue-normal numbers we have $d = \hat{w}_d(\xi^+) = w_d(\xi^+)$.

By definition, $\hat{w}_d$ and $w_d$ are increasing. Letting $w_\infty(\xi^+)=\limsup_n \tfrac{1}{n}w_d(\xi^+) \in [0,\infty]$, we may now give Mahler's hierarchy of numbers: the number $\xi^+$ is of \demph{Mahler class}
\begin{itemize}[noitemsep]
    \item[$\mathrm{A}$] when $w_\infty(\xi^+)=0$ 
    \item[$\mathrm{S}$] when $0<w_\infty(\xi^+)<\infty$ 
    \item[$\mathrm{T}$] when $w_\infty(\xi^+)=\infty$ but for all $n\in \N_{\ge 1}$ it has $w_d(\xi^+)<\infty$ 
    \item[$\mathrm{U}$] when there exists $n\in \N_{\ge 1}$ such that $w_d(\xi^+)=\infty$
\end{itemize}
The class $\mathrm{A}$ coincides with the set of algebraic numbers, and  the Schmidt subspace theorem implies that if $\xi^+\in \R$ is algebraic with $\deg(\xi) \in \N_{>1}$ then $w_d(\xi^+)=\hat{w}_d(\xi^+)=\min\{d,\deg(\xi^+)-1\}$.
Moreover, if numbers are algebraically dependent then they belong to the same Mahler class.
%
% Almost all numbers satisfy $w_d(\xi^+)\le n$, in particular thy are of class $\mathrm{S}$.

On the other hand, if $\xi^+\in \R$ is neither rational nor quadratic, then denoting $\phi=\tfrac{1+\sqrt{5}}{2}$ we have $2 \le \hat{w}_2(\xi) \le 1+\phi$, and there are such $\xi^+$ that achieve equalities (see~\cite[Theorem 2.7]{Bugeaud-Laurent_Diophantine-exponents-Sturmian_2005}).

Since our geometric setting naturally leads to approximations by (pairs of) quadratic numbers, the Mahler measure $w_2$ will play a special role.

\subsection{Transcendence of continued fractions with low complexity}
\label{subsec:transcend-low-complexity}

\subsubsection*{Complexity of sequences, languages and Subshifts}

Let us recall some notions about the complexity of languages and subshifts (see~\cite{Fogg_substitutions_2002}).
The aim is put the notions encountered for the sequences that will appear in the results on diophantine approximation in the broader context of symbolic dynamics, so that we may better relate them to the topology and homogeneous dynamics in the next Section.

Fix a set $\Al$ called the alphabet of \demph{syllables}, and let $\Al^{\star}=\bigcup_{n=0}^{\infty} \Al^n$ be the free monoïd of \demph{words}.
In such monoïds, the factors of a word are also called its subwords. 

A \demph{language} is a subset $\La \subset \Al^\star$.
It is called \demph{factorial} when closed under factors (if $U\in \Al^\star$ is a subword of $W\in \La$ then $U\in \La$) and called \demph{right-extendable} when every word is a proper left-factor of another one (if $U\in \La$ then there exists a non-trivial $W\in \Al^\star$ such that $UW\in \La$).
% The \demph{recurrence} function of a language $\mathcal{R}(\La) \colon \La \to \N\sqcup \{\infty\}$ sends $n\in \N$ to the smallest $N\in \N\sqcup\{\infty\}$ such that every word in $\La^N$ contains every word of $\La^N$ as a factor.
% The language $\La$ is \demph{uniformly recurrent} when for all $n\in \N$ we have $\mathcal{R}_n(\La)<\infty$ and \demph{linearly recurrent} when $\mathcal{R}(\La)$ is bounded by a linear function.
It is \demph{uniformly recurrent} when for every $U\in \La$ there is $n\in \N$ such that every word $W \in \La \cap \Al^n$ contains $U$ as a subword.
Its \demph{factor complexity} is the function $\fac(\La) \colon n\in \N\mapsto \Card{(\La\cap \Al^n)}\in \N$.

The set or \demph{right-infinite sequences} $\Al^\N$ with the product topology is compact; a basis of neighbourhoods for its topology is given by the clopen sets $\operatorname{Cyl}_k(U)= \{\Xx \in \Al^\Z \colon \Xx_{[k,k+n)}=U\}$ for $U\in \Al^n$ and $k \in \N$.
The function $\Shift\colon \Al^\N\to \Al^\N$ defined by $\Shift(\Xx)_{n} = \Xx_{n+1}$ is continuous.
A \demph{right-subshift} $(\XX, \Shift)$ of $(\Al^{\N}, \Shift)$ consists of a compact $\Shift$-invariant subset.
%
% A subshift $(\Xi^\prime,\Shift)$ is \demph{transitive} when it has a leaf $\xi^\prime$ whose an orbit $(\Shift^k\xi)_k$ is dense in $\Xi^\prime$ as $k\to +\infty$ and as $k\to -\infty$.
A right-subshift $(\XX,\Shift)$ is \demph{minimal} when every leaf $\Xx \in \XX$ has its orbit $(\Shift^k \Xx)_k$ that is dense in $\XX$ as $k \to + \infty$.

There is a correspondence from right-subshifts $(\XX,\Shift)\subset (\Al^\N, \Shift)$ to languages $\La(\XX)\subset \Al^\star$ that are factorial and right-extendable, given by $\La(\XX) = \{\Xx_{[k,k+n)} \colon \Xx \in \XX, k\in \Z , n\in \N\}$.
A subshift is minimal if and only if its language is uniformly recurrent.
There is an analogous discussion for bi-infinite sequences and subshifts, and a correspondence with factorial bi-entendable languages, and uniform recurrence corresponds to bi-minimality.

\begin{example}[from squences to languages and subshifts]
    A sequence $\Xx \in \Al^\N$ has \demph{shift-closure} $\operatorname{Closure}{(\{\Shift^k(\Xx)\colon k\in \N\})}\subset \Al^\Z$ a subshift whose language consists of all subwords of $X$, namely \(\La(\Xx) = \{\Xx_{[k,k+n)} \in \Al^n \colon n,k\in \N\}\).
    The sequence $\Xx$ inherits all the attributes and properties of that subshift language: factor complexity, uniform recurrence, and so on.
\end{example}

\begin{comment}
A sequence $\Xx\in \Al^\N$ is
\begin{itemize}
    \item[recurrent] when every subword appears infinitely often: for every $l\in \N$ there is $c(l)\in \N$ such that every subword of length $c(l)$ contains all subwords of length $l$
    \item[uniformly recurrent] when every subword appears infinitely often with bounded gaps: for every $l\in \N$ there is $c(l)\in \N$ such that every subword of length $c(l)$ contains all subwords of length $l$)
    \item[linearly recurrent] when it is uniformly recurrent and the gaps are linear in the size of the word: there exists $c\in \N$ such that for every $l\in \N$, every subword of length $cl$ contains all subwords of length $l$
\end{itemize}
\end{comment}

A sequence $\Xx\in \Al^\N$ is \demph{purely morphic} when it is a \demph{fixed point} of an endomorphism $\varphi\colon \Al^\star \to \Al^\star$, namely there exists a syllable $A\in \Al$ such that $\Xx=\lim_n \varphi^n(A)$. When a purely morphic sequence is uniformly recurrent (which happens as soon as its first syllable occurs at least twice), its factor complexity is roughly linear, namely there exists $k\in \N$ such that $\lvert \fac_n(\Xx)-kn\rvert < \infty$. 

\begin{comment}
A sequence is \demph{morphic} when it is the image by a morphism $\psi\colon \Al_0^\star \to \Al^\star$ of a fixed point of an endomorphism of $\Al_0^\star$.
Let us warn that in general, there are subtle differences between purely morphic and morphic sequences~\cite{Cassaigne-Nicolas_Factor-complexity_2010, Durand-Leroy-Richomme_S-adic-complexity_2013}.
However if $\Xx$ is morphic and uniformly recurrent, then its factor complexity is roughly linear (there is $k\in \N$ such that $\lvert \fac_n(\Xx)-kn\rvert <\infty$) and it is the letter-to-letter image of a purely morphic sequence.
\end{comment}

For a word $W\in \Al^\star$ and a real $r\in [0,\infty)$, define $W^r$ as the prefix of $\Al$-length $\lceil r\len(W) \rceil$ of the word $W^{\lceil r \rceil}$.
For example, over $\Al=\{A,B\}$, the word $W=ABB$ has length $3$, so $(ABB)^\pi$ is the prefix of length $\lceil 3\pi\rceil = 10$ of $(ABB)^4$, namely $(ABB)^\pi=ABBABBABBA$.

\begin{definition}[diophantine exponent and long repetitions]
    Consider a sequence $\Xx \in \Al^\N$.
    
    For $\rho \in [1,+\infty)$, say that $\Xx$ has \demph{exponent $\ge \rho$} when for all $N\in \N$ there are words $U,W\in \Al^\star$ and $r\in \R_{\ge 0}$ such that $WU^r$ is a prefix of $X$ with $\len(WU^r)\ge N$ and $\len(WU^r)/\len(WU)\ge \rho$.
    Its \demph{diophantine exponent} $\Dio(\Xx)\in [1,\infty]$ is the supremum of $\rho$ such that $\Xx$ has exponent $\ge \rho$.

    Say that $X$ has \demph{long repetitions} when $\Dio(\Xx)>1$, namely when for some $\epsilon >0$ there are infinitely many triples $U,V,W \in \Al^\star$ such that $WUVU$ is a prefix of $X$ and $\len(U)>\epsilon \len(WUVU)$.
\end{definition}

\begin{remark}[aperiodic complexity gap]
    \label{rem:aperiodic-complexity-gap}
    Let $\Xx\in \Al^\N$.
    If $\liminf_n \tfrac{1}{n}\fac_n(\Xx) > 0$ or $\Dio(\Xx)<\infty$ then $\Xx$ is aperiodic.
    Conversely if $X$ is aperiodic then we have $\liminf_n \tfrac{1}{n} \fac_n(X) \ge 1$ as we explain now.
    A classical result from~\cite{Hedlund-Morse_symbolic-dynamics_1938} (see~\cite[§4.3]{Cassaigne-Nicolas_Factor-complexity_2010}) states that if there exists $n\in \N$ such that $\fac_n(\Xx)\le n$ then $\Xx$ must be periodic.
    Hence the least possible complexity function for an aperiodic sequence is $n\mapsto n+1$ (this characterises Sturmian sequences, which will appear in Section~\ref{sec:ModTorus}).
\end{remark}

\subsubsection*{Transcendence of continued fractions with low complexity}

Let us recall two transcendence criteria from~\cite{Bugeaud_Automatic-continued-fractions-transcendental-quadratic_2013, Bugeaud_Transcendence-stammering-continued-fractions_2013, Bugeaud_contfrac-low-complexity-transcendence-measures_2012, Adamczewski-Bugeaud_transcendence-measures-contfrac_2010}.

Recall that $\PSL_2(\N)$ is the free monoïd of words over the letters $\{R,L\}$, and that the transposition $A\mapsto A^{\dag}$ is the anti-involution that reverses the word while exchanging the letters $R,L$.
For $M\in \N_{\ge 1}\cup\{\infty\}$, define the set of \demph{syllables of height $\le M$} as $\Sigma_M = \{R^{c_0}L^{c_1}\colon 1\le c_0, c_1 \le M\}$: it is a $\dag$-invariant subset of $\PSL_2(\N)$ that freely generates a submonoïd $\Sigma_M^\star \subset R\PSL_2(\N)L$, on which transposition restricts to an anti-involution.

A sequence $X \in \Sigma_M^\N$ has an attractive fixed point $\xi^+ = X\cdot \infty \in \R_{>1}$, whose continued fraction expansion $\xi^+ =\Ecf{x}$ has continuants the exponents of the $\{R,L\}$-conversion of $X$; in particular the sequence $x\in (\N_{\ge 1})^\N$ is bounded by $M$.

There are several results relating the symbolic complexity of the sequence $X$ to the measures of transcendence of the number $\xi^+$. Indeed, the patterns in the continued fraction expansion $x$ encode the Diophantine approximation properties of $\xi^+$ by quadratic numbers.

If $\liminf_n \tfrac{1}{n}\fac_n(X) > 0$ or $\Dio(X)<\infty$ then $X$ is aperiodic hence $\xi^+$ is not quadratic.
Conversely if $X$ is aperiodic then we have $\liminf_n \tfrac{1}{n} \fac_n(X) \ge 1$ by Remark~\ref{rem:aperiodic-complexity-gap}.

It follows from the boundedness of $x$ that $w_2(\xi^+)\ge \Dio(X)$, in particular if $\Dio(X)>2$ then $\xi^+$ is quadratic of transcendent. 
In fact, it suffices to have $\Dio(X)>1$ to deduce transcendence, as we now recall from~\cite[Theorem 1.1]{Bugeaud_Automatic-continued-fractions-transcendental-quadratic_2013}.

\begin{theorem}[linear $\liminf$ complexity]
    \label{lem:low-complexity-stammering}
    Let $X\in \Sigma_M^\N$ and $\xi^+ = X\cdot \infty \in \R_{>1}\setminus \Q$.
    
    If $\liminf_n \tfrac{1}{n}\fac_n(X) \in [0,+\infty)$, then $\Dio(X)\in (1,+\infty]$, hence $\xi^+$ is quadratic or transcendent.
\end{theorem}
\begin{proof}[Outline of the proof]
    If $\liminf_n \tfrac{1}{n}\fac_n(X) <\infty$ then~\cite[§4]{Bugeaud_Automatic-continued-fractions-transcendental-quadratic_2013} deduces from the Schubfachprinzip that $X$ has long repetitions ($\Dio(X)>1$).
    A fantastic argument involving (repeated applications of) the Schmidt subspace theorem shows that $\xi^+$ is quadratic or transcendent.
\end{proof}

We will be interested in sequences $X$ whose factor complexity is bounded by an affine function, and for those~\cite[Theorem 2.2]{Bugeaud_contfrac-low-complexity-transcendence-measures_2012} employs the quantitative Schmidt subspace theorem to bound the Mahler measures of $\xi^+$ in terms of the symbolics of $X$.

\begin{theorem}[linear $\limsup$ complexity]
    \label{thm:linear-limsup-complexity}
    Let $X\in \Sigma_\infty^\N$ and $\xi^+ = X\cdot \infty \in \R_{>1}$.
    
    Assume that $\lim\sup \tfrac{1}{n}\fac_n(X)<\infty$. %  there that there is $\kappa \in \N_{\ge 1}$ and $n_0\in \N$ such that for $n\ge n_0$ we have $\fac_n(X) \le \kappa n$.

    If $\Dio(X)=\infty$ then $\xi^+$ is quadratic or transcendental of Mahler class $\mathrm{U}_2$ (namely $w_2(\xi^+)=\infty$).
    
    If $\Dio(X)<\infty$ then $\xi^+$ is transcendental of Mahler class $\mathrm{S}$ or $\mathrm{T}$, and more precisely there exists $c\in \R_{>0}$ such that for all $d\in \N_{\ge 1}$ we have $w_d(\xi^+)\le \exp\left(c (\log 3d)^5 (\log \log 3d)^4\right)$.
    % \[\exists c \in \R_{>0} \quad \forall d\in \N_{\ge 1} \colon \quad w_d(\xi^+)\le \exp\left(c (\log 3d)^5 (\log \log 3d)^4\right)\]
    % and we also have \[\max\{2,\Dio(\xi^+)-1\} \le w_2(\xi^+) \le 118000 \kappa^3 \Dio(\xi^+) (\log(M+1))^4.\]
\end{theorem}

A special class of sequences with sub-affine factor complexity are the purely morphic sequences that are uniformly recurrent.
We may now recall~\cite[Theorem 2.2.2]{Adamczewski-Bugeaud_transcendence-measures-contfrac_2010}

\begin{theorem}[morphic continued fractions]
    \label{thm:morphic-contfrac}
    Let $X\in \Sigma^\N$ and $\xi^+=X\cdot \infty\in \R_{> 1}\setminus \Q$.
    
    If $X\in \Sigma_M^\N$ is uniformly recurrent and morphic, then either $\xi^+$ is quadratic or else it is transcendental of Mahler class $\mathrm{S}$ or $\mathrm{T}$, and more precisely there exists $c\in \R_{>0}$ such that for all $d\in \N_{\ge 1}$ we have $w_d(\xi^+)\le \exp\left(c (\log 3d)^3 (\log \log 3d)^2\right)$.
\end{theorem}

%\bibliographystyle{alpha} %apalike
%\bibliography{biblio}

%% file: images/tikz/PSL2Z-pavage-PH.tex
% première figure
\begin{tikzpicture}[line cap=round,line join=round,>=triangle 45,x=4.2cm,y=4.2cm]
\clip(-1.3,-1.3) rectangle (1.3,1.3);
\fill[line width=2.pt,color=marron,fill=grey,fill opacity=0.2] (-0.5,0.) -- (0.,0.) -- (0.,-1.) -- cycle;
\fill[line width=2.pt,color=marron,fill=grey,fill opacity=0.2] (0.,1.) -- (0.,0.) -- (0.5,0.) -- cycle;
\fill[line width=2.pt,color=marron,fill=grey,fill opacity=0.2] (0.,1.) -- (-0.5,0.) -- (-0.6666666666666666,0.3333333333333333) -- cycle;
\fill[line width=2.pt,color=marron,fill=grey,fill opacity=0.2] (0.5,0.) -- (0.,-1.) -- (0.6666666666666666,-0.3333333333333333) -- cycle;
\fill[line width=0.pt,color=marron,fill=grey,fill opacity=0.2] (0.7272727272727274,0.4545454545454545) -- (0.5714285714285714,0.7142857142857144) -- (0.8,0.6) -- cycle;
\fill[line width=0.pt,color=marron,fill=grey,fill opacity=0.2] (0.7272727272727274,0.4545454545454545) -- (1.,0.) -- (0.888888888888889,0.3333333333333333) -- cycle;
\fill[line width=0.pt,color=marron,fill=grey,fill opacity=0.2] (0.6666666666666666,-0.3333333333333333) -- (0.7272727272727274,-0.4545454545454545) -- (1.,0.) -- cycle;
\fill[line width=0.pt,color=marron,fill=grey,fill opacity=0.2] (0.7272727272727274,-0.4545454545454545) -- (0.5714285714285714,-0.7142857142857144) -- (0.,-1.) -- cycle;
\fill[line width=0.pt,color=marron,fill=grey,fill opacity=0.2] (0.8,-0.6) -- (0.7272727272727274,-0.4545454545454545) -- (0.888888888888889,-0.3333333333333333) -- cycle;
\fill[line width=0.pt,color=marron,fill=grey,fill opacity=0.2] (-0.663823553438191,-0.336176446561809) -- (-1.,0.) -- (-0.5,0.) -- cycle;
\fill[line width=0.pt,color=marron,fill=grey,fill opacity=0.2] (-0.663823553438191,-0.336176446561809) -- (-0.7272727272727274,-0.4545454545454546) -- (0.,-1.) -- cycle;
\fill[line width=0.pt,color=marron,fill=grey,fill opacity=0.2] (-0.7272727272727274,-0.4545454545454546) -- (-0.8,-0.6) -- (-0.5714285714285714,-0.7142857142857144) -- cycle;
\fill[line width=0.pt,color=marron,fill=grey,fill opacity=0.2] (-0.7272727272727274,-0.4545454545454546) -- (-0.888888888888889,-0.33333333333333337) -- (-1.,0.) -- cycle;
\fill[line width=0.pt,color=marron,fill=grey,fill opacity=0.2] (-0.7272727272727272,0.4545454545454543) -- (-0.6666666666666666,0.3333333333333333) -- (-1.,0.) -- cycle;
\fill[line width=0.pt,color=marron,fill=grey,fill opacity=0.2] (-0.7272727272727272,0.4545454545454543) -- (-0.8,0.6) -- (-0.8888888888888888,0.3333333333333333) -- cycle;
\fill[line width=0.pt,color=marron,fill=grey,fill opacity=0.2] (-0.7272727272727272,0.4545454545454543) -- (-0.5714285714285714,0.7142857142857144) -- (0.,1.) -- cycle;
\fill[line width=0.pt,color=marron,fill=grey,fill opacity=0.2] (-0.008274834626541288,0.9999657629698646) -- (0.7272727272727274,0.4545454545454545) -- (0.6666666666666667,0.33333333333333326) -- cycle;
\fill[line width=0.pt,color=marron,fill=grey,fill opacity=0.2] (0.6666666666666667,0.33333333333333326) -- (0.5,0.) -- (1.,0.) -- cycle;

\draw [line width=2.pt] (0.,0.) circle (4.166666666666667cm);

\draw [line width=2.pt,color=green] (-0.7272727272727272,0.4545454545454543)-- (-1.,0.);
\draw [line width=2.pt,color=green] (-0.7272727272727272,0.4545454545454543)-- (0.,1.);
\draw [line width=2.pt,color=green] (-0.7272727272727274,-0.4545454545454546)-- (-1.,0.);
\draw [line width=2.pt,color=green] (-0.7272727272727274,-0.4545454545454546)-- (0.,-1.);
\draw [line width=2.pt,color=green] (1.,0.)-- (0.,1.);
\draw [line width=2.pt,color=green] (0.,1.)-- (-0.5,0.);
\draw [line width=2.pt,color=green] (-0.5,0.)-- (0.,-1.);
\draw [line width=2.pt,color=green] (0.,1.)-- (0.5,0.);
\draw [line width=2.pt,color=green] (0.5,0.)-- (0.,-1.);
\draw [line width=2.pt,color=green] (0.,1.)-- (-1.,0.);
\draw [line width=2.pt,color=green] (1.,0.)-- (0.,-1.);
\draw [line width=2.pt,color=green] (0.7272727272727274,0.4545454545454545)-- (0.8,0.6);
\draw [line width=2.pt,color=green] (0.7272727272727274,-0.4545454545454545)-- (0.8,-0.6);
\draw [line width=2.pt,color=green] (-0.7272727272727274,-0.4545454545454546)-- (-0.8,-0.6);
\draw [line width=2.pt,color=green] (-0.7272727272727272,0.4545454545454543)-- (-0.8,0.6);
\draw [line width=2.pt,color=green] (0.,1.)-- (0.7272727272727274,0.4545454545454545);
\draw [line width=2.pt,color=green] (0.7272727272727274,0.4545454545454545)-- (1.,0.);
\draw [line width=2.pt,color=green] (0.7272727272727274,-0.4545454545454545)-- (0.,-1.);
\draw [line width=2.pt,color=green] (1.,0.)-- (0.7272727272727274,-0.4545454545454545);
\draw [line width=2.pt,color=green] (-0.5,0.)-- (-1.,0.);
\draw [line width=2.pt,color=green] (0.5,0.)-- (1.,0.);

\draw [line width=2.5pt,color=forestgreen] (0.,1.)-- (1.,0.);
\draw [line width=2.5pt,color=forestgreen] (-1.,0.)-- (-0.8,-0.6);
\draw [line width=2.5pt,color=forestgreen] (-0.8,0.6)-- (-1.,0.);
\draw [line width=2.5pt,color=forestgreen] (0.,1.)-- (0.,-1.);
\draw [line width=2.5pt,color=forestgreen] (-1.,0.)-- (0.,1.);
\draw [line width=2.5pt,color=forestgreen] (-1.,0.)-- (0.,-1.);
\draw [line width=2.5pt,color=forestgreen] (0.,-1.)-- (1.,0.);
\draw [line width=2.5pt,color=forestgreen] (-0.8,0.6)-- (0.,1.);
\draw [line width=2.5pt,color=forestgreen] (0.,1.)-- (0.8,0.6);
\draw [line width=2.5pt,color=forestgreen] (0.8,0.6)-- (1.,0.);
\draw [line width=2.5pt,color=forestgreen] (1.,0.)-- (0.8,-0.6);
\draw [line width=2.5pt,color=forestgreen] (0.8,-0.6)-- (0.,-1.);
\draw [line width=2.5pt,color=forestgreen] (0.,-1.)-- (-0.8,-0.6);
\draw [line width=2.5pt,color=forestgreen] (-0.8,-0.6)-- (-1.,0.);
\draw [line width=2.5pt,color=forestgreen] (-1.,0.)-- (-0.8,0.6);

\draw [line width=2.pt,color=marron] (0.5,0.)-- (0.7272727272727274,0.4545454545454545);
\draw [line width=2.pt,color=marron] (0.5,0.)-- (0.7272727272727274,-0.4545454545454545);
\draw [line width=2.pt,color=marron] (-0.5,0.)-- (-0.7272727272727274,-0.4545454545454546);
\draw [line width=2.pt,color=marron] (-0.5,0.)-- (-0.7272727272727272,0.4545454545454543);
\draw [line width=2.pt,color=marron] (0.5,0.)-- (0.7272727272727274,-0.4545454545454545);
\draw [line width=2.pt,color=marron] (0.7272727272727274,0.4545454545454545)-- (0.888888888888889,0.3333333333333333);
\draw [line width=2.pt,color=marron] (0.7272727272727274,0.4545454545454545)-- (0.5714285714285714,0.7142857142857144);
\draw [line width=2.pt,color=marron] (0.7272727272727274,-0.4545454545454545)-- (0.5714285714285714,-0.7142857142857144);
\draw [line width=2.pt,color=marron] (0.7272727272727274,-0.4545454545454545)-- (0.888888888888889,-0.3333333333333333);
\draw [line width=2.pt,color=marron] (0.,0.)-- (0.5,0.);
\draw [line width=2.pt,color=marron] (-0.5,0.)-- (0.,0.);
\draw [line width=2.pt,color=marron] (-0.7272727272727272,0.4545454545454543)-- (-0.8888888888888888,0.3333333333333333);
\draw [line width=2.pt,color=marron] (-0.7272727272727274,-0.4545454545454546)-- (-0.5714285714285714,-0.7142857142857144);
\draw [line width=2.pt,color=marron] (-0.7272727272727274,-0.4545454545454546)-- (-0.888888888888889,-0.33333333333333337);
\draw [line width=2.pt,color=marron] (-0.5714285714285714,0.7142857142857143)-- (-0.7272727272727272,0.4545454545454543);

%Arcs de cercles ---------
%S
\draw[line width=1.5pt,-{Stealth[length=2mm,width=1.5mm]},color=blue] ($(0.,0.03)+({0.1*cos(180)},{0.1*sin(180)})$) arc (180:0:0.1);
\draw[line width=1.5pt,{Stealth[length=2mm,width=1.5mm]}-,color=blue] ($(0.,-0.03)+({0.1*cos(-180)},{0.1*sin(-180)})$) arc (-180:0:0.1);
\draw[color=blue] (-0.15,0.2) node {\huge$S$};

%T
\draw[line width=1.5pt,-{Stealth[length=2mm,width=1.5mm]},color=blue] ($(0.53,0.)+({0.1*cos(60)},{0.1*sin(60)})$) arc (60:-60:0.1);
\draw[line width=1.5pt,-{Stealth[length=2mm,width=1.5mm]},color=blue] ($(0.47,0.03)+({0.1*cos(180)},{0.1*sin(180)})$) arc (180:60:0.1);
\draw[line width=1.5pt,-{Stealth[length=2mm,width=1.5mm]},color=blue] ($(0.47,-0.03)+({0.1*cos(300)},{0.1*sin(300)})$) arc (300:180:0.1);
\draw[color=blue] (0.5,0.3) node {\huge$T$};

%arc de cercle magenta matrice R
\draw[line width=1.5pt,-{Stealth[length=2mm,width=1.5mm]},color=magenta] ($(0.,1.)+({0.3*cos(-85)},{0.3*sin(-85)})$) arc (-85:-50:0.3);
\draw[color=magenta] (-0.1,0.65) node {\huge$R$};

%flèche rose L
\draw[line width=1.5pt,-{Stealth[length=2mm,width=1.5mm]},color=magenta] ($(0.,-1.)+({0.3*cos(85)},{0.3*sin(85)})$) arc (85:50:0.3);
\draw[color=magenta] (-0.1,-0.65) node {\huge$L$};

\begin{scriptsize}
\draw [fill=marron] (0.6666666666666667,0.33333333333333326) ++(-3.0pt,0 pt) -- ++(3.0pt,3.0pt)--++(3.0pt,-3.0pt)--++(-3.0pt,-3.0pt)--++(-3.0pt,3.0pt);
\draw [fill=marron] (0.5714285714285714,0.7142857142857143) ++(-3.0pt,0 pt) -- ++(3.0pt,3.0pt)--++(3.0pt,-3.0pt)--++(-3.0pt,-3.0pt)--++(-3.0pt,3.0pt);
\draw [fill=marron] (0.8888888888888888,0.33333333333333326) ++(-3.0pt,0 pt) -- ++(3.0pt,3.0pt)--++(3.0pt,-3.0pt)--++(-3.0pt,-3.0pt)--++(-3.0pt,3.0pt);
\draw [fill=marron,shift={(0.5,0.)},rotate=270] (0,0) ++(0 pt,4.5pt) -- ++(3.8971143170299736pt,-6.75pt)--++(-7.794228634059947pt,0 pt) -- ++(3.8971143170299736pt,6.75pt);
\draw [fill=marron,shift={(0.7272727272727274,0.4545454545454545)},rotate=90] (0,0) ++(0 pt,4.5pt) -- ++(3.8971143170299736pt,-6.75pt)--++(-7.794228634059947pt,0 pt) -- ++(3.8971143170299736pt,6.75pt);
\draw [fill=marron,shift={(0.7272727272727274,-0.4545454545454545)},rotate=90] (0,0) ++(0 pt,4.5pt) -- ++(3.8971143170299736pt,-6.75pt)--++(-7.794228634059947pt,0 pt) -- ++(3.8971143170299736pt,6.75pt);
\draw [fill=marron] (0.5714285714285714,-0.7142857142857143) ++(-3.0pt,0 pt) -- ++(3.0pt,3.0pt)--++(3.0pt,-3.0pt)--++(-3.0pt,-3.0pt)--++(-3.0pt,3.0pt);
\draw [fill=marron] (0.8888888888888888,-0.33333333333333326) ++(-3.0pt,0 pt) -- ++(3.0pt,3.0pt)--++(3.0pt,-3.0pt)--++(-3.0pt,-3.0pt)--++(-3.0pt,3.0pt);
\draw [fill=marron] (0.6667552747467308,-0.3332447252532692) ++(-3.0pt,0 pt) -- ++(3.0pt,3.0pt)--++(3.0pt,-3.0pt)--++(-3.0pt,-3.0pt)--++(-3.0pt,3.0pt);
\draw [fill=marron] (-0.6668227087071634,-0.3331772912928366) ++(-3.0pt,0 pt) -- ++(3.0pt,3.0pt)--++(3.0pt,-3.0pt)--++(-3.0pt,-3.0pt)--++(-3.0pt,3.0pt);
\draw [fill=marron] (-0.5714285714285716,-0.7142857142857143) ++(-3.0pt,0 pt) -- ++(3.0pt,3.0pt)--++(3.0pt,-3.0pt)--++(-3.0pt,-3.0pt)--++(-3.0pt,3.0pt);
\draw [fill=marron] (-0.888888888888889,-0.3333333333333334) ++(-3.0pt,0 pt) -- ++(3.0pt,3.0pt)--++(3.0pt,-3.0pt)--++(-3.0pt,-3.0pt)--++(-3.0pt,3.0pt);
\draw [fill=marron] (-0.888888888888889,0.3333333333333332) ++(-3.0pt,0 pt) -- ++(3.0pt,3.0pt)--++(3.0pt,-3.0pt)--++(-3.0pt,-3.0pt)--++(-3.0pt,3.0pt);
\draw [fill=marron] (-0.5714285714285714,0.7142857142857143) ++(-3.0pt,0 pt) -- ++(3.0pt,3.0pt)--++(3.0pt,-3.0pt)--++(-3.0pt,-3.0pt)--++(-3.0pt,3.0pt);
\draw [fill=marron] (-0.6666666666666666,0.33333333333333337) ++(-3.0pt,0 pt) -- ++(3.0pt,3.0pt)--++(3.0pt,-3.0pt)--++(-3.0pt,-3.0pt)--++(-3.0pt,3.0pt);
\draw [fill=marron,shift={(-0.7272727272727272,0.4545454545454543)},rotate=270] (0,0) ++(0 pt,4.5pt) -- ++(3.8971143170299736pt,-6.75pt)--++(-7.794228634059947pt,0 pt) -- ++(3.8971143170299736pt,6.75pt);
\draw [fill=marron,shift={(-0.7272727272727274,-0.4545454545454546)},rotate=270] (0,0) ++(0 pt,4.5pt) -- ++(3.8971143170299736pt,-6.75pt)--++(-7.794228634059947pt,0 pt) -- ++(3.8971143170299736pt,6.75pt);
\draw [fill=marron,shift={(-0.5,0.)},rotate=90] (0,0) ++(0 pt,4.5pt) -- ++(3.8971143170299736pt,-6.75pt)--++(-7.794228634059947pt,0 pt) -- ++(3.8971143170299736pt,6.75pt);
\draw [fill=marron,rotate=45] (0.,0.) ++(-4.pt,0 pt) -- ++(4.pt,4.pt)--++(4.pt,-4.pt)--++(-4.pt,-4.pt)--++(-4.pt,4.pt);
\draw [fill=black] (-0.8,0.6) circle (2.5pt);
\draw[color=black] (-1.,0.7) node {\huge$-\frac{2}{1}$};
\draw [fill=black] (0.,1.) circle (2.5pt);
\draw[color=black] (0.,1.1) node {\Large$1/0$};
\draw [fill=black] (1.,0.) circle (2.5pt);
\draw[color=black] (1.1,0.) node {\huge$\frac{1}{1}$};
\draw [fill=black] (0.,-1.) circle (2.5pt);
\draw[color=black] (0.,-1.1) node {\Large$0/1$};
\draw [fill=black] (-1.,0.) circle (2.5pt);
\draw[color=black] (-1.2,0.) node {\huge$-\frac{1}{1}$};
\draw [fill=black] (0.8,0.6) circle (2.5pt);
\draw[color=black] (1.,0.7) node {\huge$\frac{2}{1}$};
\draw [fill=black] (-0.8,-0.6) circle (2.5pt);
\draw[color=black] (-1.,-0.7) node {\huge$-\frac{1}{2}$};
\draw [fill=black] (0.8,-0.6) circle (2.5pt);
\draw[color=black] (1.,-0.7) node {\huge$\frac{1}{2}$};
\end{scriptsize}
\end{tikzpicture}
%\hspace{1cm}
% deuxième figure
\begin{tikzpicture}[line cap=round,line join=round,>=triangle 45,x=4.0cm,y=4.0cm]
\clip(-0.2,-1.3) rectangle (1.,1.3);
%\clip(0.-0.1,1.+0.2) -- (0.5+0.2,0.) -- (0.-0.1,-1.-0.2) -- cycle ;
%\clip(0.-0.1,1+0.25) -- (0.+0.15,1+0.25) -- (0.5+0.1,0.) -- (0.+0.15,-1.-0.25) -- (0.-0.1,-1.-0.25) -- cycle ;
%\draw[line width=3pt](0.-0.1,1+0.25) -- (0.+0.15,1+0.25) -- (0.5+0.1,0.) -- (0.+0.15,-1.-0.25) -- (0.-0.1,-1.-0.25) -- cycle ;

%triangles fill
\fill[line width=0.pt,color=marron,fill=grey,fill opacity=0.2] (0.,1.) -- (0.,0.) -- (0.5,0.) -- cycle;

%triangle forestgreen
%\draw [line width=2.pt,color=forestgreen] (1.,0.)-- (0.,-1.);
%\draw [line width=2.pt,color=forestgreen] (0.,1.)-- (1.,0.);
\draw [line width=2.pt,color=forestgreen] (0.,1.)-- (0.,-1.);

%traits marron
\draw [line width=2.pt,color=marron] (0.,0.)-- (0.5,0.);
%\draw [line width=2.pt,color=marron] (0.4970796068201162,0.)-- (0.6666666666666666,0.3333333333333333);
%\draw [line width=2.pt,color=marron] (0.4970796068201162,0.)-- (0.6666666666666666,-0.3333333333333333);

%traits green
\draw [line width=2.pt,color=green] (0.4970796068201162,0.)-- (0.,-1.);
\draw [line width=2.pt,color=green] (0.4970796068201162,0.)-- (0.,1.);
%\draw [line width=2.pt,color=green] (0.4970796068201162,0.)-- (1.,0.);

\begin{scriptsize}

%points v_
\draw [fill=black] (0.,-1.) circle (2.0pt);
\draw[color=black] (0.,-1.1) node {\Large$0/1$};
\draw [fill=black] (0.,1.) circle (2.0pt);
\draw[color=black] (0.,1.1) node {\Large$1/0$};

%carré marron
\draw [fill=marron,rotate=45] (0.,0.) ++(-3.pt,0 pt) -- ++(3.pt,3.pt)--++(3.pt,-3.pt)--++(-3.pt,-3.pt)--++(-3.pt,3.pt);

%triangle marron
\draw [fill=marron, rotate around={30:(0.5,0)}] (0.5,0.) ++(0 pt,3.75pt) -- ++(3.2475952641916446pt,-5.625pt)--++(-6.495190528383289pt,0 pt) -- ++(3.2475952641916446pt,5.625pt);
\end{scriptsize}
\end{tikzpicture}
%\hspace{1cm}
% troisième figure
\begin{tikzpicture}[line cap=round,line join=round,>=triangle 45,x=4.0cm,y=4.0cm]

\clip(-0.1,1.1) -- (0.15,1.1) -- (0.6,0.) --(0.6,-1-0.25) -- (0.-0.1,-1-0.25) -- cycle ;
%\draw[] (0.-0.1,1+0.05) -- (0.+0.15,1+0.05) -- (0.5+0.1,0.) --(0.6,-0.25) -- (0.-0.1,-0.25) -- cycle ;
%\clip(0.-0.1,1+0.25) -- (0.+0.15,1+0.25) -- (0.5+0.1,0.) -- (0.+0.15,-1.-0.25) -- (0.-0.1,-1.-0.25) -- cycle ;

%triangles fill
\fill[line width=0.pt,color=marron,fill=grey,fill opacity=0.2] (0.038+10pt,0.84) -- (0.038-0pt,0.84)-- (0.,0.) -- (0.5,0.) -- cycle;

%traits forestgreen
\draw [line width=2.pt,color=forestgreen] (0.038-0.pt,0.84)-- (0.,0.);

%traits marron
\draw [line width=2.pt,color=marron] (0.,0.)-- (0.5,0.);

%traits green
\draw [line width=2.pt,color=green] (0.4970796068201162,0.)-- (0.038+10pt,0.84);

\begin{scriptsize}
%points forestgreen
\draw[line width=1.5pt,color=black, fill=white] (0.043,0.84) ellipse (5pt and 2.5pt);
%\draw [fill=black] (0.,1.) circle (2.0pt);
\draw[color=black] (0.,1.) node {\Large$\infty$};

%carrés marron
\draw [fill=marron,rotate=45] (0.,0.) ++(-3.pt,0 pt) -- ++(3.pt,3.pt)--++(3.pt,-3.pt)--++(-3.pt,-3.pt)--++(-3.pt,3.pt);
\draw[color=black,anchor=north] (0.,-0.05) node {\huge$i$};

%triangle marron
\draw [fill=marron, rotate around={30:(0.5,0)}] (0.5,0.) ++(0 pt,3.75pt) -- ++(3.2475952641916446pt,-5.625pt)--++(-6.495190528383289pt,0 pt) -- ++(3.2475952641916446pt,5.625pt);
\draw[color=black,anchor=north] (0.5,-0.05) node {\huge$j$};

\end{scriptsize}
\end{tikzpicture}

%% file: images/tikz/cusp-excursion.tex
\begin{tikzpicture}
    % Define clipping area
    \clip (-1,-0.7) rectangle (7,4);
    
    % Ejes
    \draw[->] (-0.5,0) -- (6.5,0) node[above] {$\Re(z)$};
    \draw[->] (0,-0.5) -- (0,3.5) node[right] {$\Im(z)>0$};
    
    % Lineas horizontales
    \draw[dashed] (-0.,2) node[left] {$\frac{x_0}{2}$} -- (6.5,2);
    % \draw[dashed] (-0.,3) -- (6.5,3);
    \draw[dashed] (-0.,3) node[left] {$\frac{1+x_0}{2}$} -- (6.5,3);
    
    % Arcos
    \draw[line width=0.7mm, forestgreen] (0,0) arc[start angle=180,end angle=0,radius=2];
    \draw[line width=0.7mm, forestgreen] (0,0) arc[start angle=180,end angle=0,radius=3];
    \draw[line width=1mm, red] (0,0) arc[start angle=180,end angle=0,radius=2.34];
    
    % Etiquetas en el eje x
    \node[below left] at (0,0) {$0$};
    \node[below] at (6,0) {\textcolor{forestgreen}{$1+x_0$}};
    \node[below] at (4,0) {\textcolor{forestgreen}{$x_0$}};
    \node[below] at (4.68,0) {\textcolor{red}{$\xi^+$}};
    
    % Sombreado
    \fill[pattern=north east lines, pattern color=gray] (0,2) rectangle (6.5,3);
\end{tikzpicture}

%% file: sec2-TranSimple.tex
\section{Transcendence of simple geodesics on modular covers}
\label{sec:TranSimple}

Subsections~\ref{subsec:topo-modular-cover} and~\ref{subsec:symbolic-simple-geodesics} introduce the background to prove the main Theorems~\ref{thm:simple-geodesic-trichtomy} and~\ref{thm:transcend-morphic-geodesics} in Subsection~\ref{subsec:transcend-simple-geodesics}.
Subsection~\ref{subsec:profinite-simple} defines profinite simple numbers and raises several questions. 

\subsection{Topology and combinatorics of finite modular covers}
% Topology of finit modular covers and combinatorics of their geodesics
\label{subsec:topo-modular-cover}

Consider a finite-index subgroup $\Gamma^\prime \subset \Gamma$, corresponding to a finite regular cover $\Gamma^\prime\subset \Gamma$.
Assume that $\Gamma^\prime$ is torsion free, so that $\M^\prime$ is a surface with no conical singularities: it has genus $g\in \N$ and $p\in \N_{>0}$ cusps, hence Euler characteristic $\chi=2-2g-p\in \Z_{<0}$.
The surface $\M^\prime$ has an ideal triangulation $\triangle^\prime= \triangle \bmod{\Gamma^\prime}$ which is dual to the embedded trivalent graph $\TT^\prime = \TT \bmod{\Gamma^\prime}$.
The surface $\M^\prime$ retracts by deformation onto $\TT^\prime$, hence $\Gamma^\prime=\pi_1(\TT^\prime,i^\prime)$ is a free group of rank $2g+p$.
\begin{comment}
(It is an exercise to construct a free basis of $\Gamma^\prime$ whose elements yield simple loops on $\M^\prime$.) 
\end{comment}

The cusps of $\M^\prime$ correspond to the orbits of rational points $\Q\Proj^1 \bmod{\Gamma^\prime}$, hence to the conjugacy classes of parabolic elements in $\Gamma^\prime$.
A cusp $o \in \Q\Proj^1 \bmod{\Gamma^\prime}$ has a \demph{width} defined as the index of stabilisers $w=[\Stab(o,\Gamma)\colon \Stab(o, \Gamma^\prime)]$, which can be characterised as follows: geometrically, the cusp neighbourhood $\Ball(o,h)$ of $o$ at height $h>0$ has area $w/h$; combinatorially, the smallest cycle of $\TT^\prime$ surrounding $o$ has $w$ edges (when $\Gamma^\prime$ is normal in $\Gamma$ all cusps have the same width).
Denote by $\mu^\prime$ the \demph{maximal cusp-width} of $\Gamma^\prime$ (when $\Gamma^\prime$ is normal in $\Gamma$ all cusps have the same width).

Denote by $\VV(\TT^\prime)$ and $\EE(\TT^\prime)$ the vertices and edges of $\TT^\prime$.
The rank of $\Gamma^\prime=\pi_1(\TT^\prime)$ equals that of the first homology groups $H_1(\TT^\prime;\Z)=H_1(\M^\prime;\Z)$, that is $\Card \EE(\TT^\prime)-\Card \EE(\TT^\prime) = 2g+p$.
Denote by $\Vec{E}(\TT^\prime)$ the set of \demph{arcs} of $\TT^\prime$, which correspond to pairs of incident edges and vertices. 
Since $\Gamma$ acts freely transitively on arcs of $\TT$, we have a bijection $C\in \Gamma/\Gamma^\prime \mapsto \vec{\Ee}_C\in \vec{\EE}(\TT^\prime)$; the \demph{base arc} is $\vec{\Ee}_\Id$ (that is the edge $i^\prime$ directed towards the vertex $j^\prime$).
When $\Gamma^\prime$ is normal in $\Gamma$, the quotient group $\Gamma/\Gamma^\prime$ acts freely transitively on the arcs of $\TT^\prime$.

An arc $\vec{\Ee}\in \Vec{E}(\TT^\prime)$ has \demph{starting} and \demph{terminating} vertices $(\partial^-\vec{\Ee}, \partial^+\vec{\Ee})\in \VV(\TT^\prime)\times \VV(\TT^\prime)$. 
% , for instance $\partial^+\vec{ij} = j$ (where $i,j\in \Tree \subset \HP$).
(Note that the map $\partial^-\times \partial^+ \colon \Vec{E}(\TT^\prime) \to \VV(\TT^\prime)\times \VV(\TT^\prime)$ may be non-injective: this would be true only if the trivalent graph $\TT^\prime$ had no multiple edges, that is when $2\Card \EE(\TT^\prime)=3\Card \VV(\TT^\prime)$.)

\begin{comment}
Define the \demph{loop semigroup} of $\TT^\prime$ as the subset of words $p=p_0\dots p_{l}$ in the tree-path groupoïd starting and terminating with $p_0=\vec{\Ee}_\Id=p_l$.
It is freely generated by the subset $\mathcal{B}$ of tree-paths having exactly two occurrences of $\vec{\Ee}_{\Id}$ (which is infinite as soon as $[\Gamma\colon \Gamma^\prime]>6$).
The tree-path semigroup is in correspondence with the \demph{positive monoïd} $\Gamma^\prime_{\ge0}=\Gamma^\prime\cap \PSL_2(\N)$.
%(Show as an exercise that $\Gamma^\prime_{\ge0}$ contains a free basis of the group $\Gamma^\prime$ whose elements yield simple loops on $\M^\prime$.)
% % (We will soon be considering submonoïds generated by finite subsets $\Al\subset \mathcal{B}$.

Define the \demph{even-positive semigroup} as $\Gamma^\prime \cap R\PSL_2(\N)L$, that is the subsemigroup of $\Gamma^\prime_{\ge0}$ freely generated by the subset $\mathcal{B}_{RL} \subset \mathcal{B}$ of edge-loops starting with $\vec{\Ee}_{\Id}\vec{\Ee}_R$ and terminating with $\vec{\Ee}_{L^{-1}}\vec{\Ee}_{\Id}$.
%
Observe that $S\in \PSL_2(\Z)$ normalises the group $\Gamma^\prime$ if and only if the transposition $\dag$ preserves the even-positive semigroup $\Gamma^\prime \cap R\PSL_2(\N)L$ hence its basis $\mathcal{B}_{RL}$; and note that in terms of tree-paths, the transposition reverses the word while inverting the arcs.
\end{comment}

\begin{example}[normal congruence subgroups]
    For $N\in \N_{\ge 2}$, the morphism $\PSL_2(\Z) \to \PSL_2(\Z/N)$ has kernel  \( \Gamma(N)=\{A \in \Gamma \mid A\equiv \Id \bmod{N}\}\), defining the 
    \demph{level-$N$ normal congruence subgroup} $\Gamma(N) \subset \Gamma$ corresponding to the $\PSL_2(\Z/N)$-Galois cover $\M(N)\to \M$.
    It is torsion-free and has cusp-width $N$.
    % The positive monoïd $\Gamma(N)_{\ge 0}=\{A\in \PSL_2(\N) \colon A\equiv \Id \bmod{N}\}$ is finitely generated if and only if $N=2$ in which case it is freely generated by $L^2, R^2$. 
    %
    Figure~\ref{fig:congruence-cover} depicts the trivalent map $\TT(N)\subset \M(N)$ for the values $N\in \{2,3,4,5\}$ such that $\M(N)$ has genus $0$.
\begin{figure}[h]
    \centering
    \scalebox{1.20}{\subfile{images/tikz/Schlegel_bipartite_2-3-4-5}}
    \caption{The embedded graphs $\TT^\prime\subset \M^\prime$ when $\Gamma^\prime=\Gamma(N)$ for $N\in \{2,3,4,5\}$.}
    \label{fig:congruence-cover}
\end{figure}
\end{example}

\subsection{Symbolic complexity of simple geodesics in finite modular covers}
\label{subsec:symbolic-simple-geodesics}

Define a \demph{combinatorial geodesic} of $\TT^\prime$ as a sequence $\vec{\Xx} \in \vec{\EE}(\TT^\prime)^\Z$ whose consecutive letters $\vec{\Xx}_{k} \vec{\Xx}_{k+1}$ satisfy $\partial^+ \vec{\Xx}_{k} = \partial^- \vec{\Xx}_{k+1}$ but $\vec{\Xx}_{k}^{-1}\ne \vec{\Xx}_{k+1}$: it corresponds to an element in $(\vec{\Xx}_0,X)\in (\Gamma/\Gamma^\prime) \rtimes \{R,L\}^\Z$, encoding the starting arc $\vec{\Xx}_0$ and the sequence $X$ of left and right turns from there.
They admit a \demph{$\Z$-action (or flow) under the $\Shift$}, given by $(\Shift \vec{\Xx})_n=\vec{\Xx}_{n+1}$ or $\Shift(\vec{\Xx}_0, X)=(X_0\cdot \vec{\Xx}_0, (X_{n+1})_n)$.

By forgetting the base arc we obtain a $\Shift$-equivariant map $\vec{\Xx}\in \vec{\EE}(\TT^\prime)^\Z\mapsto X\in \{R,L\}^\Z$ whose fibers have cardinal $\Card \vec{\EE}(\TT^\prime)$, and which is locally given by a non-erasing block-code of length $2$, in the sense that the letter $X_i\in \{R,L\}$ is determined by the length-$2$ bloc $\vec{\Xx}_{[i, i+2)} \in \vec{\EE}(\TT^\prime)^2$.
% Moreover, each letter in $\vec{\Xx}_{[i, i+2)} \in \vec{\EE}(\TT^\prime)^2$
Consequently, the factor complexities of the languages $\La(X)\subset \{R,L\}^\star$ and $\La(\vec{\Xx}) \subset \vec{\EE}(\TT^\prime)^\star$ are related by $\fac_{n}(X)\le \fac_{n+1}(\vec{\Xx}) \le (\Card \vec{\EE}(\TT^\prime))\cdot \fac_{n}(X)$.

The \demph{(complete oriented) geodesics} of $\HP$ identify with the space $\Geo(\HP)=\partial \HP \times \partial \HP\setminus \mathrm{diagonal}$, and those of $\M^\prime$ with its quotient $\Geo(\M^\prime)=\Gamma^\prime\backslash \Geo(\HP)$ by the diagonal $\Gamma^\prime$-action.
The geodesics $\xi\in \Geo(\HP)$ with a given projection $\xi^\prime \in \Geo(\M^\prime)$ form the $\Gamma^\prime$ orbit of $\xi\in \Geo(\HP)$.
Note that its $\Gamma^\prime$-orbit always contains a representative with ends $(\xi^-,\xi^+) \in (-1,0]\times (1,+\infty]$.

As $\M^\prime$ deformation retracts onto $\TT^\prime$, a geodesic $\xi^\prime\subset \M^\prime$ isotopes relative to its ends into $\TT^\prime$, yielding the $\Shift$-orbit of a combinatorial geodesic $\vec{\Xx}$.
When $(\xi^-,\xi^+) \in (-1,0)\times (1,+\infty)$, it is has a $\Shift$-representative with $\vec{\Xx}_{-1}\vec{\Xx}_0\vec{\Xx}_1 =\vec{\Ee}_{L} \vec{\Ee}_{\Id} \vec{\Ee}_{R}$, equivalently $\vec{\Xx}_0=\vec{\Ee}_{\Id}$ and $X_{-1}X_0 = LR$; in hwich case the continued fraction expansions of $-1/\xi^-=\Ecf{x_{-1};x_{-2},\dots}$ and $\xi^+=\Ecf{x_{0};x_{1},x_{2},\dots}$ yield the sequence $x\in (\N_{\ge 1})^\Z$ of exponents of $X\in \{R,L\}^\Z$, namely $X=\cdots L^{x_{-1}} R^{x_0} L^{x_1} \cdots$.

A geodesic $\xi\in \Geo(\M^\prime)$ is called \demph{simple} when there are no $\gamma \in \Gamma^\prime$ such the pairs $\{\xi^-,\xi^+\}$ and $\{\gamma \xi^-, \gamma \xi^+\}$ are disjoint and linked with respect to the cyclic order on $\partial \HP$. The simple geodesics form a closed subset $\GeoS(\M^\prime)\subset \Geo(\M^\prime)$.

\subsubsection*{Simple geodesics have low cusp excursions}
%  bounded continued fraction expansions

Say that a combinatorial geodesic has \demph{low cusp excursions} when it never makes a full turn around a directed cycle of $\TT^\prime$ surrounding a single cusp.
In particular, its associated sequence $X\in \{R,L\}^\Z$ cannot contain $R^{\mu^\prime}$ or $L^{\mu^\prime}$, so provided $X_{-1}X_0=LR$ (which we may ensure after acting by $\Shift^k$ for some integer $k\in [0,\mu^\prime)$) it can be recoded unambiguously over $\Sigma_{\mu^\prime}^\Z$.

The following well known fact extensively used by~\cite{Lehner-Sheingorn_self-intersections-modular-covers_1985}.

\begin{lemma}[simple geodesics have low cusp excursions]
    \label{lem:simple-low-cusp-excursions}
    Consider irrationals $(\xi^-,\xi^+)\in (-1,0)\times (1,\infty)$ with continued fraction expansions $-\xi^-=\Ecf{x_{-1};x_{-2},\dots}$ and $\xi^+=\Ecf{x_{0};x_{1},x_{2},\dots}$.

    Consider a finite index subgroup $\Gamma^\prime\subset \PSL_2(\Z)$ corresponding to a finite cover $\M^\prime\to \M$.

    The geodesic $\xi =(\xi^-,\xi^+) \subset \HP$ projects $\bmod{\Gamma^\prime}$ to a geodesic $\xi^\prime \subset \M^\prime$.

    If $\xi^\prime$ is simple then for all $n\in \N$ we have $x_n<\mu^\prime$ where $\mu^\prime$ is the maximal cusp-width of $\Gamma^\prime$.
\end{lemma}
\begin{proof}
    Suppose that there is an index $m\in \N$ such that $x_m\ge M$, and let $C=R^{x_0}\dots L^{x_m}$ or $C=R^{x_0}\dots R^{x_m}$ depending on whether $m$ is odd or even, so that $C^{-1}\xi = \xi_k=\Ecf{x_k,\dots}$.
    The geodesic $(C^{-1}\xi^-,C^{-1}\xi^+)$ must also have a simple projection on $\M^\prime$.
    However the cusp-width of $\infty$ is some $k\le M$ for which $R^{-k}\in \Gamma^\prime$, so the geodesics $(C^{-1}\xi^-,C^{-1}\xi^+)$ and its translate by $R^{-k}$ belong to the same $\Gamma^\prime$-orbit, but they intersect in $\HP$, so we have a contradiction.
\end{proof}
In particular, Khintchine's conjecture (Remark~\ref{rem:BAD-implies-ratio-quad-transcend}) would imply that the endpoints $\xi^\pm$ of lifts of simple geodesics $\xi^\prime \in \GeoS(\M')$ should be rational, quadratic or transcendental.

\subsubsection*{Simple geodesics have roughly linear factor complexity}

We will need to extract more information from the topological simplicity of $\xi^\prime$ to show that the symbolic factor complexity of $X$ is rougly linear. 
For for this, we must recall some notions about geodesic laminations on $\M^\prime$ and their languages (see~\cite{Haas_Diophantine-approximation-hyperbolic-surfaces_1986, Bonahon_geodesic-laminations_2001} for more).

% To extract more information about the simplicity of $\xi^\prime$, we must recall some notions about geodesic laminations on $\M^\prime$ (referring to~\cite{Haas_Diophantine-approximation-hyperbolic-surfaces_1986, Bonahon_geodesic-laminations_2001} for details).
% (where these notions are more generally defined for any torsion free lattice $\Gamma^\prime$ of $\PSL_2(\R)$).

A \demph{geodesic lamination} of $\M^\prime$ is a compact subset $\Xi^\prime\subset \M^\prime$ which is a disjoint union of complete simple geodesics called its leaves (it identifies with a certain compact of $\GeoS(\M^\prime)$).

Since $\M^\prime$ retracts by deformation onto the embedded graph $\TT^\prime$, every geodesic lamination $\Xi^\prime$ is \demph{carried by the embedded graph} $\TT^\prime$: its leaves $\xi^\prime$ are simultaneously isotopic relative to their ends in a small neighbourhood of $\TT^\prime$, yielding $\Shift$-orbits of combinatorial geodesics $\vec{\Xx} \in \vec{\EE}(\TT^\prime)^\Z$.
Altogether, their subwords define its \demph{lamination language} $\La(\Xi^\prime) = \bigcup_{\xi^\prime} \La(\vec{\Xx})\subset \vec{\EE}(\TT^\prime)^\star$. %, which is factorial and bi-extendable.
Its \demph{factor complexity} is the function $\fac(\La(\Xi^\prime)) \colon \N\to \N$ whose value at $n\in \N$ is $\fac_n(\La(\Xi^\prime))=\Card{(\La(\Xi^\prime) \cap \vec{\EE}(\TT^\prime)^n)}$.

% For example, a simple geodesic $\xi^\prime \subset \M^\prime$ has closure a (non-empty) geodesic lamination $\Xi^\prime \subset \M^\prime$, whose lamination language $\La(\Xi^\prime) \subset \vec{\EE}(\TT^\prime)^\star$ coincides with the set of subwords of $\vec{\Xx}$.

% The geodesic lamination $\Xi^\prime$ is \demph{orientable} when for every edge in $\TT$, at most one of its directions appears in its $1$-language.

%\begin{comment}
A geodesic lamination $\Xi^\prime\subset \M^\prime$ is \demph{minimal} when every one of its leaves $\xi^\prime$ has both its past and future $\xi^\pm \bmod{\Gamma^\prime}$ that are dense in $\Xi^\prime$.
This is equivalent to saying that the subshift of $\vec{\EE}(\TT^\prime)^\Z$ obtained by encoding of its leaves is minimal.
By~\cite[Proposition 3]{Bonahon_geodesic-laminations_2001}: every geodesic lamination is a disjoint union of finitely many minimal sublaminations, and finitely many isolated leaves whose ends spiral along the minimal sublamination or escape into a cusp (see~\cite[Corollary 1.7.3]{Penner-Harer_combin-train-tracks_1992}).

\begin{figure}[h]
    \centering
    \includegraphics[width=0.38\linewidth]{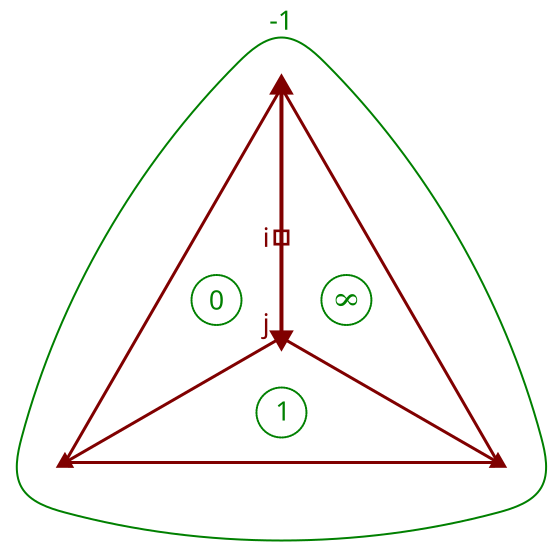}
    \includegraphics[width=0.38\linewidth]{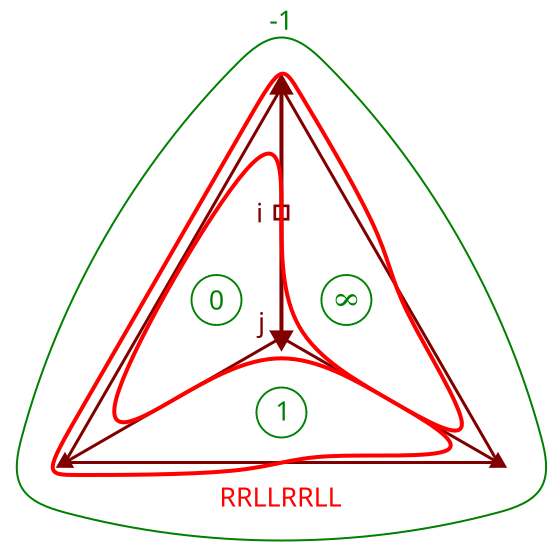}
    \caption{A geodesic (lamination) carried by $\TT(3) \subset \Gamma(3)$ yields a language over $\{R,L\}$.}
    \label{fig:lamination-train-track}
\end{figure}

We may now recall the well known Proposition~\ref{prop:simple-roughly-linear-complexity}, which follows from~\cite{Ferenczi-Zamboni_Languages-k-IET_2008},~\cite{Lopez-Narbel_lamination-languages_2013, Lopez-Narbel_non-orientable-lamination_2015}.

\begin{proposition}[roughly linear complexity]
\label{prop:simple-roughly-linear-complexity}
For a geodesic lamination $\Xi^\prime \subset \M^\prime$, its lamination language $\La(\Xi^\prime)\subset \vec{\EE}(\TT^\prime)^\star$ has roughly-linear factor complexity:
\begin{equation*}
    \exists \kappa \in \N, \; \quad \forall n\in \N, \quad \vert \fac_n(\La(\Xi^\prime)) -  \kappa n \lvert <\infty.
\end{equation*}
% In particular, a leaf $\xi^\prime\in \Xi^\prime$ yields the $\Shift$-orbit of a sequence $X\in \{R,L\}^\Z$ satisfying $\limsup_n \tfrac{1}{n}\fac_n(X) < \infty$.
% with equality if $\xi^\prime$ is dense in $\Xi^\prime$.
% Consequently, any leaf $\xi^\prime$ yields the $\Shift$-orbit of a sequence of left and right turns $X\in \{R,L\}^\Z$ with factor complexity satisfying $\limsup_n \fac_n(X)<\infty$.
\end{proposition}

\begin{proof}
    The proposition follows from much more precise results in the literature, as we now explain.
    The method boils down to computing and bounding the derivative of the factor complexity function $n\mapsto \fac_{n+1}(\La(\Xi^\prime))-\fac_n(\La(\Xi^\prime))$, by enumerating (bi)special factors (see~\cite[Theorem 4.9.3]{Cassaigne_complexite-facteurs-speciaux_1997}).
    
    If $\Xi^\prime$ is orientable (in the sense that for every edge in $\TT^\prime$, at most one of its directions appears in its $1$-language $\La(\Xi^\prime)\cap \vec{\EE}(\TT^\prime)$), then the lamination language $\La(\Xi^\prime) \subset \vec{\EE}(\TT^\prime)^\star$ has ultimate affine complexity by~\cite[Theorem B]{Lopez-Narbel_lamination-languages_2013}. 
    If $\Xi^\prime$ is non-orientable then by~\cite[§5 Remark 1]{Lopez-Narbel_non-orientable-lamination_2015} there is a ramified double cover of $\M^\prime$ in which $\Xi^\prime$ lifts to an oriented lamination carried by the lift of $\TT^\prime$, whose lamination language has ultimate affine complexity: this shows that $\La(\Xi^\prime)$ is the morphic image of a language with ultimate affine complexity, so it also has ultimate affine complexity.
\end{proof}

\subsection{Transcendence of simple geodesics in finite modular covers}
\label{subsec:transcend-simple-geodesics}

We are now ready to prove the first main Theorem~\ref{thm:simple-geodesic-trichtomy} in this work. 

\begin{theorem}[transcendence of simple geodesics in $\M^\prime$]
    \label{thm:simple-geodesic-trichtomy}
    Fix distinct $\xi^-,\xi^+ \in \R\Proj^{1}$, and behold the geodesic $\xi = (\xi^-,\xi^+)\in \Geo(\HP)$.
    Consider a finite index subgroup $\Gamma^\prime \subset \Gamma$, corresponding to a finite cover $\M^\prime\to \M$.
    The geodesic $\xi \in \Geo(\HP)$ projects $\bmod{\Gamma^\prime}$ to a geodesic $\xi^\prime \in \Geo(\M^\prime)$.
    
    If $\xi^\prime$ is simple then $\xi^+\in \R\Proj^{1}$ is either rational or quadratic or transcendental.

    More precisely in the transcendental case, either $w_2(\xi^+)=\infty$ (of type $\mathrm{U}_2$) or else there exists $c\in \R_{>0}$ such that for all $d\in \N_{\ge 1}$ we have $w_d(\xi^+)\le \exp\left(c (\log 3d)^3 (\log \log 3d)^2\right)$ (of class $\mathrm{S}$ or $\mathrm{T}$).
\end{theorem}

\begin{proof}
Since the conclusion only concerns $\xi^+ \in \R\Proj^{1}$, the statement assumes the existence of $\xi^- \in \R\Proj^1\setminus{\xi^+}$ and finite index subgroup $\Gamma^\prime\subset \Gamma$ such that $\xi^\prime\subset \M^\prime$ is simple. 
Thus, we may pass to any finite index subgroup of $\Gamma^\prime$: this will choose a lift of the geodesic $\xi^\prime$ to the corresponding finite cover of $\M^\prime$, which remains simple.
Hence we may assume that $\Gamma^\prime$ is torsion-free. %, normal in $\Gamma$ and normalized by $S$, to fully exploit the discussions from the previous paragraphs leading to Proposition~\ref{prop:simple-roughly-linear-complexity}.
% In that case, the action of $\PSL_2(\Z)$ on $\HP$ descends to the action of the Galois group of $\Gamma/\Gamma^\prime$ by deck transformations of the cover $\M^\prime\to \M$.

Up to the action of $\Gamma^\prime$, we may assume that $(\xi^-,\xi^+)\in (-1,0)\times (1,\infty)$, so the continued fraction expansions $-1/\xi^-=\Ecf{x_{-1};x_{-2},\dots}$ and $\xi^+=\Ecf{x_{0};x_{1},x_{2},\dots}$ define $x\in (\N_{\ge 1})^\Z$.
Hence the projected geodesic $\xi^\prime \subset \M^\prime$ yields a combinatorial geodesic $\vec{\Xx}\in \vec{\EE}(\TT^\prime)^\Z$ with $\vec{\Xx}_{-1}\vec{\Xx}_0\vec{\Xx}_1 =\vec{\Ee}_{L} \vec{\Ee}_{\Id} \vec{\Ee}_{R}$, also encoded by $(\vec{e}_\Id,X)\in (\Gamma/\Gamma^\prime)\rtimes \{R,L\}^\Z$ where $X= \cdots L^{x_{-1}}R^{x_0}\cdots$ has exponents $x\in (\N_{\ge 1})^\Z$.

Note that the action of $\Gamma$ on $\xi^+$, in particular the $\Shift$-action $\xi^+ = X^+\cdot \infty \mapsto (\Shift^{k}X)^+ \cdot \infty$, preserves the (rational, quadratic, transcendental) trichotomy and Mahler exponents.

The closure of $\xi^\prime$ is a geodesic lamination $\Xi^\prime \subset \M^\prime$ whose lamination language $\La(\Xi^\prime) \subset \vec{\EE}(\TT^\prime)^\star$ consists of the set of subwords of $\vec{\Xx}$, in particular they have the same factor complexity, which is roughly linear by Proposition~\ref{prop:simple-roughly-linear-complexity}. Since $\fac_{n}(X)\le \fac_{n+1}(\vec{\Xx}) \le \Card \vec{\EE}(\TT^\prime) \fac_{n}(X)$ the factor complexity of $X\in \{R,L\}^\star$ is also roughly linear.
We arranged $X$ to have $X_{-1}X_0=LR$, so by Lemma~\ref{lem:simple-low-cusp-excursions} it can be unambiguously recoded as a sequence $X^\prime$ over the alphabet of syllables $\Sigma_{\mu^\prime}$, therefore
$\fac_n(X^\prime) \le \fac_{\mu^\prime n}(n)$.

Consequently $X^\prime\in \Sigma_{\mu^\prime}^\Z$ has $\limsup_n \tfrac{1}{n} \fac_n(X^\prime)<\infty$ so we may apply Theorem~\ref{thm:linear-limsup-complexity}.
% : the number $\xi^+$ is quadratic or transcendental and in the latter case it satisfies the announced bounds on Mahler exponents.
\end{proof}

To prove the second main Theorem~\ref{thm:transcend-morphic-geodesics} in this work, let us first recall some notion about pseudo-Anosov mapping classes and their stable laminations (see~\cite{Farb-Margalit_MappingClassGroups_2012, Thurston_geo_dyna_diff-surf_1988, Masur_IET-MF_1982} for details).

The group of orientation preserving mapping classes $\Mod^o(\M^\prime)$ contains the pure mapping class group $\Mod(\M^\prime, \partial \M')=\Homeo^o(\M^\prime, \partial \M')/\Homeo_0(\M^\prime, \partial \M')$ as a finite index normal subgroup: it is the subgroup of $\Out(\Gamma^\prime)$ which fixes all parabolic conjugacy classes of $\Gamma^\prime$ (associated to the so called peripheral loops which surround the cusps), or the kernel of the surjective morphism $\Mod^o(\M^\prime)\to \mathfrak{S}(\partial \M^\prime)$ to the permutation group of the punctures.
The pure mapping class group $\Mod(\M^\prime, \partial \M^\prime)$ is generated by Dehn-twists along simple closed loops (and to generate $\Mod(\M^\prime)$ one must also add half-twists braiding any two punctures, as well as any orientation reversing class).
In any case, every element in $\Mod(\M^\prime)$ has a finite power in $\Mod(\M^\prime, \partial \M^\prime)$.

Every element in $\Mod(\M^\prime, \partial \M^\prime)$ has infinite order, and is represented by an isometry of the hyperbolic surface $\M^\prime$, which is unique modulo an isometry isotopic to the identity.
Hence the group $\Mod(\M^\prime, \partial \M^\prime)$ acts on the space of simple geodesics, and on the space of geodesic laminations.

\begin{example}[pseudo-Anosov classes]
An element $\varphi^\prime \in \Mod(\M^\prime,\partial \M^\prime)$ is \demph{pseudo-Anosov} when no simple closed geodesics are invariant under its action.
A pseudo-Anosov element admits a unique stable fixed geodesic lamination $\Xi^\prime \subset \M^\prime$, it is minimal with no closed leaves (see~\cite{Thurston_geo_dyna_diff-surf_1988, Masur_IET-MF_1982}). % (it is minimal and uniquely ergodic, namely it admits a unique transverse measure~\cite{Masur_IET-MF_1982}).
In such a stable lamination, there may sometimes also be a simple geodesic that is fixed by $\varphi^\prime$.
\end{example}

\begin{theorem}[transcendence of geodesics fixed by isometries]
    \label{thm:transcend-morphic-geodesics}
    If $\xi^{\prime}\in \Geo(\M^\prime)$ is preserved by a pseudo-Anosov mapping class $\varphi^\prime\in \Mod(\M^\prime, \partial \M^\prime)$, then $\xi^\pm$ are transcendental numbers of Mahler class $\mathrm{S}$ or $\mathrm{T}$, more precisely $\exists c\in \R_{>0},\; \forall d\in \N_{\ge 1} \colon w_d(\xi^\pm)\le \exp\left(c (\log 3d)^3 (\log \log 3d)^2\right)$.
\end{theorem}

\begin{proof}
    The closure of $\xi^\prime$ is a minimal lamination $\Xi^\prime$ with no closed leaves that is stable by $\varphi^\prime$, in particular the geodesic $\xi^\prime$ cannot have ends in the cusp and cannot be periodic, hence it yields a combinatorial geodesic $\vec{\Xx} \in \vec{\EE}(\TT^\prime)^\Z$ whose factor complexity satisfies $\fac_n(\vec{x})>n$ and whose language $\La(\vec{\Xx})=\La(\Xi^\prime)\subset \vec{\EE}(\TT^\prime)$ is uniformly recurrent. 
    
    The map $\varphi^\prime$ yields a self-covering map of the graph $\TT^\prime$, so it yields a morphism of monoïds $\Phi \colon \vec{\EE}(\TT^\prime)^\star \to \vec{\EE}(\TT^\prime)^\star$, hence a map $\Phi \colon \vec{\EE}(\TT^\prime)^\Z \to \vec{\EE}(\TT^\prime)^\Z$.
    By assumption, the map $\Phi$ preserves the $\Shift$-orbit of $\vec{\Xx}$, so there is $k\in \Z$ such that $\Phi(\vec{\Xx})=\Shift^k(\vec{\Xx})$.
    This means that $\Phi^{2k}(\vec{\Xx})=\Shift^{2k}(\vec{\Xx})$ hence $\Shift^{k}\vec{\Xx}$ is fixed by the conjugate morphism $\Shift^{k}\Phi^{2k}\Shift^{-k}$.
    Consequently, the sequence $\Shift^{k}X\in \{R,L\}^\Z$ is also fixed a morphism of $\{R,L\}^\star$, and as $\vec{\Xx}$ it is aperiodic and recurrent.

    Note that the numbers $(\Shift^{k}X)^\pm \cdot \infty$ are the translates of $\xi^\pm$ by an element of $\PSL_2(\Z)$, so this $\Shift$-action action preserves Mahler exponents.
    % Moreover, the $\Shift$ preserves diophantine exponents.
    Hence the result follows from Theorem~\ref{thm:morphic-contfrac}.
\end{proof}

\begin{question}[same Mahler class?]
    Under the hypotheses of Theorem~\ref{thm:transcend-morphic-geodesics}, we believe that the transcendental numbers $\xi^\pm$ should belong to the same Mahler class.
    % However, we are unable to show this using the current results in the literature (combining~\cite{Adamczewski-Bugeaud_measures-transcendance-quantitative-Schmidt_2010, Bugeaud_contfrac-low-complexity-transcendence-measures_2012} and~\cite{Berthe-Holton-Zamboni_Sturmian_2006}).
    
    One approach could go as follows: one the one hand, generalize the results in~\cite{Berthe-Holton-Zamboni_Sturmian_2006} to purely morphic interval exchanges and show that their leaves $\Xx\in \Al^\Z$ satisfy $\Dio(\Xx^+)=\Dio(\Xx^-)$ (this motivated our~\cite[mathoverflow question 512076]{CLS_mathoverflow-quest-Dio-IET_2026}); on the other hand refine the Mahler transcendence measures in~\cite{Adamczewski-Bugeaud_measures-transcendance-quantitative-Schmidt_2010, Bugeaud_contfrac-low-complexity-transcendence-measures_2012} using the quantitative Schmidt subspace theorem.
    % \href{https://mathoverflow.net/questions/512076/diophantine-exponents-for-interval-exchanges-and-lamination-languages}{mathoverflow questions 512076}
\end{question}

\begin{remark}[generalisations to infinite products]
There are several steps one may consider to generalize the previous result and bridge the gap from geodesics that are fixed by pseudo-Anosov maps, to leaves of geodesic laminations stable by pseudo-Anosov maps, and then to more general minimal geodesic laminations $\Xi^\prime$ that are stable fixed points by infinite products of mapping classes.
This can be seen as diophantine approximation in the space of measured laminations of $\M^\prime$ (the boundary of the character variety of $\Gamma$), under the action of the pure mapping class group $\Mod(\M^\prime, \partial \M^\prime)$. 

First it is convenient to fix a projective class of measures on the minimal geodesic lamination $\Xi^\prime$ (that is a point in a finite dimensional real convex polytope associated to $\Xi^\prime$). If $\Xi^\prime$ is stable and expanding under the action of an isometry, then this projective measure is unique.
This enables to work in the piecewize linar space of measured geodesic laminations, and makes it easyer to define the notion of stably-expanding fixed point by a mapping class.

Next, every measured minimal geodesic lamination $\Xi^\prime$ fills a certain subsurface $\M^\prime[\Xi^\prime]$, and is a stably-expanding fixed point point by an infinite product of mapping classes that are pseudo-Anosov in restriction $\M^\prime[\Xi^\prime]$. % (since every the latter are dense in the space of measured geodesic laminations).

Such an infinite product would yield an $S$-adic representation by substitutions of the associated symbolic space $X\subset \vec{\EE}(\TT^\prime)$.
However, this product may not be unique, and it is a subtle matter to define a procedure for producing a preferred sequence of (pseudo-Anosov) mapping classes with the desired lamination as fixed point.

On the other hand, one may start with special kinds of infinite products of mapping classes with a unique stably-expanding lamination $\Xi^\prime$ and discuss the diophantine approximation properties of its laves in terms of the complexity of the sequence defining the product (growth of exponents or stretch factors).

For instance, one may consider an alphabet of pseudo-Anosov elements $\varphi_1,\dots,\varphi_d$, and assume that the stretch factor increases under products, namely $\lambda(\varphi_i\varphi_j)>\max\{\lambda(\varphi_i),\lambda(\varphi_j)\}$.

Another example would be to consider two non-empty simple multiloops $\alpha=\sqcup_1^m \alpha_i$ and $\beta=\sqcup_1^n \beta_j$, and denote the associated Dehn multitwists by $D_\alpha= \prod_i^m D_{\alpha_i}$ and $D_\beta = \prod_1^n D_{\beta_j}$.
Assume that the union $\alpha\cup \beta$ is filling $\M^\prime$, in the sense that $\M^\prime\setminus(\alpha\cup\beta)$ is homotopic to a disjoint union of discs with at most one puncture.
In that case the product $D_\alpha D_\beta$ is pseudo-Anosov (this is the so called Penner construction for pseudo-Anosov mapping classes, see~\cite[§6]{Thurston_geo_dyna_diff-surf_1988}). Now for a sequence $x\in (\N_{\ge 1})^{\N}$ let $D_x = D_\alpha^{x_0} D_\beta^{x_1} \dots$ be the infinite product of those Dehn multitwists.
By contrast with the previous example, ne may allow arbitrarilly large powers of Dehn multitwists. 
\end{remark}

\begin{comment}
More generally, consider the collection $\mathcal{SC}(\TT)$ of all simple closed loops on $\M^\prime$ supported by $\TT$.
For any subset $\mathcal{S}\subset \GeoS(\TT)$, consider the (positive) monoïd $\Mod^{(+)}(\TT)\subset \Mod(\M)$ generated by (positive) Dehn twists $D_{\gamma}$ along those simple loops $\gamma\in \mathcal{SC}(\TT)$.
% This monoïd $\Mod^+(\TT)$ preserves the set of conjugacy classes in $\Gamma^+_{RL}$, as well as the combinatorial geodesics in $\TT$.

One may construct a sequence of submonoïds, that "close up" on a lamination $\X^\prime$.

Consider a minimal lamination $\Xi^\prime$, and up to the action of $\Gamma^\prime$, we may assume that its lamination language is recurrent $\La(\Xi^\prime) \cap \vec{\EE}(\TT^\prime)^3$ contains the word $\vec{\Ee}_{L^{-1}} \vec{\Ee}_{\Id} \vec{\Ee}_R$.
Now in its involution language $\La(\Xi^\prime) \subset \Al^\star$, consider the set of words associated to simple loops: the Dehn twists along those simple loops generate a submonoïd of the mapping class group that we denote by $\Mod(\La(\Xi^\prime))\subset \Mod(\M^\prime)$.
One may show that $\Xi^\prime$ is fixed by an infinite product of Dehn twists in $\Mod(\La(\Xi^\prime))$. 
\end{comment}

\subsection{Questions and conjectures on the set of profinite simple numbers}

\label{subsec:profinite-simple}
% \subsubsection*{On the set of profinite (congruence) simple numbers}

Let us give a name to the numbers covered by Theorem~\ref{thm:simple-geodesic-trichtomy}.

\begin{definition}[profinite simple]
    \label{def:profinitely-simple}
    Define the subset $\widetilde{\GeoS}(\Gamma)\subset \Geo(\HP)$ of \demph{profinite simple geodesics} (for $\Gamma$) as consisting of those $\xi \in \Geo(\HP)$ such that there exists a finite index subgroup $\Gamma^\prime \subset \Gamma$ such that $\xi^\prime =  \bmod{\Gamma^\prime} \in \GeoS(\M^\prime)$.
    Define the subset $\widetilde{\GeoS}^+(\Gamma) \subset \R\Proj^{1}$ of \demph{profinite simple numbers} (for $\Gamma$) as its image by the projection on (any) one of the two factors $\Geo(\HP)\to \R\Proj^{1}$.

    We may define the subsets of \demph{congruence-profinitely simple} geodesics and numbers $\mathcal{CS}^+(\Gamma)$, by restricting to congruence covers $\M(N)\to \M$ associated to the congruence subgroups $\Gamma(N)\subset \Gamma$ defined for $N\in \N_{>2}$ as the kernel of the reduction $\bmod{N}$ morphism $\PSL_2(\Z)\to \PSL_2(\Z/N)$.

    One may similarly define the subsets of \demph{purely morphic} profinitely simple geodesics, as well as their intersections with the congruence-profinitely simple geodesics, hence all the corresponding subsets of real numbers. 
\end{definition}

\begin{remark}[Hausdorff dimension]
    \label{rem:profinitely-simple-dimH}
    It follows from~\cite{Birman-Series_Hausdorff-dimension-simple-geodesics_1985} that for every finite cover $\M^\prime\to \M$ we have $\dim_H \Geo(\M^\prime)=0$.
    Thus $\widetilde{\GeoS}(\Gamma)\subset (\R\Proj^{1})^2$ hence $\widetilde{\GeoS}^+(\Gamma)\subset \R\Proj^{1}$ have Hausdorff dimension $0$.
\end{remark}

\begin{question}[continued fractions]
    \label{quest:profinitely-simple-contfrac}
    A combinatorial characterisation of the continued fractions of profinite simple numbers may be given: they are all $\{L,R\}$-sequences associated to lamination languages in trivalent graphs, which have mostly been characterised in~\cite{Ferenczi-Zamboni_Languages-k-IET_2008, Lopez-Narbel_lamination-languages_2013}.
    
    What if we restrict to congruence-profinitely simple numbers? Can we characterize describe the continued fraction expansions of ends of simple geodesics in the congruence cover $\M(N)$? (The case $N=3$ boils down to the Markov and Sturmian sequences discussed in Section~\ref{sec:ModTorus}.) 

    In upcoming work, we will study in detail the combinatorics associated to simple geodesics in the normal congruence covers of genus $0$ from Figure~\ref{fig:congruence-cover}. 
\end{question}

\begin{question}[Lagrange spectrum]
    \label{quest:profinitely-simple-Lagrange-Spec}
    By Lemma~\ref{lem:simple-low-cusp-excursions}, profinitely simple reals have finite Lagrange constant: can we describe the metric topology of the subsets $\LC(\mathcal{CS}^+(\Gamma))\subset \LC(\mathcal{S}^+(\Gamma))\subset \LC(\R)$?
    What is their Hausdorff dimension? (Note: $\exists \mathscr{S} \subset \R$ with $\dim_H(\mathscr{S})=0$ but $\dim_H(\LC(\mathscr{S}))=1$.)

    % Let us note that the fact that $\dimH \widetilde{\GeoS}=0$ does not say anything about $\dimH \LC \widetilde{\GeoS}$. 
    % Indeed, one may construct subsets $\mathscr{S}\subset \R$ with $\dimH(\mathscr{S})=0$ but $\dimH(\LC (\mathscr{S}))=1$ 
\end{question}

\begin{question}[Mahler spectra $w_2, \hat{w}_2$]
    \label{quest:profinitely-simple-Mahler-Spec}
    Can we describe the metric topology of the Mahler spectra $w_2(\widetilde{\mathcal{CS}}(\Gamma)) \subset w_2(\widetilde{ \GeoS}(\Gamma))$ and $\hat{w}_2(\widetilde{\mathcal{CS}}(\Gamma)) \subset \hat{w}_2(\widetilde{\GeoS}(\Gamma)) \subset [2,1+\phi]$?
    
    We will recall in~\ref{thm:w2-characteristic-Sturmian} the~\cite[Theorem 3.1]{Bugeaud-Laurent_Diophantine-exponents-Sturmian_2005} describing those spectra in restriction to numbers arising from simple geodesics in the modular torus $\M'$ (associated to the derived group $\Gamma'$) that pass through the Weierstrass point.
    This gives a (very) partial answer to the previous question, and some hints on how to refine it, and eventually adress it. 
\end{question}

\begin{comment}
\subsubsection*{Equidistribution of angles of geodesics with algebraic endpoints}

Consider a finite index subgroup $\Gamma^\prime \subset \Gamma$ corresponding to a finite cover $\M^\prime\to \M$.
    
Consider a real algebraic $\xi^+\in \R\Proj^{1}$ of degree $>2$ and let $\xi^-=0 \in \R\Proj^{1}$.

The geodesic $\xi =(\xi^-,\xi^+) \subset \HP$ projects $\bmod{\Gamma^\prime}$ to a geodesic $\xi^\prime \subset \M^\prime$, which descends from the cusp $0^\prime$ and then winds around $\M^\prime$ according to the continued fraction expansion of $\xi^+$.

\begin{conjecture}[equidistribution of self-intersection angles]
    If $\xi^+$ is algebraic, then the projected geodesic $\xi^\prime\subset \M^\prime$ equidistributes in the unit tangent bundle.

    In particular, it has a collection of self-intersection points that is dense in $\M^\prime$ and a collection of self-intersection angles that is dense in $\R\bmod{2\pi}$.
\end{conjecture}
\end{comment}

% \bibliographystyle{alpha} %apalike
% \bibliography{biblio}

%% file: images/tikz/Schlegel_bipartite_2-3-4-5.tex
\tikzset{
	% Filled solid triangles
	vtx_tri/.style={
		regular polygon,
		regular polygon sides=3,
		fill=black,
		inner sep=1pt,
	},
	% Hollow outline squares
	vtx_sq/.style={
		rectangle,
		draw=black,
        fill=white,
		inner sep=1pt,
	}
}

%%%%%%%%%%%% TRIANGLE %%%%%%%%%%%%
\begin{tikzpicture}[scale=1]
	
	\coordinate (T) at (0,1);   % Top
	\coordinate (B) at (0,-1);  % Bottom
	\coordinate (M) at (0,0);     % Middle
	\coordinate (L) at (-1,0);    % Left
	\coordinate (R) at (1,0);     % Right
    
	% --- Outer Frame ---
	% Draw the main rectangle outline
	\draw (T) -- (R) -- (B) -- (L) -- cycle;
    % Draw vertical
    \draw (T) -- (B);
    
	% --- Top Triangle ---
	\node[vtx_tri, rotate=180] (top_tri) at (T) {};
	
	% --- Bottom Triangle ---
	\node[vtx_tri, rotate=0] (bottom_tri) at (B) {};

    % --- Outer Nodes (Squares) ---
	\node[vtx_sq] at (L) {};
	\node[vtx_sq] at (R) {};
	\node[vtx_sq] at (M) {};
    
\end{tikzpicture}
%%%%%%%%%%%% TETRAHEDRON %%%%%%%%%%%%
\begin{tikzpicture}[scale=3]

% Vertices
\coordinate (A) at (0,0);
\coordinate (B) at (1,0);
\coordinate (C) at (0.5,0.866);
\coordinate (D) at (0.5,0.32);

% Edges
\draw (A)--(B)--(C)--cycle;
\draw (D)--(A) (D)--(B) (D)--(C);

% Midpoints
\foreach \X/\Y in {A/B, B/C, C/A, D/A, D/B, D/C} %
{\path node[vtx_sq] at ($(\X)!0.5!(\Y)$) {};}

% Vertices (triangles)
\foreach \P in {A,B,C,D} %
{\node[vtx_tri] at (\P) {};}

\end{tikzpicture}
%%%%%%%%%%%% CUBE %%%%%%%%%%%%
\begin{tikzpicture}[scale=3]

% Outer square
\coordinate (A) at (0,0);
\coordinate (B) at (1,0);
\coordinate (C) at (1,1);
\coordinate (D) at (0,1);

% Inner square
\coordinate (E) at (0.3,0.3);
\coordinate (F) at (0.7,0.3);
\coordinate (G) at (0.7,0.7);
\coordinate (H) at (0.3,0.7);

% Edges
\draw (A)--(B)--(C)--(D)--cycle;
\draw (E)--(F)--(G)--(H)--cycle;
\draw (A)--(E) (B)--(F) (C)--(G) (D)--(H);

% Midpoints
\foreach \X/\Y in {
  A/B,B/C,C/D,D/A,%
  E/F,F/G,G/H,H/E,%
  A/E,B/F,C/G,D/H} %
  {\path node[vtx_sq] at ($(\X)!0.5!(\Y)$) {};}

% Vertices
\foreach \P in {A,B,C,D,E,F,G,H} %
{\node[vtx_tri] at (\P) {};}

\end{tikzpicture}
%%%%%%%%%%%% DODECAHEDRON %%%%%%%%%%%%
\begin{tikzpicture}[scale=2]

% Outer pentagon (A)
\coordinate (A1) at (90:1.2);
\coordinate (A2) at (162:1.2);
\coordinate (A3) at (234:1.2);
\coordinate (A4) at (306:1.2);
\coordinate (A5) at (18:1.2);

% Second layer pentagon (B)
\coordinate (B1) at (90:0.75);
\coordinate (B2) at (162:0.75);
\coordinate (B3) at (234:0.75);
\coordinate (B4) at (306:0.75);
\coordinate (B5) at (18:0.75);

% Third layer pentagon (C) - offset by 36 degrees
\coordinate (C1) at (54:0.4);
\coordinate (C2) at (126:0.4);
\coordinate (C3) at (198:0.4);
\coordinate (C4) at (270:0.4);
\coordinate (C5) at (342:0.4);

% Innermost pentagon (D)
\coordinate (D1) at (54:0.15);
\coordinate (D2) at (126:0.15);
\coordinate (D3) at (198:0.15);
\coordinate (D4) at (270:0.15);
\coordinate (D5) at (342:0.15);

% Edges and Midpoints
\foreach \X/\Y in {
  % Outer bounds
  A1/A2, A2/A3, A3/A4, A4/A5, A5/A1,%
  % Inner bounds
  D1/D2, D2/D3, D3/D4, D4/D5, D5/D1,%
  % A to B spokes
  A1/B1, A2/B2, A3/B3, A4/B4, A5/B5,%
  % C to D spokes
  C1/D1, C2/D2, C3/D3, C4/D4, C5/D5,%
  % Zigzag intermediate connections B to C
  B1/C1, B1/C2,%
  B2/C2, B2/C3,%
  B3/C3, B3/C4,%
  B4/C4, B4/C5,%
  B5/C5, B5/C1%
} {
\draw (\X) -- (\Y);%
  \path node[vtx_sq] at ($(\X)!0.5!(\Y)$) {};
}

% Vertices
\foreach \P in {
  A1,A2,A3,A4,A5,%
  B1,B2,B3,B4,B5,%
  C1,C2,C3,C4,C5,%
  D1,D2,D3,D4,D5%
} {
  \node[vtx_tri] at (\P) {};
}

\end{tikzpicture}

%% file: sec3-ModTorus.tex
\section{Simple geodesics in the modular torus and abelian cover}
\label{sec:ModTorus}

Subsection~\ref{subsec:M'-H1(M';Z)} defines the modular torus $\M'$ and its mapping class group $\SL_2(\Z)$.
Subsection~\ref{subsec:M''-hexp} describes its universal abelian cover $\M''$ and its symmetries (the only original work in this Section).
Subsection~\ref{subsec:simple-M'-contfrac} uses the mapping class group action of $\Mod(\M')$ to find the continued fractions associated to simple geodesics in $\M'$, first observed by~\cite{Cohn_Markoff-perforated-torus_1971}, then revisited by~\cite{Haas_geometry-Markoff-forms_1987} and~\cite{Series_Geo-Markov-Num_1985}, see also~\cite{Springborn_hyperbolic-Markov-forms_2017, Springborn_worst-approx-ratioals_2024}.
The last Subsection~\ref{subsec:Diophantine-Approx-Markov-Sturm} surveys properties on the diophantine approximation properties of simple geodesics in the modular torus, from~\cite{ADQZ_transcendence-sturmian_2001, Bugeaud-Laurent_Diophantine-exponents-Sturmian_2005, Adamczewski-Bugeaud_transcendence-measures-contfrac_2010}. 

\subsection{The derived modular group and the modular torus}

\label{subsec:M'-H1(M';Z)}

The abelianisation $\Z/2*\Z/3\to \Z/2\times \Z/3$ of the modular group corresponds to the $\Z/6$-Galois cover $\M' \to \M$ of the modular orbifold by the modular torus, which is a cusped torus whose fundamental group $\pi_1(\M')=\PSL_2(\Z)'$ is freely generated by 
% $A = [T,S^{-1}] = RL$ and $B = [T^{-1},S] = LR$. 
\begin{equation*}
    A = [T,S^{-1}] = RL =
    \begin{psmallmatrix}
    2 & 1\\
    1 & 1
    \end{psmallmatrix}
    \quad \mathrm{and} \quad
    B = [T^{-1},S] = LR =
    \begin{psmallmatrix}
    1 & 1\\
    1 & 2
    \end{psmallmatrix}
    .
\end{equation*}
(As an exercise, show that $\Gamma'$ consists of all infinite order $C\in \Gamma$ with $\Rad(C)=0\bmod{6}$.)
% and contains the kernel $\Gamma(6)$ of $\PSL_2(\Z)\to \PSL_2(\Z/6)$.
% the intersection of the kernels $\Z/2*\Z/3 \to \Z/2$ and $\Z/2*\Z/3 \to \Z/3$

\begin{figure}[h]
    \centering
    \scalebox{0.48}{\subfile{images/tikz/action-LR-RL-HP}}
    \scalebox{0.48}{\subfile{images/tikz/hexagon-to-modular-monodromy}}
    \caption{The free group $\PSL_2(\Z)'$ acts on $\HP$ with quotient a cusped torus $\M'$.
    The Galois group $\PSL_2(\Z)/\PSL_2(\Z)'=\Z/6$ acts on $\M'$ with quotient $\M$.}
    \label{fig:Hex-torus}
\end{figure}

The homotopy classes of loops in $\M'$ correspond to the conjugacy classes in $\pi_1(\M')$, hence to the reduced cyclic words on $\{A,A^{-1},B,B^{-1}\}$. 
For $n\in \Z^\ast$, the $n$-th power of the commutator $[A,B]$ corresponds to the loop winding $n$ times around the cusp; every other non-trivial homotopy class contains a unique geodesic. %, and every non-trivial homology class contains a unique simple geodesic.

The cusp-compactification $\overline{\M'} = \M'\sqcup \partial\M'=\Gamma'\backslash (\HP\sqcup \Q\Proj^1)$ is homeomorphic to a torus with a marked point, and its fundamental group is the quotient of $\Gamma'$ by the normal subgroup generated by $[A,B]$; that is $\pi_1(\overline{\M'})=\Gamma'/\Gamma''=H_1(\Gamma';\Z)=\Z_A\oplus \Z_B$. 

In $\M'$, the simple geodesics with both ends in the cusp are in bijection with the simple closed geodesics, and correspond to the non-trivial simple loops in $\overline{\M'}$.
Hence, the abelianisation map yields a map $\Gamma'\bmod{conj} \to H_1(\M';\Z)$ that restricts to a bijection between classes associated to simple geodesics in $\M'$ and primitive vectors of the lattice $H_1(\Gamma';\Z)$.

More precisely the abelianisation map yields a map $(\Gamma' \bmod{conj})^2 \to H_1(\M';\Z)^2$ that restricts to a bijection from independent conjugacy classes of bases of the free group $\Gamma'$ to bases of the free abelian group $H_1(\Gamma';\Z)$, and those correspond to pairs of simple closed geodesics in $\M'$ having one intersection point (equivalently, they are represented by pairs of elements in $\Gamma'$ whose commutator corresponds to the loop winding once around the cusp).

The mapping class group $\Mod(\overline{\M'})$ is identified by the Dehn-Nielsen-Baer~\cite[Theorem 8.1]{Farb-Margalit_MappingClassGroups_2012} with the outer automorphism group $\Out(\pi_1(\overline{\M'}))=\GL(\Z_A\oplus \Z_B)$,
and contains the orientation preserving mapping class group $\Mod^o(\overline{\M'}) = \SL(\Z_A \oplus \Z_B)$ with index $2$, the quotient being generated by $D_J\colon (A,B) \mapsto (B,A)$.
The group $\Mod^o(\overline{\M'}) = \SL(\Z_A\oplus \Z_B)$ is generated by the positive and negative Dehn twists $D_R^{\pm 1}$ and $D_L^{\pm 1}$ along the simple loops $A$ and $B$:
\begin{align*}
    &D_R\colon (A,B)\mapsto (A,AB) \quad
    &D_R^{-1}\colon (A,B)\mapsto (A,A^{-1}B) \\
    &D_L\colon (A,B)\mapsto (BA,B) \quad
    &D_L^{-1}\colon (A,B)\mapsto (B^{-1}A,B)
\end{align*}
and its relation are generated by $D_L D_R^{-1} D_L= D_R^{-1} D_L D_R^{-1}$.
It contains the elements $D_T=D_LD_R^{-1}$ and $D_S = D_L D_R^{-1} D_L$
whose common power $D_S^2 = D_T^3 = D_{-\Id}$ has order $2$.
\begin{comment}
\begin{equation*}
    D_T=D_LD_R^{-1} \colon (A,B)\mapsto (BA,A^{-1})
    \qquad 
    D_S = D_L D_R^{-1} D_L \colon (A,B)\mapsto (A^{-1}BA,A^{-1})
\end{equation*}
whose common powers $D_S^2 = D_T^3 = D_{-\Id} \colon(A,B)\mapsto \left(B^{-1} A^{-1}B, B^{-1}AB^{-1}A^{-1}B\right)$ has order $2$.
\end{comment}

The group $\Mod(\M')$ of diffeotopy classes of diffeomorphisms of $\M'$ fixing the cusp fits into Birman's short exact sequence $1\to \pi_1(\M') \to \Mod(\M') \to \Mod(\overline{\M'})\to 1$, and the associated map $\Mod(\overline{\M'})\to \Out(\pi_1(\M'))$ is an isomorphism.
Restricting to orientation preserving classes yields $1\to \Z_A*\Z_B \to \Map^+(\M') \to \SL(\Z_A\oplus \Z_B)\to 1$.

\subsection{The second derived modular group and hexpunctured plane}
\label{subsec:M''-hexp}

The abelianisation $\Gamma'\to \Gamma'/\Gamma''$ associated to Hurwicz' map $\pi_1(\M') \to H_1(\M';\Z)$, given by $\Z_A*\Z_B \to \Z_A\times \Z_B$, corresponds to the universal abelian cover $\M'' \to \M'$ with Galois group $\Gamma'/\Gamma''= H_1(\M';\Z)$.
The Jacobian integration map of $\M'$ based at the cusp $\infty\in \partial \M'$ lifts to the total space $\M''$ and identifies it with the lattice-punctured plane $H_1(\M';\R)\setminus H_1(\M';\Z)$.

The kernel of the abelianization of $\pi_1(\M')=\Gamma'$ yields the fundamental group $\pi_1(\M'')=\Gamma''$, which is freely generated by an infinite set of elements (whose conjugacy classes in $\Gamma'$ are) indexed by $H_1(\M';\Z)=\Gamma'/\Gamma''$.
For example, $\Gamma''$ is freely generated by the set of commutator-conjugates $\{(A^mB^n)\cdot [A,B] \cdot (A^mB^n)^{-1} \colon (m,n)\in \Z^2\}$ and by the set of commutators $\{[A^m,B^n] \colon (m,n)\in \Z^2\}$.

The quotient $\HP \to \Gamma''\backslash \HP$ maps the ideal triangulation $\triangle$ with vertices $\Gamma/\langle R\rangle\subset \partial \HP$ to an ideal triangulation $\triangle''$ of $\M''$ with vertices $\Gamma''\backslash \Gamma/\langle R\rangle$, and the trivalent tree $\TT \subset \HP$ dual to $\triangle$ to a trivalent graph $\Hex \subset \M''$ dual to $\triangle$.
Hence the inclusion $\Hex\subset \M''$ is a homotopy equivalence, and the complementary regions in $\M''\setminus \Hex$ are punctured regular hexagons (for the hyperbolic metric).

\begin{figure}[h]
    \centering
    %\scalebox{0.8}{\subfile{images/tikz/action-L-R-Hex}}
    % \vspace{-1cm}
    % \scalebox{0.9}{\subfile{images/tikz/action-LR-RL-Hex}}
    \includegraphics[width=0.5\linewidth]{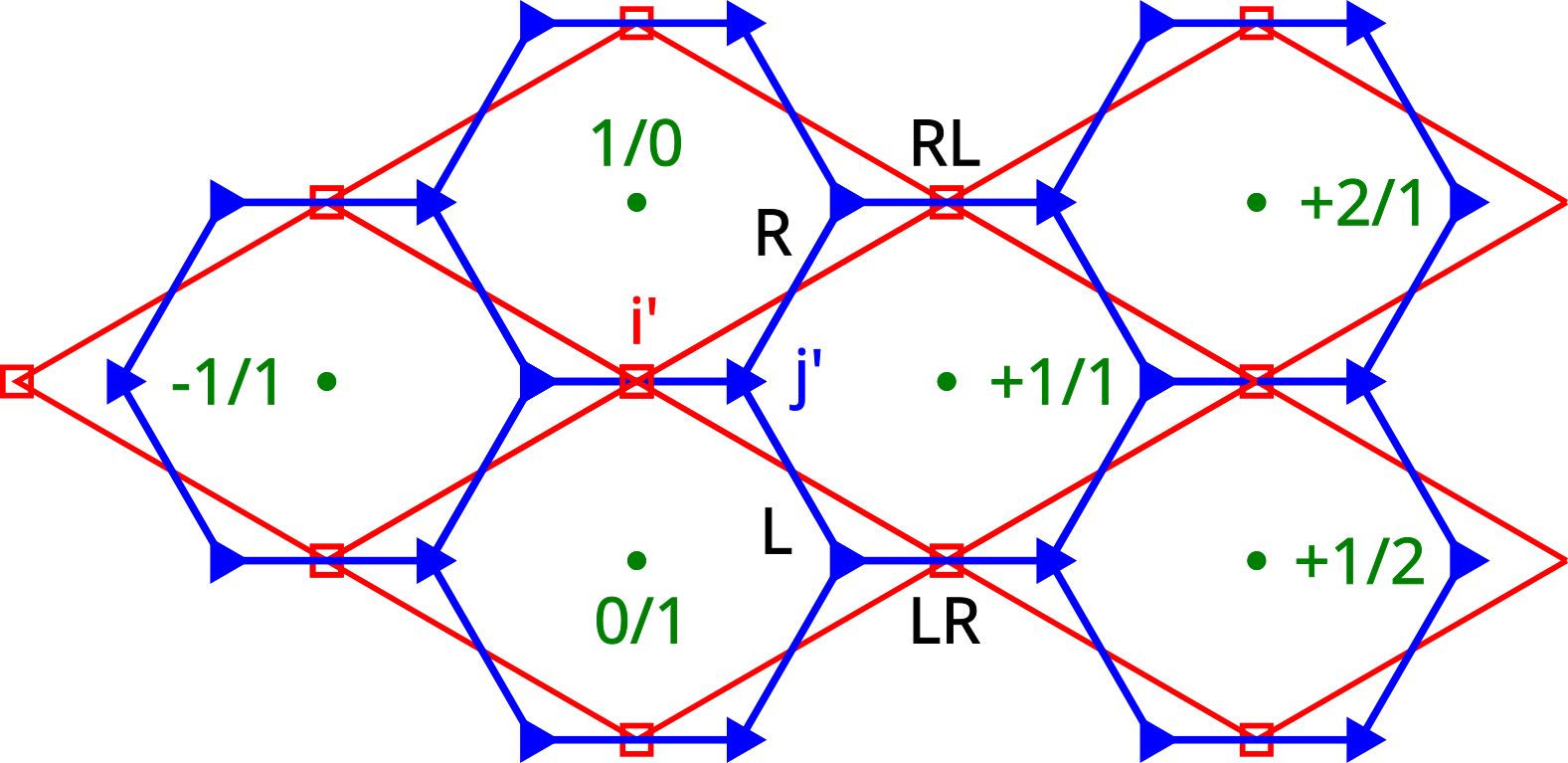}
    \hfill
    \includegraphics[width=0.45\linewidth]{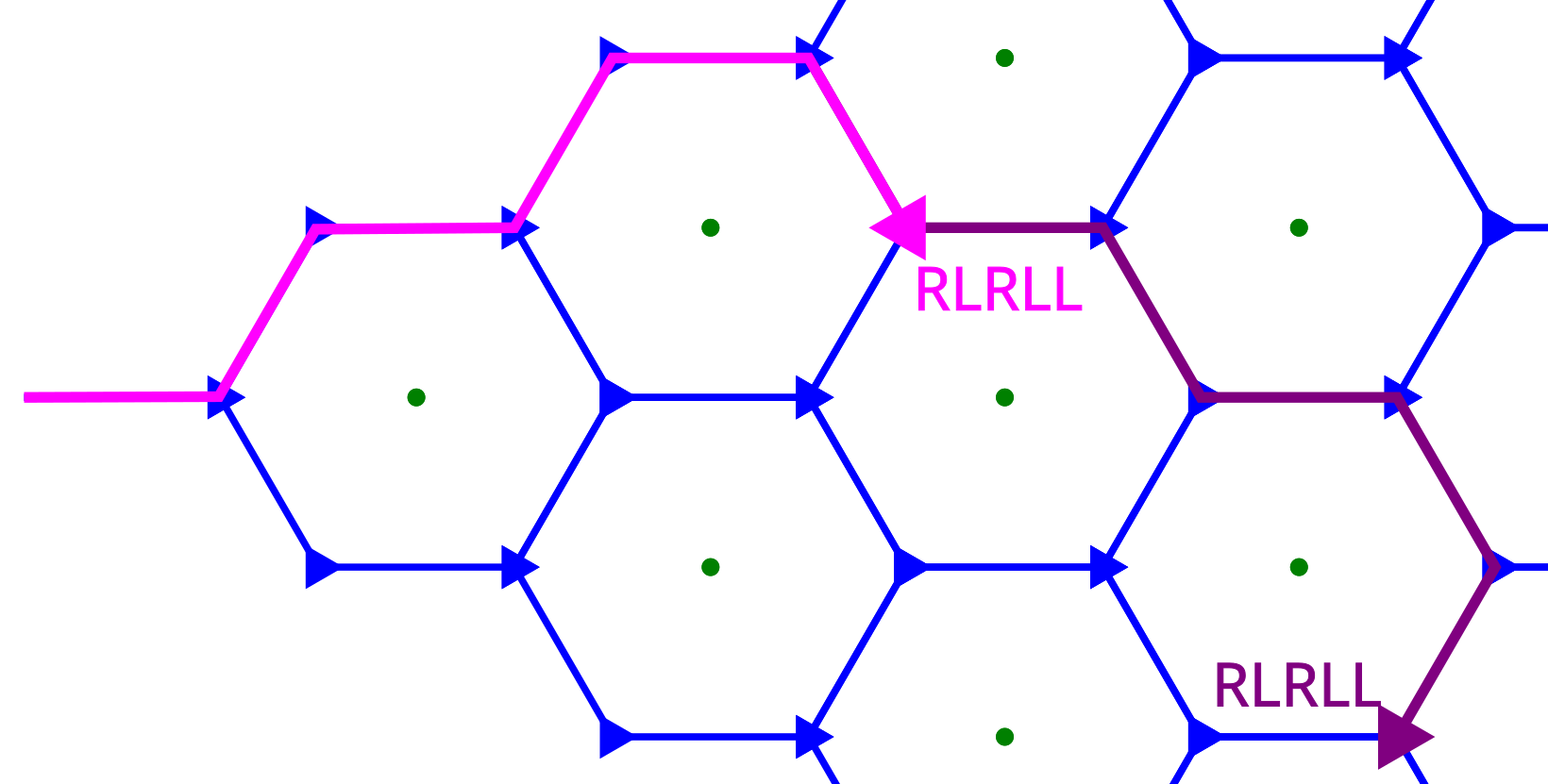}
    %\scalebox{0.8}{\subfile{images/tikz/Hex-path-RLRLL-1}}
    %\qquad
    %\scalebox{0.7}{\subfile{images/tikz/Hex-path-RLRLL-2}}
    \caption{Action of $\Gamma/\Gamma''$ on the honeycomb graph $\Hex$ with base arc $[i',j'\rangle$, and cusps $\Gamma''\backslash \Gamma /\langle R\rangle$.
    Paths encoded by $RLRLL$ and $(RLRLL)^2$ in the honeycomb graph $\Hex$.}
    \label{fig:Hex-fundomain}
\end{figure}

Let us describe the Galois action of $\Gamma/\Gamma''$ on the metabelian cover $\M'' \to \M$ combining the action of $\Gamma'/\Gamma''$ on $\M'' \to \M'$ with the action of $\Gamma/\Gamma'$ on $\M'\to \M$. (Recall that Galois actions are given by left multiplication whereas monodromy actions are given by right multiplication: these coincide on the identity coset, but otherwize are related by conjugacy by elements of the coset.)

The group $\Gamma/\Gamma''$ acts freely transitively on the arcs of $\Hex$. 
More precisely the generators $S$ and $T$ act by rotations of order $2$ and $3$ around the mid-edges and vertices of $\Hex$.
The subgroup $\Gamma'/\Gamma''$ generated by $A=RL$ and $B=LR$ acts by translation of $\Hex$ with fundamental domain a tripod formed by three adjacent edges. % as in figure~\ref{fig:Hex-fundomain}.

Consequently the action of $C\in \Gamma$ on $\Hex$ sends the base arc to another arc whose angle is determined by the coset $C\bmod{\Gamma'}$ with a representative $C'\in \pm \{\Id,L,R,S,T,T^{-1}\}$, and whose base point is determined by $C(C')^{-1}\bmod{\Gamma''}$. 
In particular $\Gamma'/\Gamma''$ identifies with the arcs of $\Hex$ that are parallel to the base arc. %  (and less canonically with the cusps).
% The metabelianisation map $\Gamma \to \Gamma/\Gamma''$ can be visualised in terms of the left-translation action on the cover $\TT\to \Hex$ as follows: 

This description leads to the following Theorem observed in~\cite[Section 3.2]{CLS_phdthesis_2022}.

\begin{theorem}[hexagonal symmetries]
\label{thm:CLS_hexagonal-group}
The group $\Gamma/\Gamma''$ is the semi-direct product 
%$\Gamma'/\Gamma'' \rtimes \Gamma/\Gamma'$
\begin{equation*}
    \Gamma/\Gamma''=\Gamma'/\Gamma'' \rtimes \Gamma/\Gamma'
\end{equation*} 
where the action of the quotient $\Gamma/\Gamma'=\Z/2\times \Z/3$ by outer-automorphisms (given by conjugacy) on the kernel $\Gamma'/\Gamma''=\Z_A\oplus \Z_B$ is generated by
\begin{equation*}
\AD_S(A,B) = (A^{-1},B^{-1}) 
\qquad
\AD_T(A,B) = (B^{-1}, AB^{-1}). 
\end{equation*}
This represents $\Gamma/\Gamma''$ as the affine isometry group of the oriented hexagonal lattice $H_1(\Gamma';\Z)=\Z_A\oplus \Z_B$ with $angle{(A,B)}=\frac{2\pi}{6}$, where the translation action of $\Gamma'/\Gamma''$ is by $\Z^2$-translation while the outer-automorphism action of $\Gamma/\Gamma'$ is by $\Z/6$-rotation.

Finally, the involutions \(K=\begin{psmallmatrix}
-1 & 0 \\ 0 & 1\end{psmallmatrix}\) and \(J=\begin{psmallmatrix}
0 & 1 \\ 1 & 0\end{psmallmatrix}\) generating \(\Gamma^\pm /\Gamma=\Z/2\) act as reflections of $H_1(\Gamma';\Z)=\Z_A\oplus \Z_B$ across the horizontal and vertical axes, namely:
\begin{equation*}
\AD_J(A,B) = (B, A) \qquad 
\AD_K(A,B)= (B^{-1}, A^{-1})
\end{equation*}
This describes $\Gamma^\pm/\Gamma''$ as a semi-direct product $(\Gamma/\Gamma'')\rtimes (\Gamma^\pm/\Gamma)=(\Gamma'/\Gamma'')\rtimes (\Gamma^\pm/\Gamma')$ and represents it as the group of affine isometries of the hexagonal lattice $H_1(\M';\Z)$.
\end{theorem}

The Jacobian map yields a conformal diffeomorphism sending the hyperbolic metric on $\M'$ to the Euclidean metric on $(H_1(\M';\R)\setminus H_1(\M';\Z) )\bmod{H_1(\M';\Z)}$).
In $\M'$, a simple Euclidean geodesic is homotopic to a unique simple hyperbolic geodesic: their lifts in $\M''$ are homotopic relative to their ends, so they cut the edge of the triangulation $\triangle''$ dual to $\Hex$ according to the same combinatorics.

\subsection{Combinatorics of simple geodesics in the modular torus}

\label{subsec:simple-M'-contfrac}

The aim of this section is to recall the description of simple geodesics $\xi'\subset \M'$ in terms of the continued fraction expansions of their endpoints $\xi^\pm \in \R\Proj^1$.
The method is to lift these geodesics to simple geodesics $\xi''\subset \M''$, which intersect $\Delta''$ according to the same combinatorics as (segments of) straight lines in the Euclidean plane $H_1(\M';\R)$, and use the action of $\Mod^o(\M')=\SL_2(\Z_A\oplus \Z_B)$.

% Indeed, the Jacobian map yields a conformal diffeomorphism sending the hyperbolic metric on $\M'$ to the Euclidean metric on $(H_1(\M';\R)\setminus H_1(\M';\Z) )\bmod{H_1(\M';\Z)}$). In $\M'$, a simple Euclidean geodesic is homotopic to a unique simple hyperbolic geodesic: their lifts in $\M''$ are homotopic relative to their ends, so they cut the edge of the triangulation $\triangle''$ dual to $\Hex$ according to the same combinatorics.

There are seven \emph{topological types} of simple geodesics $\xi'$ in $\M'$, depending on the asymptotic behaviour of its ends, which could either escape to the cusp or else accumulate on a minimal geodesic lamination, either periodic or without periodic leaves (see~\cite[Lemma 5.1]{Haas_Diophantine-approximation-hyperbolic-surfaces_1986}); the types are: cusp-to-cusp, cusp-to-periodic, periodic-to-cusp, periodic, cusp-to-aperiodic, aperiodic-to-cusp, aperiodic.
The action of $\Mod^o(\M')$ preserves simple geodesics and their topological type.

The geodesics on $\M'$ correspond to sequences over the symmetric set of generators $\{A^{\pm 1},B^{\pm 1}\}$ of the fundamental group $\Gamma'$, indexed by an interval of $\Z$ which is considered up to shift of indices (which acts as cyclic permutation on finite words).
% (finite or half-infinite or bi-infinite)
For simple geodesics, this sequence involves at most two letters: up to the action of the finite order mapping classes $\AD_S, \AD_T\in \Out(\Gamma')$ by $\AD_S(A,B)=(A^{-1}, B^{-1})$, $\AD_T(A,B)=(B^{-1}, AB^{-1})$, we may assume that the alphabet is $\{A,B\}$.
% and that all sequences have $A$ at the position indexd by $0$

Let us first recall from~\cite[Chapter 6]{Fogg_substitutions_2002} and~\cite{Aigner_Markov-uniqueness-100-years_2013} the definitions of Christoffel words and Cohn matrices to associated them Markov rational and quadratic real numbers.
%~\cite[Chapter 2]{Lothaire_algebraic-combinatorics-words_2002},

\begin{definition}[Christoffel word, Cohn matrix, Markov numbers]
    \label{def:Christoffel-Cohn-Markov}
    To a rational slope $\sigma \in [1,\infty]$ with continued fraction expansion $\sigma = \Ecf{c_0;c_1,\dots, c_{2k-1}}$, we associate the following data.
    
    The (lower) \demph{Christoffel word} in $\{A,B\}^\star$ is obtained by applying the composition of substitutions $D_\sigma = D_R^{c_0} \dots D_L^{c_{2k-1}}$ to the seed $A$.
    This yields a matrix in $\Gamma'\cap \PSL_2(\N)$, whose factorisation in $\{L,R\}^\star$ is obtained from the substitutions $(A,B)\mapsto (RL,LR)$, that turns out to be conjugate to a unique symmetric matrix in $R\PSL_2(\N)L$, which we call the \demph{Cohn matrix} $C_\sigma\in \Gamma'\cap \PSL_2(\N)$.
    We thus define the \demph{Markov rational} $\rho_\sigma = C_\sigma \cdot \infty\in \Q_{>1}$ and 
    % its right \emph{companions} of level $n\in \N_{>0}$ are $C_\sigma^n \cdot 1$.
    % The \demph{Markov triple of rationals} is $C_\sigma \cdot (0,1,\infty)$, and its center is the \emph{} \emph{companions} of level $n\in \N$ are $C_\sigma^n \cdot (0,1,\infty)$.
    the \demph{pair of Markov quadratic conjugates} as its (repulsive, attractive) fixed points: $(\xi^{-}_\sigma,\xi^+_\sigma) = (C_\sigma^{-\infty} \cdot \infty, C_\sigma^{+\infty} \cdot \infty) \in (-1,0)\times (1,\infty)$.
\end{definition}

\begin{remark}[cyclically primitive and symmetric]
    Since $C_\sigma \in R\PSL_2(\N)L$ is cyclically primitive and symmetric, the continued fraction expansion of $\rho_\sigma$ is a palindrome which is the minimal period of the purely periodic number $\xi^+_\sigma$, whose Galois conjugate satisfies $\xi^-_\sigma=-1/\xi^+_\sigma$.
\end{remark}

\begin{remark}[all rational slopes]
    \label{rem:all-rational-slopes}
    We may extend Definition~\ref{def:Christoffel-Cohn-Markov} to all slopes in $\Q_{\ge 0}$ using the action of $J\colon \sigma \mapsto 1/\sigma$ hence of $D_{J\sigma}=D_J D_\sigma$, so that $C_{J\sigma}$ is obtained from $C_\sigma$ by applying the substitution $(A,B)\mapsto (B,A)$ namely $\AD_J\colon (R,L)\mapsto (L,R)$, which by $\dag$-symmetry amounts to reversing the order of the $\{L,R\}$ word and inverting $\rho_\sigma, \xi_\sigma^\pm$ (so for $\sigma\in [0,1)$ we have $\rho_\sigma, \xi_\sigma^+, -1/\xi_\sigma^- \in (0,1)$).

    One could further extend Definition~\ref{def:Christoffel-Cohn-Markov} to all slopes in $\Q\Proj^1$ using the action of $S\colon (0,\infty] \ni \sigma \mapsto -1/\sigma \in (-\infty, 0]$ so that $D_{S\sigma}=D_S D_\sigma$ hence $C_{S\sigma}$ is obtained from $C_\sigma$ by applying the substitution $(A,B)\mapsto (B,A^{-1})$ (that is $D_S\colon (A,B)\mapsto (A^{-1}BA,A^{-1})$ followed by a cyclic permutation), and we may similarly define the associated Markov numbers.
\end{remark}

\begin{remark}[slope is frequency of $C_\sigma$]
    The slope $\sigma\in \Q\Proj^1$ is read from $C_\sigma$ by considering its abelianisation $(m,n)\in H_1(\M';\Z)$: if $C_\sigma \equiv A^mB^n \bmod{\Gamma''}$ then $\sigma = m/n \in \Q\Proj^1$.
\end{remark}

\begin{example}[from slopes to Markov numbers]
Let us choose a slope $\sigma\in \Q_{\ge 0}$ and compute the corresponding Christoffel word $C_\sigma$, and quadratic pair $(\xi_\sigma^-,\xi_\sigma^+)$.

For $\sigma = \infty = \Ecf{}$, apply nothing to $A=RL$ and cyclically permute to find $C_\sigma = RL$, yielding $\rho_\sigma = 2/1$ and $\xi_\sigma^\pm =\Ecf{(1,1,)^\N}=\tfrac{1}{2}(1\pm \sqrt{5})$.

For $\sigma = 1/1 = \Ecf{0,1}$, apply $D_R^{0}D_L^{1}$ to $A$ to find $BA$ and cyclically permute to obtain 
$C_\sigma = RRLL$, yielding $\rho_\sigma = 5/2$ and $\xi_\sigma^\pm =\Ecf{(2,2,)^\N}=1\pm \sqrt{2}$.

For $n\in \N$ and $\sigma = 1+1/n = \Ecf{1,n}$, apply $D_R^{1}D_L^{n}$ to $A$ to find $(AB)^nA$ and cyclically permute to obtain 
$C_\sigma = RL(RRLL)^{n-1} RL$ thus $\rho_\sigma = \Ecf{1,1,(2,2,)^{n-1},1,1}$ and $\xi_\sigma^+=\Ecf{(1,1,(2,2,)^{n-1},1,1)^\N}$.

For $n\in \N$ and $\sigma = n+1 = \Ecf{n,1}$, apply $D_R^{n}D_L^{1}$ to $A$ to find $BA^nA=(LR)(RL)^{n+1}$ and cyclically permute to obtain $C_\sigma = RR(RL)^{n}LL$ thus $\rho_\sigma = \Ecf{2,1,\dots,1,2}$ and $\xi_\sigma^+=\Ecf{(2,1,\dots,1,2)^\N}$.

For $n\in \N_{>0}$ and $\sigma = 1/n = \Ecf{0,n}$, apply $D_R^{0}D_L^{n}$ to $A$ to find $B^nA=(LR)^nRL$ and cyclically permute to obtain 
$C_\sigma = LL(RL)^{n-1}RR$ thus $\rho_\sigma = \Ecf{0;2,1,\dots,1,2}$ and $\xi_\sigma^+=\Ecf{0;(2,1,\dots,1,2,)^\N}$.
\end{example}

The next proposition is explained in~\cite[§14]{Springborn_worst-approx-ratioals_2024}.

\begin{proposition}[simple geodesics in $\M'$ with rational slope]
    \label{prop:simple-geodesics-Markov}
    Consider a primitive vector $(m,n)\in H_1(\M';\Z)=\Z_A+\Z_B$ and denote the associated slope $\sigma=m/n\in \Q\Proj^1$.
    
    The simple closed geodesic homologous to $A^mB^n$ is $\xi'_{\sigma}\in \Geo(\M')$, namely the projection $\bmod{\Gamma'}$ of the geodesic $\xi_\sigma = (\xi^-_\sigma,\xi^+_\sigma) \in \Geo(\HP)$.
    
    The unique simple geodesic from cusp-to-cusp that does not intersect $\xi'_\sigma$ is the projection $\bmod{\Gamma'}$ of the geodesic $(0,\rho_\sigma)\in \Geo(\HP)$.
    
    There are two more simple geodesics avoiding $\xi'_\sigma$, their past arise from the cusp and future accumulate on $\xi'_\sigma$: those are the projections $\bmod{\Gamma'}$ of the geodesics $(0,\xi^+_\sigma)$ and $(\infty,\xi^-_\sigma)$ in $\Geo(\HP)$.

    These $4$ geodesics of slope $\sigma$ are disjoint.
    
    Every simple geodesic of the type cusp-to-cusp or cusp-to-periodic or periodic arises as such, in a unique way (namely distinct slopes give rise to distinct quadruples), and any two such geodesics with different slope have intersections.
\end{proposition}

\begin{figure}[h]
    \centering
    \includegraphics[width=0.5\linewidth]{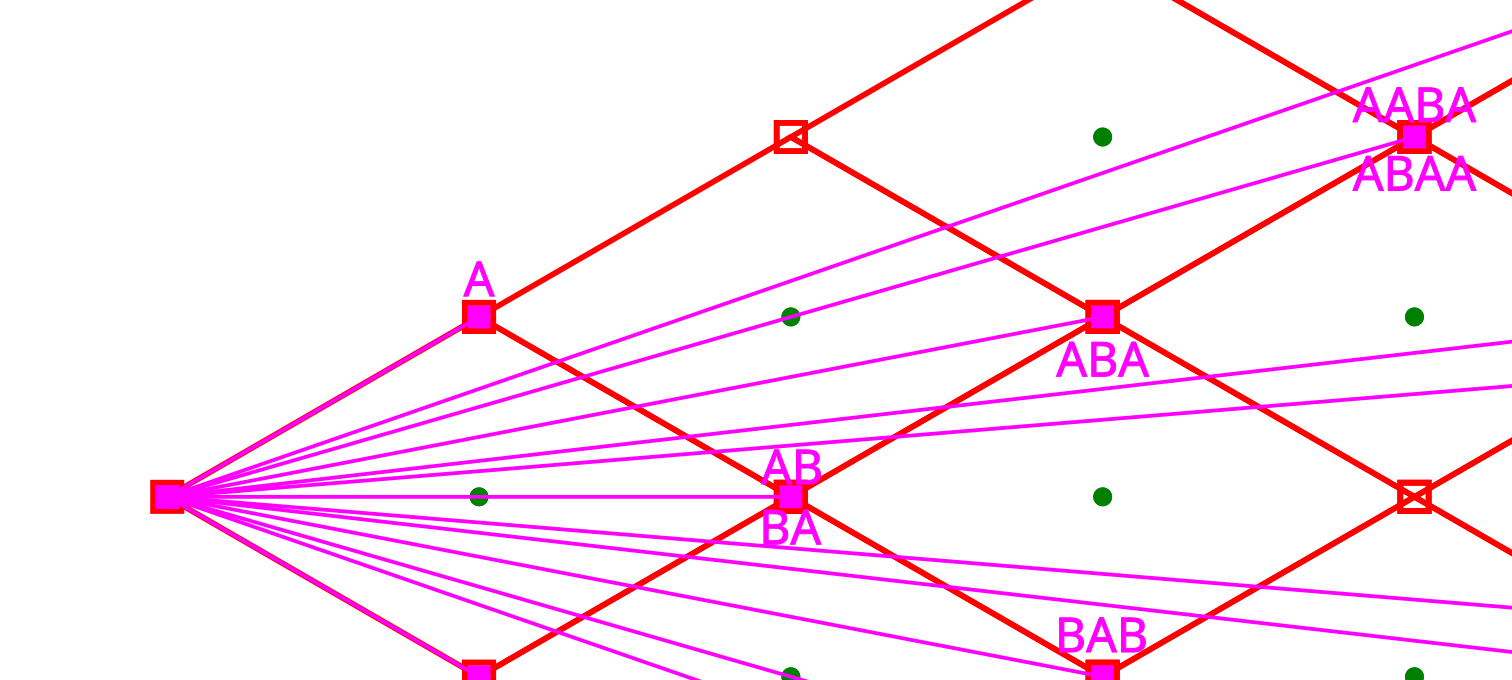}
    \includegraphics[width=0.49\linewidth]{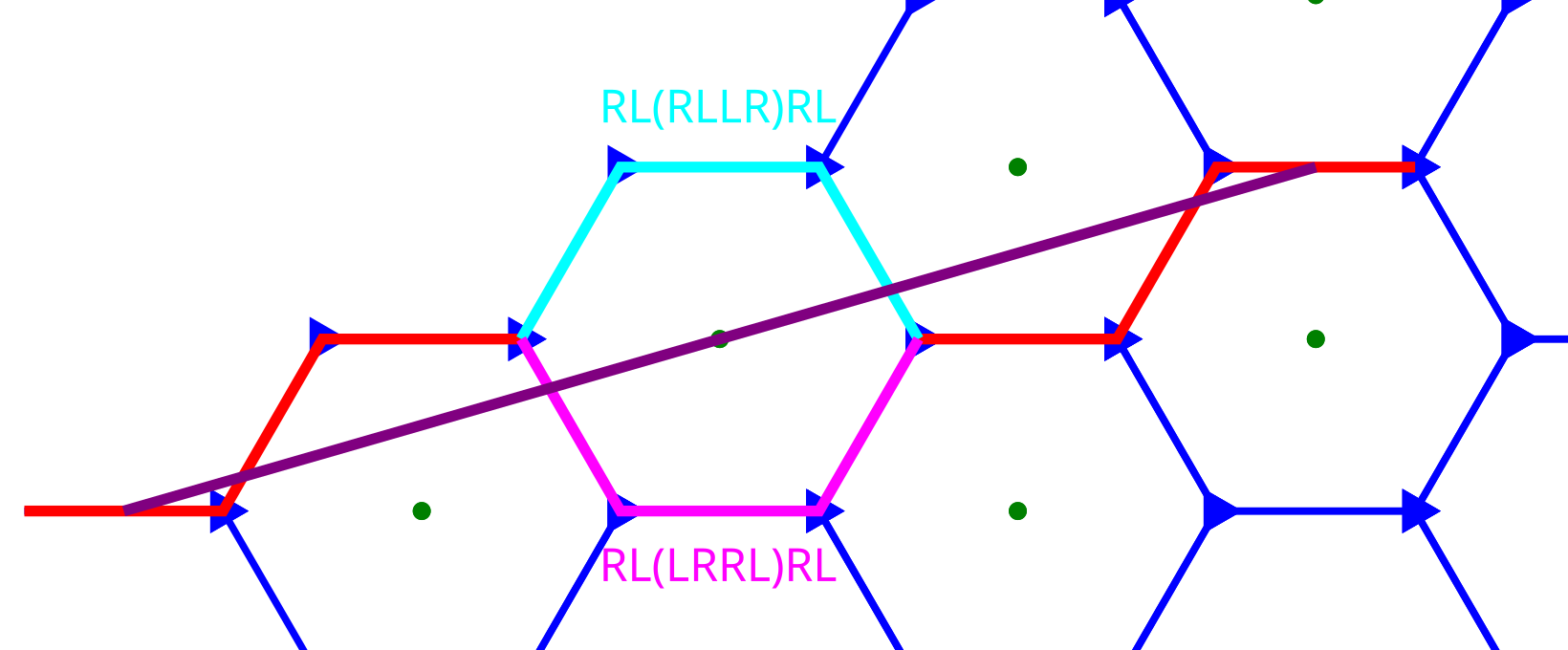}
    \caption{Simple closed geodesics in $\M'$ lift to lines of rational slope in $H_1(\M';\R)$ through $i'$.\\
    The corresponding combinatorial hex-paths are unique only up to cyclic permutation.}
\end{figure}

We now recall the notion of Sturmian sequences from~\cite[Chapter 6]{Fogg_substitutions_2002} to associate them pair of real numbers (which are transcendental as we will recall~\ref{thm:Sturm-transcendent}).

\begin{definition}[Sturmian sequences and numbers]
\label{def:Sturmian}
Define the set of \emph{Sturmian parameters} as the (slope, intercepts) $(\sigma, \tau) \in (0,1)\times [0,1)$ such that $\sigma \notin \Q$ and $\tau \notin (1/2+\Z)+\sigma(1/2+\Z)$.

To such Sturmian parameters, is associated the (lower) \emph{Sturmian sequence} $X_{\sigma,\tau}\in \{A,B\}^\Z$ by
\begin{equation*}
\forall n\in \Z \colon \quad
X_{\sigma,\tau}[n]= 
A^{r_{\sigma,\tau}(n)}B \quad \mathrm{where} \quad r_{\sigma,\tau}(n)=\lfloor (n+1)\sigma+\tau \rfloor - \lfloor n\sigma+\tau \rfloor \in \{0,1\}
% \begin{cases}
%    A & \mathrm{\:if\:} \lfloor (n+1)\sigma+\tau \rfloor - \lfloor n\sigma+\tau \rfloor = 1 \\
%    B & \mathrm{\:if\:} \lfloor (n+1)\sigma+\tau \rfloor - \lfloor n\sigma+\tau \rfloor = 0
% \end{cases}
\end{equation*}
encoding the $\Z$-orbit of $\tau$ under addition by $\sigma_1$ for the partition $\R\bmod{\Z}=[-\sigma_1,1-\sigma_1)$ in two intervals $[-\sigma_1,1-2\sigma_1) \sqcup [1-2\sigma_1, 1-\sigma_1)$, as $r_{\sigma,\tau}(n)$ is the $n$-the return time to the left interval.

We may thus define the \emph{pair of conjugate Sturmian numbers} $(\xi^-_{\sigma,\tau},\xi^+_{\sigma,\tau})\in (-\infty,-1)\times (0,1)$ so that $-1/\xi^-_{\sigma,\tau}$ and $1/\xi^+_{\sigma,\tau}$ are the fixed points of the infinite products of the generators $A,B\in \Gamma'$:
\begin{equation*}
    -1/\xi^-_{\sigma,\tau} = X_{\sigma,\tau}[-1]\cdot X_{\sigma,\tau}[-2] \cdots \cdot \infty 
    \qquad \xi^+_{\sigma,\tau} = X_{\sigma,\tau}[0]\cdot X_{\sigma,\tau}[1]\cdots \cdot \infty
\end{equation*}
whose continued fraction expansions in $\{1,2\}^\N$ can be recovered from $X\in \{A,B\}^\Z$ by applying the morphisms $(A,B)\mapsto (RL,LR)$ to find two sequences in $\{L,R\}^\Z$ and grouping exponents.

When $t=0$, the sequence $X_{\sigma,0}$ and numbers $\xi^\pm_{\sigma,0}$ are often called \emph{characteristic Sturmian}.
\end{definition}

\begin{remark}[upper Sturmian sequence]
    One may also define the upper Sturmian sequence $X_{\sigma,\tau}$ using a ceiling in the definition of $r_{\sigma,\tau}(n)$, but the $\{L,R\}$-encoding of $X_{\sigma,\tau}$ will only differ by $\Shift$.
\end{remark}

\begin{remark}[all irrational slopes]
    \label{rem:all-irrational-slopes}
    We may extend Definition~\ref{def:Sturmian} to all positive irrational slopes $\sigma\in \R_{>0}$ by the action of $J\colon (\sigma,\tau) \mapsto (1/\sigma, 1-\tau)$ which changes $X_{\sigma,\tau}$ by applying the substitution $D_J\colon (A,B)\mapsto (B,A)$.
    We may further extend Definition~\ref{def:Sturmian} to all irrational slopes $\sigma\in \R$ by the action of $S\colon (\sigma,\tau) \mapsto (-1/\sigma, -\tau)$ which changes $X_{\sigma,\tau}$ by applying the substitution $D_S\colon (A,B)\mapsto (B,B^{-1}A^{-1}B)$, that is $(A,B)\mapsto (B,A^{-1})$ followed by a cyclic permutation.
    
    Hence we may similarly define the associated Sturmian numbers.
\end{remark}

\begin{remark}[From Ostrowski to $S$-adic expansion]
    Consider Sturmian parameters $(\sigma, \tau)\in (0,1) \times [0,1)$ and write $\sigma = \Ecf{0,s_1,\dots}$.% and recall $1/\sigma_k = \Ecf{0;s_k,\dots}$.
    The \emph{Ostrowski expansion} of $\tau$ in base $\sigma$ is the unique $t\in \N^\N$ satisfying for all $n \in \N$ that $0\le t_n\le s_n$ and $t_{n+1}=s_{n+1} \implies t_n=0$, such that:
    \begin{equation*}
        \tau = \sum_{n=0}^\infty \tfrac{t_n}{\sigma_n}
        = t_1  \cdot \Ecf{0;s_1,\dots} + t_2  \cdot \Ecf{0;s_2,\dots} + \dots
    \end{equation*}    

    With the sequences $s\in (\N_{\ge1})^{\N_{\ge_1}}$ and $t\in (\N_{\ge1})^{\N_{\ge_1}}$, one may write $\xi_{\sigma,\tau} \in \{A,B\}^\Z$ as the attractive fixed point under an infinite composition of homeomorphisms of $\{A,B\}^\Z$, namely the $\Shift$ and the extensions of $D_L,D_R$:
    \begin{equation*}
        X_{\sigma,\tau} =
        \Shift^{t_1} \circ D_L^{s_1} \circ 
        \Shift^{t_2} \circ D_R^{s_2} \circ
        \Shift^{t_3} \circ D_L^{s_3} \circ
        \Shift^{t_4} \circ D_R^{s_4} \circ \dots \cdot (A^\Z)
    \end{equation*} 
    In particular $X_{\sigma,\tau}$ is morphic (that is the image by a morphism of $(\{A,B\},\Shift)$ by the fixed point of a single morphism of the monoid $\{A,B\}^\star$), if and only if both $s,t$ are eventually periodic, which means that $\sigma$ is quadratic and $\tau \in \Q(\sigma)$~\cite[Theorem 2.19 and Proposition 2.11]{Berthe-Holton-Zamboni_Sturmian_2006}.
    
    Note that in the characteristic Sturmian case $t=0$ one may forget the $\Shift^{t_k}$.
    It follows that a Sturmian number $\xi^+_{\sigma,\tau}$ is a limit of Markov quadratic numbers if and only if $\tau=0$.
\end{remark}

\begin{remark}[cutting sequence and frequencies]
    Behold the Euclidean plane $H_1(\M';\R)$ with coordinates $(A,B)$ punctured along the rhombic lattice translated by $(1/2,1/2)$, the line $\alpha = \sigma\beta+\tau$ starting from the point $-(1/2,1/2)+(0,\tau)$ is homotopic (relative its endpoints) to a broken line in the Cayley graph of $H_1(\M';\Z)$, and this yields the Sturmian sequence $X_{\sigma,\tau}\in \{A,B\}^\Z$.

    The sequence $X_{\sigma,\tau}\in \{A,B\}$ admits frequencies: for any sequence of subwords $X_{\sigma,\tau}([k,k+l))$ of growing length $l\in \N$, the ratio between the numbers of $A$'s and $B$'s converges to the slope $\sigma$.
\end{remark}

\begin{remark}[characterization by factor complexity]
Over any alphabet, a sequence $X$ is:
\begin{itemize}[noitemsep]
    \item periodic $\iff  \exists n\in \N \colon \fac_n(X)\le n$ (see~\cite{Hedlund-Morse_symbolic-dynamics_1938}, and~\cite[§4.3]{Cassaigne-Nicolas_Factor-complexity_2010}).
    \item Sturmian $\iff \forall n\in \N \colon \fac_{n}(X) = n+1$ (see~\cite[Chapter 6]{Fogg_substitutions_2002})
\end{itemize}

\end{remark}

\begin{remark}[recovering parameters from symbolics and geometry]
    The following data determine one-other:
    \begin{itemize}[noitemsep]
        \item the $\Shift$-orbit of the sequence $X_{\sigma,\tau}\in \{A,B\}^\Z$
        \item the slope $\sigma\in \R_{>0}$ and the class of the intercept $\tau \in \R\bmod{(\Z+\sigma \Z)}$
        \item the $\Gamma'$-orbit of $(\xi^-_{\sigma,\tau},\xi^+_{\sigma,\tau})\in \Geo(\HP)$, that is the simple geodesic $\xi_{\sigma,\tau} \in \GeoS(\M')$
    \end{itemize}
\end{remark}

\begin{proposition}[simple geodesics on $\M'$ of irrational slope]
    \label{prop:simple-geodesics-Sturm}
    For an irrational slope $\sigma\in \R\setminus \Q$ there is a minimal geodesic lamination $\Xi_{\sigma}'\subset \M'$ consisting of the set of simple geodesics $\xi_{\sigma,\tau}'\in \GeoS(\M')$ for all $\tau\in \R\bmod{\Z+\sigma\Z}$ such that $\tau \notin (1/2+\Z)+\sigma(1/2+\Z)$.
    These are all the non-empty minimal geodesic lamination with no closed leaves in $\M'$.
\end{proposition}

\begin{remark}[limit of Markov if and only if characteristic]
    A Sturmian number $\xi^+_{\sigma,\tau}$ is a limit of Markov quadratic numbers $\xi^+_\sigma$ if and only if it is characteristic ($\tau=0\bmod{\Z}$).
    Geometrically, those are the leaves of minimal geodesic laminations that pass through one of the three Weierstrass points of $\M'$, namely $\{i, Li, Ri \}\bmod{\Gamma}$.
\end{remark}

\subsection{Diophantine approximation of simple geodesics in \texorpdfstring{$\M'$}{M'}}
\label{subsec:Diophantine-Approx-Markov-Sturm}

\subsubsection*{Transcendence of Sturmian continued fractions}

We finally mention the initial motivation of this work,

\begin{context}
    Fix an irrational slope $\sigma \in \R_{>0}$, consider its continued fraction expansion $\sigma= \Ecf{s_0,s_1,\dots}$, and denote $\lambda_\sigma = \limsup_k \Ecf{s_k; s_{k-1},\dots, s_1, s_0}$. 
    
    Let $\tau \in (0,1)$ such that $\tau \notin  (1/2+\Z)+\sigma (1/2+\Z)$ so as to obtain the Sturmian sequence $X_{\sigma, \tau}\in \{A,B\}^\Z \subset \{L,R\}^\Z$ and real numbers $\xi^\pm_{\sigma,\tau}\in \R$.
\end{context}

The transcedence of $\xi^\pm_{\sigma,\tau}$ was initially proved in~\cite[Proposition 3 and Theorem 7]{ADQZ_transcendence-sturmian_2001}: their strategy consists of showing that a Sturmian number is very well approximated by quadratic numbers, well enough to satisfy the hypotheses in the following of Schmidt~\cite{Schmidt_simult-approxim-algebraic-by-rational_1967}: for a real irrational $\xi^+ \in \R\setminus \Q$, if there is real $\epsilon>3$ and infinitely many quadratic irrationals $\xi^+_k$ such that $\lvert \xi^+-\xi^+_k\rvert <H(\xi_k)^{-3-\epsilon}$, then $\xi^+$ must be transcendental.
Let us mention another proof of transcendence using the fact that they begin with infinitely many palindromes (\cite[§2]{Adamczewski-Allouche_palindromes-contfrac_2007}).
The following~\cite[Theorem 2.2.1]{Adamczewski-Bugeaud_transcendence-measures-contfrac_2010} gives the state of the art on their Mahler measures.

\begin{theorem}[transcendence of Sturmian numbers]
    \label{thm:Sturm-transcendent}
    The numbers $\xi^\pm_{\sigma,\tau}$ are transcendent, and more precisely: if $s$ is unbounded then $w_2(\xi^\pm_{\sigma,\tau})=\infty$ (Mahler type $\mathrm{U}_2$), and if $s$ is bounded then $\exists c\in \R_{>0},\; \forall d\in \N_{\ge 1} \colon \: w_d(\xi^+)\le \exp\left(c (\log 3d)^3 (\log \log 3d)^2\right)$ (Mahler class $\mathrm{S}$ or $\mathrm{T}$).
\end{theorem}

\begin{comment}
% \begin{proof}
The transcedence of $\xi^\pm_{\sigma,\tau}$ was initially proved in~\cite[Proposition 3 and Theorem 7]{ADQZ_transcendence-sturmian_2001}: their strategy consists of showing that a Sturmian number is very well approximated by quadratic numbers, well enough to satisfy the hypotheses in the following of Schmidt~\cite{Schmidt_simult-approxim-algebraic-by-rational_1967}: for a real irrational $\xi^+ \in \R\setminus \Q$, if there is real $\epsilon>3$ and infinitely many quadratic irrationals $\xi^+_k$ such that $\lvert \xi^+-\xi^+_k\rvert <H(\xi_k)^{-3-\epsilon}$, then $\xi^+$ must be transcendental.
%
Let us mention another proof of transcendence using the fact that they begin with infinitely many palindromes (\cite[§2]{Adamczewski-Allouche_palindromes-contfrac_2007}).
%
The~\cite[Theorem 2.2.1]{Adamczewski-Bugeaud_transcendence-measures-contfrac_2010} gives the state of the art on their Mahler measures.%
%This motivated the generalization~\cite[Theorem 2.2.2]{Adamczewski-Bugeaud_transcendence-measures-contfrac_2010} which we recalled as Theorem \label{thm:morphic-contfrac}.
% \end{proof}
\end{comment}

\cite[Theorem 3.1]{Bugeaud-Laurent_Diophantine-exponents-Sturmian_2005} computes the Mahler measures for characteristic Sturmian numbers.

\begin{theorem}[transcendence measures of characteristic Sturmian numbers]
\label{thm:w2-characteristic-Sturmian}
For an irrational $\sigma\in \R_{>0}$ with $\sigma= \Ecf{s_0,s_1,\dots}$, denote $\lambda_\sigma = \limsup_k \Ecf{s_k; s_{k-1},\dots, s_1, s_0}$.

The characteristic Sturmian number $\xi_{\sigma, 0}^+$ has transcendence measures:
\begin{equation*}
    % \hat{w}_2^\prime (\xi_{\sigma,0}^+)=1+\tfrac{\lambda_\sigma}{1+\lambda_\sigma}
    \hat{w}_2 (\xi_{\sigma,0}^+)=1 +\tfrac{2}{\lambda_\sigma}
    \qquad
    %w_2^\prime (\xi_{\sigma,0}^+)=1
    w_2 (\xi_{\sigma,0}^+)=1+2\lambda_\sigma
\end{equation*}
\end{theorem}

\begin{remark}[Mahler spectra of all Sturmian]
    The Cassaigne spectrum $\sigma \mapsto \lambda_\sigma$ appears in various related settings and its metric topology still work in progress (see~\cite{Cassaigne_recurrence-quotient_1999,  Carminati-Tiozzo_bifurcation-locus-BAD_2022, Kaneko-Steiner_Markov-Lagrange-one-sided_2024}).
    Its understanding is equivalent to that of the Mahler $w_2, \hat{w}_2$ spectra of all Sturmian numbers.
\end{remark}

\subsubsection*{Lagrange constants of Markov quadratics and Sturmian transcendentals}

The following Theorem is due to~\cite{Haas_geometry-Markoff-forms_1987} building on work of Cohn~\cite{Cohn_Markoff-perforated-torus_1971}, where one can also learn the connection with the arithmetic of Markov quadratic forms.
A recent geometric proof in the spirit of our previous discussion can be found in~\cite{Springborn_hyperbolic-Markov-forms_2017}.

\begin{theorem}[Lagrange constants of Markov and Sturmian numbers]
\label{thm:Markov-Cohn-quadratic}
For $\xi^+\in \R\setminus \Q$:
\begin{itemize}[noitemsep]
    \item $\LC(\xi^+)>3 \iff \PGL_2(\Z)\cdot \xi^+$ contains a Markov quadratic
    \item $\LC(\xi^+)=3 \iff \PGL_2(\Z)\cdot \xi^+$ contains a Sturmian transcendental
\end{itemize}
In those cases, the orbit $\PSL_2(\Z)\cdot \xi^+$ contains a unique Markov or Sturmian number up to the action of $S$ and of the $\Shift$ on $\{R,L\}^\N$.
For a Markov or Sturmian number $\xi^+\in \R$, its conjugate is the unique $\xi^-\in \R\setminus \{\xi^+\}$ such that $\MC(\xi^-,\xi^+)=\LC(\xi^+)$, and the simple geodesic $\xi'\subset \M'$ has length given by $\MC(\xi^-,\xi^+)=3\coth\left(\tfrac{1}{2}\ell_{\M'}(\xi') \right)$.
\end{theorem}

\begin{remark}[worst approximable rationals]
   ~\cite{Springborn_worst-approx-ratioals_2024} proposes to define the Cohn constant of $r\in \Q$ as 
    \(\MC(\infty, r)=\sup \{q^2\lvert r-p/q \rvert \colon (p,q)\in (\Z\times \N^\ast),\, r\ne p/q\}\)
    and characterises those rationals with $\MC(r)>3$ as the set of Markov rationals $C_\sigma \cdot \infty$ and their companions $C_\sigma^n \cdot \{0,\infty)\}$ for $n\in \N_{\ge 2}$.
\end{remark}

%\bibliographystyle{alpha} %apalike
%\bibliography{biblio}

%% file: images/tikz/action-LR-RL-HP.tex
\begin{tikzpicture}[line cap=round,line join=round,>=triangle 45,x=4.166666666666667cm,y=4.166666666666667cm]
\clip(-1.3,-1.2) rectangle (1.3,1.2);

%triangle fill
\fill[line width=0.pt,color=marron,fill=green,fill opacity=0.2] (-0.5,0.) -- (0.5,0.) -- (0.,1.) -- cycle;
\fill[line width=0.pt,color=marron,fill=red,fill opacity=0.2] (-0.5,0.) -- (0.5,0.) -- (0.,-1.) -- cycle;
\fill[line width=0.pt,color=marron,fill=red,fill opacity=0.2] (1.,0.) -- (0.5,0.) -- (0.,1.) -- cycle;
\fill[line width=0.pt,color=marron,fill=red,fill opacity=0.2] (-1.,0.) -- (-0.5,0.) -- (0.,1.) -- cycle;
\fill[line width=0.pt,color=marron,fill=green,fill opacity=0.2] (1.,0.) -- (0.5,0.) -- (0.,-1.) -- cycle;
\fill[line width=0.pt,color=marron,fill=green,fill opacity=0.2] (-1.,0.) -- (-0.5,0.) -- (0.,-1.) -- cycle;

%cercle
\draw [line width=2.pt] (0.,0.) circle (4.166666666666667cm);

%line green
\draw [line width=2.pt,color=green] (1.,0.)-- (0.,1.);
\draw [line width=2.pt,color=green] (0.,1.)-- (-0.5,0.);
\draw [line width=2.pt,color=green] (-0.5,0.)-- (0.,-1.);
\draw [line width=2.pt,color=green] (0.,1.)-- (0.5,0.);
\draw [line width=2.pt,color=green] (0.5,0.)-- (0.,-1.);
\draw [line width=2.pt,color=green] (0.,1.)-- (-1.,0.);
\draw [line width=2.pt,color=green] (1.,0.)-- (0.,-1.);
\draw [line width=2.pt,color=green] (-0.5,0.)-- (-1.,0.);
\draw [line width=2.pt,color=green] (0.5,0.)-- (1.,0.);

%line marron (à retailler)
\draw [line width=2.pt,color=marron] (0.5,0.)-- (0.666,0.333);
\draw [line width=2.pt,color=marron] (0.5,0.)-- (0.666,-0.333);
\draw [line width=2.pt,color=marron] (-0.5,0.)-- (-0.666,-0.333);
\draw [line width=2.pt,color=marron] (-0.5,0.)-- (-0.666,0.333);
\draw [line width=2.pt,color=marron] (0.,0.)-- (0.5,0.);
\draw [line width=2.pt,color=marron] (-0.5,0.)-- (0.,0.);

%line forestgreen
\draw [line width=2.5pt,color=forestgreen] (0.,1.)-- (1.,0.);
\draw [line width=2.5pt,color=forestgreen] (0.,1.)-- (0.,-1.);
\draw [line width=2.5pt,color=forestgreen] (-1.,0.)-- (0.,1.);
\draw [line width=2.5pt,color=forestgreen] (-1.,0.)-- (0.,-1.);
\draw [line width=2.5pt,color=forestgreen] (0.,-1.)-- (1.,0.);

%arrows
\draw [line width=1.pt,color=black,-{Stealth[length=3.mm,width=2.5mm]}] (45:-0.65)-- (45:0.65);
\draw [line width=1.pt,color=black,-{Stealth[length=3.mm,width=2.5mm]}] (-45:-0.65)-- (-45:0.65);
\draw[color=black,anchor=south west] (45:0.7) node {\Large$RL$};
\draw[color=black,anchor=north west] (-45:0.7) node {\Large$LR$};

\begin{scriptsize}
%carrés
\draw [fill=marron] (0.666,0.333) ++(-4.pt,0 pt) -- ++(4.pt,4.pt)--++(4.pt,-4.pt)--++(-4.pt,-4.pt)--++(-4.pt,4.pt);
\draw [fill=marron] (0.666,-0.333) ++(-4.pt,0 pt) -- ++(4.pt,4.pt)--++(4.pt,-4.pt)--++(-4.pt,-4.pt)--++(-4.pt,4.pt);
\draw [fill=marron] (-0.666,-0.333) ++(-4.pt,0 pt) -- ++(4.pt,4.pt)--++(4.pt,-4.pt)--++(-4.pt,-4.pt)--++(-4.pt,4.pt);
\draw [fill=marron] (-0.666,0.333) ++(-4.pt,0 pt) -- ++(4.pt,4.pt)--++(4.pt,-4.pt)--++(-4.pt,-4.pt)--++(-4.pt,4.pt);
\draw [fill=marron,rotate=45] (0.,0.) ++(-4.pt,0 pt) -- ++(4.pt,4.pt)--++(4.pt,-4.pt)--++(-4.pt,-4.pt)--++(-4.pt,4.pt);

%triangles
\draw [fill=marron,shift={(0.5,0.)},rotate=270] (0,0) ++(0 pt,4.5pt) -- ++(3.8971143170299736pt,-6.75pt)--++(-7.794228634059947pt,0 pt) -- ++(3.8971143170299736pt,6.75pt);
\draw [fill=marron,shift={(-0.5,0.)},rotate=90] (0,0) ++(0 pt,4.5pt) -- ++(3.8971143170299736pt,-6.75pt)--++(-7.794228634059947pt,0 pt) -- ++(3.8971143170299736pt,6.75pt);

%points
%\draw [fill=black] (-0.8,0.6) circle (2.5pt);
%\draw[color=black] (-0.9,0.65) node {\Large$-\frac{2}{1}$};
\draw [fill=black] (0.,1.) circle (2.5pt);
\draw[color=black] (0.,1.12) node {\Large$1/0$};
\draw [fill=black] (1.,0.) circle (2.5pt);
\draw[color=black] (1.12,0.) node {\Large$1/1$};
\draw [fill=black] (0.,-1.) circle (2.5pt);
\draw[color=black] (0.,-1.15) node {\Large$0/1$};
\draw [fill=black] (-1.,0.) circle (2.5pt);
\draw[color=black] (-1.2,0.) node {\Large$-1/1$};
%\draw [fill=black] (0.8,0.6) circle (2.5pt);
%\draw[color=black] (0.89,0.62) node {\Large$\frac{2}{1}$};
%\draw [fill=black] (-0.8,-0.6) circle (2.5pt);
%\draw[color=black] (-0.91,-0.64) node {\Large$-\frac{1}{2}$};
%\draw [fill=black] (0.8,-0.6) circle (2.5pt);
%\draw[color=black] (0.87,-0.64) node {\Large$\frac{1}{2}$};

\end{scriptsize}
\end{tikzpicture}

%% file: images/tikz/hexagon-to-modular-monodromy.tex
\begin{tikzpicture}[line cap=round,line join=round,>=triangle 45,x=4.166666666666667cm,y=4.166666666666667cm]
\clip(-1.2,-1.2) rectangle (1.6,1.2);
%\draw[](-1.2,-1.2) rectangle (1.6,1.2);
%hexagone
\draw [line width=2.pt,color=black] plot[samples at={-150,-90,...,210},variable=\x] 
  (\x:1);

%triangle fill
\fill[line width=0.pt,color=marron,fill=red,fill opacity=0.2] (0.,0.) -- (-150:1.) -- (-90:1.) -- cycle;
\fill[line width=0.pt,color=marron,fill=green,fill opacity=0.2] (0.,0.) -- (-90:1.) -- (-30:1.) -- cycle;
\fill[line width=0.pt,color=marron,fill=red,fill opacity=0.2] (0.,0.) -- (-30:1.) -- (30:1.) -- cycle;
\fill[line width=0.pt,color=marron,fill=red,fill opacity=0.2] (0.,0.) -- (150:1.) -- (90:1.) -- cycle;
\fill[line width=0.pt,color=marron,fill=green,fill opacity=0.2] (0.,0.) -- (90:1.) -- (30:1.) -- cycle;
\fill[line width=0.pt,color=marron,fill=green,fill opacity=0.2] (0.,0.) -- (150:1.) -- (-150:1.) -- cycle;

%line green
%carrés at={-150,-90,..,150}
\draw [line width=2.pt,color=green] (0.,0.)-- (-150:1.);
\draw [line width=2.pt,color=green] (0.,0.)-- (-90:1.);
\draw [line width=2.pt,color=green] (0.,0.)-- (-30:1.);
\draw [line width=2.pt,color=green] (0.,0.)-- (30:1.);
\draw [line width=2.pt,color=green] (0.,0.)-- (90:1.);
\draw [line width=2.pt,color=green] (0.,0.)-- (150:1.);

%line forestgreen
%triangles at={-180,-120,...,180}
\draw [line width=2.pt,color=forestgreen] (0.,0.)-- (-180:0.8657);
\draw [line width=2.pt,color=forestgreen] (0.,0.)-- (-120:0.8657);
\draw [line width=2.pt,color=forestgreen] (0.,0.)-- (-60:0.8657);
\draw [line width=2.pt,color=forestgreen] (0.,0.)-- (0:0.8657);
\draw [line width=2.pt,color=forestgreen] (0.,0.)-- (60:0.8657);
\draw [line width=2.pt,color=forestgreen] (0.,0.)-- (120:0.8657);

%arrows
\draw [color=black,line width=1.pt, -{Stealth[length=3.mm,width=2.5mm]}] (0:1.1) to[out=90, in=-20,looseness=1] (60:1.);
\draw[color=black] (35:1.25) node {\Huge$ \Z/6$};

%arrow squid
\draw[color=black] (0:1.35) node {\Huge$\rightsquigarrow$};

\begin{scriptsize}
%centre
\draw [line width=1.5pt,color=black, fill=white] (0:0) circle (3.5pt);
%carrés
\draw [fill=marron, rotate around={45:(-180:0.8657)}] (-180:0.8657) ++(-4.pt,0 pt) -- ++(4.pt,4.pt)--++(4.pt,-4.pt)--++(-4.pt,-4.pt)--++(-4.pt,4.pt);
\draw [fill=marron, rotate around={10:(-120:0.8657)}] (-120:0.8657) ++(-4.pt,0 pt) -- ++(4.pt,4.pt)--++(4.pt,-4.pt)--++(-4.pt,-4.pt)--++(-4.pt,4.pt);
\draw [fill=marron, rotate around={-10:(-60:0.8657)}] (-60:0.8657) ++(-4.pt,0 pt) -- ++(4.pt,4.pt)--++(4.pt,-4.pt)--++(-4.pt,-4.pt)--++(-4.pt,4.pt);
\draw [fill=marron, rotate around={45:(0:0.8657)}] (0:0.8657) ++(-4.pt,0 pt) -- ++(4.pt,4.pt)--++(4.pt,-4.pt)--++(-4.pt,-4.pt)--++(-4.pt,4.pt);
\draw [fill=marron, rotate around={10:(60:0.8657)}] (60:0.8657) ++(-4.pt,0 pt) -- ++(4.pt,4.pt)--++(4.pt,-4.pt)--++(-4.pt,-4.pt)--++(-4.pt,4.pt);
\draw [fill=marron, rotate around={-10:(120:0.8657)}] (120:0.8657) ++(-4.pt,0 pt) -- ++(4.pt,4.pt)--++(4.pt,-4.pt)--++(-4.pt,-4.pt)--++(-4.pt,4.pt);

%triangles at={-150,-90,...,150}
\draw [fill=marron,shift={(-150:1)},rotate=180] (0,0) ++(0 pt,4.5pt) -- ++(3.897pt,-6.75pt)--++(-7.794pt,0 pt) -- ++(3.897pt,6.75pt);
\draw [fill=marron,shift={(-90:1)},rotate=0] (0,0) ++(0 pt,4.5pt) -- ++(3.897pt,-6.75pt)--++(-7.794pt,0 pt) -- ++(3.897pt,6.75pt);
\draw [fill=marron,shift={(-30:1)},rotate=180] (0,0) ++(0 pt,4.5pt) -- ++(3.897pt,-6.75pt)--++(-7.794pt,0 pt) -- ++(3.897pt,6.75pt);
\draw [fill=marron,shift={(30:1)},rotate=0] (0,0) ++(0 pt,4.5pt) -- ++(3.897pt,-6.75pt)--++(-7.794pt,0 pt) -- ++(3.897pt,6.75pt);
\draw [fill=marron,shift={(90:1)},rotate=180] (0,0) ++(0 pt,4.5pt) -- ++(3.897pt,-6.75pt)--++(-7.794pt,0 pt) -- ++(3.897pt,6.75pt);
\draw [fill=marron,shift={(150:1)},rotate=0] (0,0) ++(0 pt,4.5pt) -- ++(3.897pt,-6.75pt)--++(-7.794pt,0 pt) -- ++(3.897pt,6.75pt);

\end{scriptsize}
\end{tikzpicture}
\begin{tikzpicture}[line cap=round,line join=round,>=triangle 45,
    rotate=270,
    x=6cm,y=6cm]
\clip(-0.2,-0.2) rectangle (1.3,0.6);
%\draw[](-0.2,-0.2) rectangle (1.3,0.6);
%hexagone
\draw [line width=2.pt,color=black] (0:0.8657) -- (30:1.);

%line green
\draw [line width=2.pt,color=green] (30.:0.05)-- (30:1.);

%line forestgreen
\draw [line width=2.pt,color=forestgreen] (0.:0.05)-- (0:0.8657);

\begin{scriptsize}
%centre
\draw [line width=1.5pt,color=black, fill=white] (0.05,0.015) ellipse (3pt and 4.5pt);
%\draw [line width=1.5pt,color=black, fill=white] (0:0) circle (3.5pt);
%carrés
\draw [fill=marron, rotate around={45:(0:0.8657)}] (0:0.8657) ++(-4.pt,0 pt) -- ++(4.pt,4.pt)--++(4.pt,-4.pt)--++(-4.pt,-4.pt)--++(-4.pt,4.pt);

%triangle
\draw [fill=marron,shift={(30:1)},rotate=0] (0,0) ++(0 pt,4.5pt) -- ++(3.897pt,-6.75pt)--++(-7.794pt,0 pt) -- ++(3.897pt,6.75pt);

\end{scriptsize}
\end{tikzpicture}